\documentclass[12pt]{amsart}
\usepackage[utf8]{inputenc}
\usepackage{amsmath}
\usepackage{amsfonts}
\usepackage{amssymb}
\usepackage{amsthm}
\usepackage{bbm}
\usepackage{hyperref}

\theoremstyle{plain}
\newtheorem{theorem}{Theorem}
\newtheorem{lemma}[theorem]{Lemma}

\newtheorem{proposition}[theorem]{Proposition}
\theoremstyle{definition}

\theoremstyle{remark}

\newtheorem*{rmk}{Remark}

\numberwithin{equation}{section}
\numberwithin{theorem}{section}

\author{Kyle Pratt}
\address{Department of Mathematics \\
 The University of Illinois at Urbana-Champaign \\
 1409 West Green Street, Urbana, IL 61801}
\email{{\href{mailto:kpratt4@illinois.edu}{kpratt4@illinois.edu}},{\href{mailto:kvpratt@gmail.com}{kvpratt@gmail.com}}}


\title{Primes from sums of two squares and missing digits}

\begin{document}
\date{}

\maketitle

\begin{abstract}
Let $\mathcal{A}'$ be the set of integers missing any three fixed digits from their decimal expansion. We produce primes in a thin sequence by proving an asymptotic formula for counting primes of the form $p = m^2 + \ell^2$, with $\ell \in \mathcal{A}'$. 

The proof draws on ideas from the work of Friedlander-Iwaniec on primes of the form $p = x^2+y^4$, as well as ideas from the work of Maynard on primes with restricted digits.
\end{abstract}

\tableofcontents

\section{Introduction}\label{sec: introduction}

Some of the most interesting questions in analytic prime number theory arise from interactions with themes and ideas from other areas of mathematics. The famous twin prime conjecture, for example, arises from placing the multiplicative notion of a prime number in an additive context.

An early instance of this phenomenon is in Fermat's 1640 ``Christmas letter'' to Marin Mersenne \cite[pp. 212-217]{TanHen}, wherein he describes which numbers may be written as a sum of two integral squares (Fermat phrased his observations in terms of integers appearing as hypotenuses of right triangles). Along the way he noted that every prime $p \equiv 1 \pmod{4}$ may be written as $p = x^2+y^2$, but in true Fermat fashion he supplied no proof\footnote{Euler finally found a proof more than a century later \cite{Eul}.}.

At first glance the equation $p = x^2+y^2$ looks like an additive equation involving primes, but with the benefit of substantial hindsight we see this is in fact a multiplicative problem, for $x^2+y^2$ is the norm form of the algebraic number field $\mathbb{Q}(i)$.

Other famous problems in prime number theory concern primes in ``thin'' sequences, such as primes in short intervals, or primes of the form $p = n^2+1$. A set of integers $\mathcal{S}\subset[1,x]$ is thin if there are few elements of $\mathcal{S}$ relative to $x$ (think $|\mathcal{S}| \leq x^{1-\epsilon}$ for some $\epsilon > 0$). It is natural to ask under what conditions $\mathcal{S}$ contains prime numbers, but often these questions are very hard. Most often one needs the set $\mathcal{S}$ to have some nice multiplicative structure to exploit.

Several authors have proved the existence of infinitely many primes within different thin sequences. Fouvry and Iwaniec \cite{foiw} proved there are infinitely many primes of the form $p = m^2+q^2$, where $q$ is a prime number. The set $\{m^2+q^2 \leq x : q \text{ prime}\}$ has size $\approx x(\log x)^{-2}$, and so is thin in the sense that it has zero density inside of the primes. This is a nice example of additively-structured primes in a thin sequence.

Friedlander and Iwaniec \cite{friediw2} built on the foundation laid by Fouvry and Iwaniec, and proved there are infinitely many primes of the form $p = x^2+y^4$. This is a much thinner sequence of primes than those considered by Fouvry and Iwaniec, and consequently the proof is much more difficult. It is crucial for the work of Friedlander and Iwaniec that $x^2+y^4 = x^2 + (y^2)^2$.

Other striking examples are the works of Heath-Brown \cite{hb} on primes of the form $p =x^3+2y^3$ and Heath-Brown and Moroz \cite{hbm} on primes represented by cubic forms, and Maynard \cite{maynard2} on primes represented by incomplete norm forms. Heath-Brown and Li \cite{hbli} refined the theorem of Friedlander and Iwaniec by showing there are infinitely many primes of the form $p = x^2+q^4$, where $q$ is a prime. Each of these results relies heavily on the fact that the underlying polynomial is related to the norm form of an algebraic number field.

Polynomials offer one source of thin sequences, but they are not the only source. Particularly attractive are other, more exotic, thin sequences, like the set of integers missing a fixed digit from their decimal expansion. To be precise, let $a_0 \in \{0,1,2,\ldots,9\}$ be fixed, and let $\mathcal{A}$ be the set of nonnegative integers without the digit $a_0$ in their decimal expansion. We write $\mathbf{1}_\mathcal{A}$ for the indicator function of this set. We define 
\begin{align*}
\gamma_0 = \frac{\log 9}{\log 10} = 0.954\ldots,
\end{align*}
and note that
\begin{align}\label{eq: number of elts of cal A}
\sum_{\ell \leq y} \mathbf{1}_\mathcal{A}(\ell) \asymp y^{\gamma_0}, \ \ \ y \geq 2.
\end{align}

Our goal is to tie together several different mathematical strands by proving there are infinitely many primes $p$ of the form $p = m^2 + \ell^2$, where $\ell \in \mathcal{A}$. Note that $\frac{1}{2} + \frac{\gamma_0}{2} < 1$, so this sequence of primes is indeed thin.

The present work was inspired by Maynard's beautiful paper \cite{maynard}, wherein he showed there are infinitely many primes in the thin sequence $\mathcal{A}$. The key to the whole enterprise is that the Fourier transform of $\mathcal{A}$ has remarkable properties. Exploiting this Fourier structure has been vital in works on digit-related problems (see, for example, \cite{bcs,bs,bour,col,coq1,coq2,dm1,dm2,drm,ems1,ems2,kon} and the works cited therein). We also rely on this Fourier structure.

It turns out that we ultimately use few of the tools Maynard developed. Rather, our work is closer in spirit to the work of Fouvry and Iwaniec \cite{foiw} and the work of Friedlander and Iwaniec \cite{friediw2}.

Our basic strategy is to use a sieve to count the primes $p = m^2+\ell^2, \ell \in \mathcal{A}$. It is tempting to try and estimate the sum
\begin{align}\label{eq: naive sum}
\mathop{\sum \sum}_{\substack{m^2+\ell^2 \leq x \\ (\ell,10)=1}} \mathbf{1}_\mathcal{A}(\ell) \Lambda(m^2+\ell^2),
\end{align}
but for technical reasons it is convenient to prove a stronger theorem in which $\ell$ has no small prime factors. This ensures that $\ell$ is almost always coprime to other variables.

\begin{theorem}\label{main theorem}
Let $x$ be large, and let $A > 0$ be fixed. Let $P$ be a parameter which satisfies
\begin{align*}
(\log \log x)^4 \leq \log P \leq \frac{\sqrt{\log x}}{\log \log x},
\end{align*}
and define
\begin{align*}
\Pi = \prod_{p \leq P} p.
\end{align*}
We then have
\begin{align*}
\mathop{\sum \sum}_{\substack{m^2 + \ell^2\leq x \\ (\ell,\Pi)=1}} \mathbf{1}_\mathcal{A}(\ell) \Lambda(m^2+\ell^2) &= \frac{4C \kappa_1}{\pi} \frac{e^{-\gamma}}{\log P} \mathop{\sum \sum}_{m^2 + \ell^2 \leq x} \mathbf{1}_\mathcal{A}(\ell) + O \left(x^{\frac{1}{2} + \frac{\gamma_0}{2}} (\log x)^{-A} \right),
\end{align*}
where $\gamma$ denotes Euler's constant,
\begin{align*}
C &= \prod_p \left(1 - \frac{\chi(p)}{(p-1) (p- \chi(p))} \right),
\end{align*}
$\chi$ is the nonprincipal character modulo 4, and
\begin{align*}
\kappa_1 &=
\begin{cases}
\frac{10}{9}, \ \ \ \ \ \ \ \  &(a_0,10) \neq 1, \\
\frac{10 (\varphi(10)-1)}{9\varphi(10)}, \ \ \  &(a_0,10) = 1.
\end{cases}
\end{align*}
The implied constant depends on $A$ and is ineffective.
\end{theorem}

\begin{rmk}
It is potential exceptional zeros for certain Hecke $L$-functions that make the implied constant in Theorem \ref{main theorem} ineffective.
\end{rmk}

We are able to avoid more sophisticated sievs like Harman's sieve \cite{harman}, and instead we require only Vaughan's identity (see \eqref{vaughans identity}). The application of Vaughan's identity reduces the problem to the estimation of ``Type I'' and ``Type II'' sums. The Type I information, which is quite strong, comes from a general result of Fouvry and Iwaniec (see Lemma \ref{Type I remainder}). The strength of the Type I bound relies on the homogeneous nature of the polynomial $x^2+y^2$. For the Type II sums we follow the outlines of the argument of Friedlander-Iwaniec. Our argument is less complicated in some places and more complicated in others. The desired cancellation eventually comes from an excursion into a zero-free region for Hecke $L$-functions. 

We obtain Type II information in a wide interval, much wider than that which is required given our amount of Type I information. This suggests the possibility of finding primes of the form $p = m^2 + \ell^2$, where $\ell$ is missing more than one digit in its decimal expansion.

\begin{theorem}\label{second theorem}
Let $\mathcal{B} \subset \{0,1,\ldots,9\}$ satisfy $1 \leq |\mathcal{B}| \leq 3$, and let $\mathcal{A}'$ denote the set of nonnegative integers whose decimal expansions consist only of the digits in $\{0,1,\ldots,9\} \backslash \mathcal{B}$. Let 
\begin{align*}
\gamma_\mathcal{B} = \frac{\log (10 - |\mathcal{B}|)}{\log 10}.
\end{align*}
Then, with the notation as above, we have
\begin{align*}
\mathop{\sum \sum}_{\substack{m^2 + \ell^2\leq x \\ (\ell,\Pi)=1}} \mathbf{1}_{\mathcal{A}'}(\ell) \Lambda(m^2+\ell^2) &= \frac{4C \kappa_\mathcal{B}}{\pi} \frac{e^{-\gamma}}{\log P} \mathop{\sum \sum}_{m^2 + \ell^2 \leq x} \mathbf{1}_{\mathcal{A}'}(\ell) + O \left(x^{\frac{1}{2} + \frac{\gamma_\mathcal{B}}{2}} (\log x)^{-A} \right),
\end{align*}
where
\begin{align*}
\kappa_\mathcal{B} = \frac{10}{\varphi(10)} \frac{\varphi(10) +|\{a \in \mathcal{B} : (a,10) \neq 1\}|-|\mathcal{B}|}{10-|\mathcal{B}|}.
\end{align*}
The implied constant depends on $A$ and is ineffective.
\end{theorem}

\begin{rmk}
When $|\mathcal{B}| = 3$, Theorem \ref{second theorem} shows the existence of primes in a set of integers of size $\ll x^{\frac{1}{2}+\frac{1}{2}\frac{\log 7}{\log 10}} \approx x^{0.9225}$. One may take $|\mathcal{B}|$ to be larger by using a more complicated sieve argument and imposing extra conditions on the elements of $\mathcal{B}$, but we do not pursue this here.
\end{rmk}

Throughout the paper we make use of asymptotic notation $\ll, \gg, O(\cdot)$, and $o(\cdot)$. We write $f \asymp g$ if $f \ll g$ and $f \gg g$. Usually the implied constants are absolute, but from section \ref{sec: first bilinear section} onward we allow the implied constants to depend on $L$ (see \eqref{defn of theta}) without indicating this in the notation. A subscript such as $f \ll_\epsilon g$ means the implied constant depends on $\epsilon$.

We use the convention that $\epsilon$ denotes an arbitrarily small positive quantity that may vary from one occurrence to the next. Thus, we may write $x^{\epsilon + o(1)} \leq x^\epsilon$, for example, with no difficulties.

In order to economize on space, we often write the congruence $n \equiv v \pmod{d}$ as $n \equiv v (d)$. The notation $n \mid m^\infty$ means there is some positive integer $N$ such that $n$ divides $m^N$. We use the symbol $\star$ to denote Dirichlet convolution.

We write $\varphi$ for the Euler totient function, and $P^+(n),P^-(n)$ for the largest and smallest prime factors of $n$, respectively.

\section{Initial manipulations and outline}\label{sec: outline}

We begin the proof of Theorem \ref{main theorem} by setting out to estimate
\begin{align*}
S(x) := \sum_{n \leq x} a(n) \Lambda(n),
\end{align*}
where
\begin{align*}
a(n) := \mathop{\sum \sum}_{\substack{m^2 + \ell^2 = n \\ (\ell,\Pi)=1}} \mathbf{1}_\mathcal{A}(\ell).
\end{align*}
In the definition of $a(n)$ we allow $m$ to range over both positive and negative integers. 

Let $U,V > 2$ be real parameters to be chosen later (see \eqref{sample choices of U,V}). For an arithmetic function $f: \mathbb{N}\rightarrow \mathbb{C}$ and $W \geq 1$, define
\begin{align*}
f_{\leq W} (n)  := 
\begin{cases}
f(n), \ &n \leq W, \\
0, \ &n > W,
\end{cases}
\end{align*}
and write $f_{>W} = f - f_{\leq W}$. Then Vaughan's identity is
\begin{align}\label{vaughans identity}
\Lambda &= \Lambda_{\leq U} + \mu_{\leq V} \star \log - \Lambda_{\leq U} \star \mu_{\leq V} \star 1 + \Lambda_{> U} \star \mu_{>V} \star 1.
\end{align}
The different pieces of Vaughan's identity decompose $S(x)$ into several sums, which we handle with different techniques. The first term $\Lambda_{\leq U}$ we treat trivially, since we may choose $U$ to be small compared to $x$. The terms $\mu_{\leq V} \star \log$ and $\Lambda_{\leq U} \star \mu_{\leq V} \star 1$ are Type I sums, and require estimation of the congruence sums
\begin{align*}
A_d(x) &:= \sum_{\substack{n \leq x \\ n \equiv 0 (d)}} a(n), \\
A_d'(x) &:= \sum_{\substack{n \leq x \\ n \equiv 0 (d)}} a(n) \log n.
\end{align*}
The last term $\Lambda_{> U} \star \mu_{>V} \star 1$ gives rise to a Type II or ``bilinear'' sum, and the estimation of this sum requires much more effort than estimating the Type I sums.

Let us carry out this decomposition explicitly. Inserting \eqref{vaughans identity} into the definition of $S(x)$ gives
\begin{align}\label{eq: decompose S by Vaughan}
S(x) &= \sum_{n \leq x} a(n) \Lambda(n) = \sum_{n \leq U} a(n) \Lambda(n) + \sum_{n \leq x} a(n) (\mu_{\leq V}\star\log)(n) \\ 
&-\sum_{n \leq x} a(n) (\mu_{\leq V} \star \Lambda_{\leq U}\star 1)(n) + \sum_{n \leq x} a(n) (\mu_{> V} \star \Lambda_{> U}\star 1)(n). \nonumber
\end{align}
By trivial estimation
\begin{align*}
\sum_{n \leq U} a(n) \Lambda(n) &\leq (\log U)\sum_{n \leq U} a(n) = (\log U) \mathop{\sum \sum}_{\substack{m^2 + \ell^2 \leq U \\ (\ell,\Pi)=1}} \mathbf{1}_\mathcal{A}(\ell) \\
&\leq (\log U) \left(\sum_{|m| \leq U^{1/2}} 1 \right) \left(\sum_{\ell \leq U^{1/2}} \mathbf{1}_\mathcal{A}(\ell) \right) \ll (\log U) U^{\frac{1}{2} + \frac{\gamma_0}{2}},
\end{align*}
the last inequality following by \eqref{eq: number of elts of cal A}. In what follows we shall have many occasions to use the bound
\begin{align*}
\sum_{n \leq z} \left(\mathop{\sum \sum}_{m^2 + \ell^2 = n} \mathbf{1}_\mathcal{A}(\ell) \right) \ll z^{\frac{1}{2} + \frac{\gamma_0}{2}},
\end{align*}
and we do so without further comment.

For the second sum in \eqref{eq: decompose S by Vaughan} we interchange the order of summation and separate the logarithmic factors to obtain
\begin{align*}
\sum_{n \leq x} a(n) (\mu_{\leq V}\star\log)(n) &= \sum_{d \leq V} \mu(d) \sum_{n \leq x} a(n) \log (n/d) \\ 
&= \sum_{d \leq V} \mu(d) A_d'(x) - \sum_{d \leq V} \mu(d) (\log d) A_d(x).
\end{align*}
We similarly show that the third sum is
\begin{align*}
-\sum_{n \leq x} a(n) (\mu_{\leq V} \star \Lambda_{\leq U}\star 1)(n) &= -\sum_{d \leq V} \sum_{m \leq U} \mu(d) \Lambda(m) A_{dm}(x).
\end{align*}

For the last sum in \eqref{eq: decompose S by Vaughan}, the Type II sum, we interchange the order of summation and change variables to obtain
\begin{align*}
\sum_{n \leq x}a(n) (\mu_{>V} \star \Lambda_{>U} \star 1)(n) &= \mathop{\sum \sum}_{\substack{mn \leq x \\ n > V \\ m > U}} \mu(n) (\Lambda_{>U} \star 1)(m) a(mn) \\
&= \sum_{U < m \leq x/V} (\Lambda_{>U} \star 1)(m) \sum_{V < n \leq x/m} \mu(n) a(mn).
\end{align*}

In short,
\begin{align}\label{vaughan decomposition}
S(x) &= A(x;U,V) + B(x;U,V) + O((\log U) U^{\frac{1}{2} + \frac{\gamma_0}{2}}),
\end{align}
where
\begin{align}\label{vaughan type I}
A(x;U,V) &:= \sum_{d \leq V} \mu(d) \left(A_d'(x) - A_d(x) \log d - \sum_{m \leq U} \Lambda(m) A_{dm}(x) \right)
\end{align}
and
\begin{align}\label{vaughan Type II}
B(x;U,V) &:= \sum_{U < m \leq x/V} (\Lambda_{>U} \star 1)(m) \sum_{V < n \leq x/m} \mu(n) a(mn).
\end{align}
We can exchange $A_d'(x)$ in $A(x;U,V)$ for quantities involving $A_d(t)$ using partial summation:
\begin{align}\label{type I vaughan part sum}
A_d'(x) &= A_d(x) \log x - \int_1^x A_d(t) \frac{dt}{t}.
\end{align}

Define
\begin{align*}
M_d(x) &:= \frac{1}{d} \sum_{n \leq x}  a_d(n),
\end{align*}
where
\begin{align*}
a_d(n) &:= \mathop{\sum \sum}_{\substack{m^2 + \ell^2 = n \\ (\ell,\Pi)=1}} \mathbf{1}_\mathcal{A}(\ell) \rho_\ell(d)
\end{align*}
and $\rho_\ell(d)$ denotes the number of solutions $\nu$ to $\nu^2 + \ell^2 \equiv 0 \pmod{d}$. We expect that $M_d(x)$ is a good approximation to $A_d(x)$, at least on average. We therefore define the remainder terms
\begin{align}\label{eq: defn of remainders}
R_d(x) &:= A_d(x) - M_d(x), \ \ \ \ \ \ \ \ R(x,D) := \sum_{d \leq D} |R_d(x)|.
\end{align}
Inserting \eqref{type I vaughan part sum} into \eqref{vaughan type I} and writing $A_d(x) = M_d(x) + R_d(x)$, we obtain
\begin{align}\label{type i approx main and error}
A(x;U,V) = M(x;U,V) + R(x;U,V),
\end{align}
where
\begin{align}\label{eq: defn of MxUV}
M(x;U,V) &= \sum_{n \leq x} \sum_{d \leq V} \frac{\mu(d)}{d} \left(a_d(n) \log(n/d) - \sum_{m \leq U} \frac{\Lambda(m)}{m}a_{dm}(n) \right)
\end{align}
and
\begin{align}\label{eq: defn of RxUV}
R(x;U,V) &= \sum_{d \leq V} \mu(d) \left(R_d(x) \log(x/d) - \int_1^x R_d(t) \frac{dt}{t} - \sum_{m \leq U} \Lambda(m)R_{md}(x) \right).
\end{align}
This completes our preliminary manipulations of $S(x)$.

The outline of the rest of the paper is as follows. In section \ref{sec: sieve remainder} we show that $R(x;U,V)$ contributes only to the error term in Theorem \ref{main theorem}. The analysis in section \ref{sec: sieve main term, before fund} gives a partial analysis of $M(x;U,V)$, showing that, up to the condition $(\ell,\Pi)=1$, the term $M(x;U,V)$ yields the main term of Theorem \ref{main theorem}. We use the fundamental lemma of sieve theory to remove this condition in section \ref{sec: sieve main, using fund}, and this yields the desired main term. 

We estimate the bilinear form $B(x;U,V)$ in sections \ref{sec: first bilinear section} through \ref{sec: endgame}. In section \ref{sec: first bilinear section} we perform some technical reductions like separating variables. These reductions allow us to enter the Gaussian domain $\mathbb{Z}[i]$ in section \ref{sec: bilinear form transformations}. A congruence modulo $\Delta$ arises, and this introduces further complications. We address many of these in section \ref{sec: congruence exercises}. A particularly delicate issue is that $\mathcal{A}$ is not well-distributed in arithmetic progressions modulo $\Delta$ when $\Delta$ shares a factor with 10. At the end of section \ref{sec: polar boxes} we are mostly able to remove the congruence modulo $\Delta$, which simplifies our working considerably. With the congruence removed we devote section \ref{sec: endgame} to extracting cancellation from the sign changes of the M\"obius function using the theory of Hecke $L$-functions. Theorem \ref{main theorem} follows from \eqref{bound for sieve remainder}, \eqref{mxuv mt, removed f ell, before fund}, \eqref{mt after fund lemma}, \eqref{intermed summary}, and Proposition \ref{main proposition}.

In the last section, section \ref{sec: mods for other theorem}, we show how to modify the proof of Theorem \ref{main theorem} to prove Theorem \ref{second theorem}.

\section{The sieve remainder term}\label{sec: sieve remainder}

Our goal in this section is to show that
\begin{align}\label{bound for sieve remainder}
R(x;U,V) \ll x^{\frac{1}{2} + \frac{\gamma_0}{2} - \epsilon},
\end{align}
provided $U,V > 2$ and $UV \leq x^{\gamma_0 - \epsilon}$.

Applying the triangle inequality to \eqref{eq: defn of RxUV}, we get
\begin{align}\label{triangle bound on vaughan type I error}
|R(x;U,V)| &\ll (\log x) R(x,UV) + \int_1^x R(t,V) \frac{dt}{t}.
\end{align}
The following is the key result we use to estimate remainder terms.
\begin{lemma}\label{Type I remainder}
For $1 \leq D \leq x$ and $\epsilon > 0$ we have
\begin{align*}
R(x,D) &=\sum_{d \leq D} |R_d(x)| \ll D^{\frac{1}{4}} x^{\frac{1}{2} + \frac{\gamma_0}{4} + \epsilon},
\end{align*}
the implied constant depending only on $\epsilon$.
\end{lemma}
\begin{proof}
This is a specialization of \cite[Lemma 4]{foiw}. In the notation of \cite{foiw} we take $\lambda_{\ell} = \mathbf{1}_\mathcal{A}(\ell)$ for $\ell \leq x^{1/2}$. We then observe that
\begin{align*}
\| \lambda\| &\leq \left(\sum_{\ell \leq x^{1/2}} \mathbf{1}_\mathcal{A}(\ell) \right)^{1/2} \ll x^{\frac{\gamma_0}{4}},
\end{align*}
the last inequality following by \eqref{eq: number of elts of cal A}.
\end{proof}

With Lemma \ref{Type I remainder} in hand we can show the contribution from \eqref{triangle bound on vaughan type I error} is sufficiently small. The contribution from $R(x,UV)$ is negligible provided
\begin{align}\label{upper bound on UV}
UV \leq x^{\gamma_0 - \delta},
\end{align}
where $\delta > 0$ is any small fixed quantity. We henceforth assume \eqref{upper bound on UV}. We can also immediately estimate the part of the integral with $t \geq V$:
\begin{align}\label{eq: remainder large t}
\int_V^x R(t,V) \frac{dt}{t} \ll \int_V^x V^{\frac{1}{4}} t^{\frac{1}{2} + \frac{\gamma_0}{4} + \epsilon} \frac{dt}{t} \ll V^{\frac{1}{4}} x^{\frac{1}{2} + \frac{\gamma_0}{4} + \epsilon}.
\end{align}
This is sufficiently small provided $V \leq x^{\gamma_0 - \delta}$, which already follows from \eqref{upper bound on UV} since $U > 2$. To show \eqref{bound for sieve remainder} it therefore suffices to prove
\begin{align}\label{eq: remainder small t}
\int_1^V R(t,V) \frac{dt}{t} \ll V^{\frac{1}{2} + \frac{\gamma_0}{2} + \epsilon}.
\end{align}
We write
\begin{align*}
R(t,V) &= \sum_{d \leq V} |R_d(t)| \leq \sum_{d \leq V}\left(A_d(t) + M_d(t) \right)
\end{align*}
and estimate the sums involving $A_d$ and $M_d$ separately. 

For the term involving $A_d$ we use the divisor bound to obtain
\begin{align}\label{sum of Ad}
\sum_{d \leq V} A_d(t) &\leq \sum_{d \leq V} \mathop{\sum \sum}_{\substack{m^2 + \ell^2 \leq t \\ m^2 + \ell^2 \equiv 0 (d)}} \mathbf{1}_\mathcal{A}(\ell) \leq \mathop{\sum \sum}_{\substack{m^2 + \ell^2 \leq t}} \mathbf{1}_\mathcal{A}(\ell) \tau(m^2 + \ell^2) \\ \nonumber
&\ll t^\epsilon \mathop{\sum \sum}_{\substack{m^2 + \ell^2 \leq t}} \mathbf{1}_\mathcal{A}(\ell) \ll t^{\frac{1}{2} + \frac{\gamma_0}{2} + \epsilon}. \nonumber
\end{align}

The estimation of the term involving $M_d$ is slightly more complicated due to the presence of the function $\rho_\ell(d)$. Recall that $\rho_\ell(d)$ counts the number of residue classes $\nu \pmod{d}$ such that $\nu^2 + \ell^2 \equiv 0 \pmod{d}$. If $\ell$ is coprime to $d$, then we can divide both sides of the congruence by $\ell^2$ and we find that $\rho_\ell(d) = \rho(d)$, where $\rho(d)$ counts the number of solutions to $\nu^2 +1 \equiv 0 \pmod{d}$. In general, a slightly more complicated relationship holds.
\begin{lemma}\label{lemma about rho ell d}
Let $\ell,d$ be positive integers. Let $r(d)$ denote the largest integer $r$ such that $r^2 \mid d$. Then
\begin{align*}
\rho_\ell(d) = (r(d),\ell) \rho (d/(d,\ell^2)).
\end{align*}
\end{lemma}
\begin{proof}
See \cite[(3.4)]{foiw}.
\end{proof}
Observe that Lemma \ref{lemma about rho ell d} implies 
\begin{align*}
\rho_\ell(d) \leq \rho(d) \leq \tau(d)
\end{align*}
whenever $d$ is squarefree or coprime to $\ell$. If $p$ divides $\ell$, then 
\begin{align*}
\rho_\ell(p^e) \leq 2 p^{e/2}.
\end{align*}
The following lemma illustrates how we estimate sums involving $\rho_\ell$.
\begin{lemma}\label{lem: sums of rho ell}
Let $y \geq 2$, and let $\ell$ be an integer. Then
\begin{align*}
\sum_{n \leq y} \frac{\rho_\ell(n)}{n} \ll (\log y)^2 \prod_{p \mid \ell} \left(1 + \frac{7}{p^{1/2}} \right).
\end{align*}
\end{lemma}
\begin{proof}
We factor $n$ as $n = em$, where $e \mid \ell^\infty$ and $m$ is coprime to $\ell$. By multiplicativity and Lemma \ref{lemma about rho ell d} we obtain
\begin{align*}
\sum_{n \leq y} \frac{\rho_\ell(n)}{n} &\leq \sum_{e \mid \ell^\infty} \frac{\rho_\ell(e)}{e} \sum_{\substack{m \leq y \\ (m,\ell)=1}} \frac{\rho_\ell(m)}{m} \leq \sum_{e \mid \ell^\infty} \frac{\rho_\ell(e)}{e} \sum_{\substack{m \leq y \\ (m,\ell)=1}} \frac{\tau(m)}{m} \ll (\log y)^2 \sum_{e \mid \ell^\infty} \frac{\rho_\ell(e)}{e}.
\end{align*}
We use multiplicativity and Lemma \ref{lemma about rho ell d} again to obtain
\begin{align*}
\sum_{e \mid \ell^\infty} \frac{\rho_\ell(e)}{e} &= \prod_{p \mid \ell} \left(\sum_{j=0}^\infty \frac{\rho_\ell(p^j)}{p^j} \right) \leq \prod_{p \mid \ell} \left(1 + 2\sum_{j=1}^\infty \frac{1}{p^{j/2}} \right) = \prod_{p \mid \ell} \left(1 + \frac{2}{p^{1/2}-1}\right) \\
&\leq \prod_{p \mid \ell} \left(1 + \frac{7}{p^{1/2}}\right).
\end{align*}
\end{proof}

By the definition of $M_d(t)$ we find
\begin{align*}
\sum_{d \leq V} M_d(t) &\leq \mathop{\sum \sum}_{m^2 + \ell^2 \leq t} \mathbf{1}_\mathcal{A}(\ell) \sum_{d \leq V} \frac{\rho_\ell(d)}{d}.
\end{align*}
We apply Lemma \ref{lem: sums of rho ell} and obtain
\begin{align}\label{sum of Md}
\sum_{d \leq V} M_d(t) &\ll \mathop{\sum \sum}_{m^2 + \ell^2 \leq t} \mathbf{1}_\mathcal{A}(\ell) (\log V)^2 \tau(\ell) \ll t^{\frac{1}{2} + \frac{\gamma_0}{2}} (tV)^\epsilon,
\end{align}
and combining this with our bound \eqref{sum of Ad} yields \eqref{eq: remainder small t}.

\section{The sieve main term}\label{sec: sieve main term, before fund}
In this section we begin to show how $M(x;U,V)$ yields the main term for Theorem \ref{main theorem}: we show that $M(x;U,V)$ is equal to
\begin{align*}
\frac{4}{\pi} C \mathop{\sum \sum}_{\substack{m^2 + \ell^2 \leq x \\ (\ell,\Pi) = 1}} \mathbf{1}_\mathcal{A}(\ell),
\end{align*}
up to negligible error. The estimates involved are standard, but we give details for completeness.

From \eqref{eq: defn of MxUV} we derive
\begin{align}\label{eq: main term MxUV expanded}
M(x;U,V) = \mathop{\sum \sum}_{\substack{g^2 + \ell^2 \leq x \\ (\ell,\Pi) = 1}} \mathbf{1}_\mathcal{A}(\ell)\Bigg(&\log(g^2 + \ell^2) \sum_{d \leq V} \frac{\mu(d) \rho_\ell(d)}{d} - \sum_{d \leq V} \frac{\mu(d) \rho_{\ell}(d) \log d}{d} \\ 
&- \sum_{m \leq U} \frac{\Lambda(m)}{m} \sum_{d \leq V} \frac{\mu(d) \rho_\ell(dm)}{d}\Bigg).\nonumber
\end{align}
The main term arises from the second term on the right side of \eqref{eq: main term MxUV expanded}, and the other two terms contribute only to the error.

We begin by estimating
\begin{align*}
\sum_{d \leq V} \frac{\mu(d) \rho_\ell(d)}{d}
\end{align*}
uniformly in $\ell$. We note that
\begin{align*}
\rho_\ell(p) &=
\begin{cases}
1 + \chi(p), \ &p \nmid \ell, \\
1, \ \ \ \ \ \ \ \ \ \ &p \mid \ell,.
\end{cases}
\end{align*}
(Recall that $\chi$ is the nonprincipal character modulo 4.)
The prime number theorem in arithmetic progressions then gives
\begin{align*}
\sum_{p \leq z} \frac{\rho_\ell(p)}{p} = \log \log z + c_\ell + O_\ell( \exp(-c \sqrt{\log z})),
\end{align*}
for some constant $c_\ell$ depending on $\ell$. By \cite[(2.4)]{friediw1}, this implies
\begin{align}\label{eq: sum rho ell is zero}
\sum_{d = 1}^\infty\frac{\mu(d) \rho_\ell(d)}{d} = 0.
\end{align}
From \eqref{eq: sum rho ell is zero} and partial summation it follows that
\begin{align}\label{rho par sum}
&\sum_{d \leq V} \frac{\mu(d) \rho_\ell(d)}{d} = -\sum_{d > V} \frac{\mu(d) \rho_\ell(d)}{d} = \lim_{H \rightarrow \infty} \left(-\sum_{V < d \leq H} \frac{\mu(d) \rho_\ell(d)}{d} \right) \\ 
&= \lim_{H \rightarrow \infty} \left(-H^{-1} \sum_{d \leq H} \mu(d) \rho_\ell(d)  V^{-1} \sum_{d \leq V} \mu(d) \rho_\ell(d) +\int_V^H \frac{1}{t^2} \left(\sum_{d \leq t} \mu(d) \rho_\ell(d) \right)dt \right). \nonumber
\end{align}
We will show
\begin{align}\label{mob rho cancellation}
\sum_{d \leq z} \mu(d) \rho_\ell(d) \ll \prod_{p \mid \ell} \left(1 + \frac{26}{p^{2/3}} \right) \ z \exp(-c \sqrt{\log z}),
\end{align}
uniformly in $\ell$ and $z \geq 1$. The bound is trivial if $z$ is bounded, so we may suppose that $z$ is large.

Let $y = z \exp(-b \sqrt{\log z})$, where $b>0$ is a parameter to be chosen later. Let $g$ be a smooth function, supported in $[1/2,z]$, which is identically equal to one on $[y,z-y]$, and satisfies $g^{(j)} \ll_j y^{-j}$. Estimating trivially,
\begin{align}\label{mellin approx}
\sum_{d \leq z} \mu(d) \rho_\ell(d) &= O(y \log z) + \sum_{d} \mu(d) \rho_\ell(d) g(d).
\end{align}
Mellin inversion yields
\begin{align*}
\sum_{d} \mu(d) \rho_\ell(d) g(d) &= \frac{1}{2\pi i} \int_{(2)} \widehat{g}(s) \sum_{d=1}^\infty \frac{\mu(d) \rho_\ell(d)}{d^s} ds.
\end{align*}
From the derivative bounds on $g$ we find that the Mellin transform $\widehat{g}(s)$ satisfies
\begin{align}\label{mellin invert}
\widehat{g}(s) \ll z^\sigma \left(1 + (y/z)^2 t^2 \right)^{-1},
\end{align}
where $s = \sigma + it$ and $\sigma \geq \frac{2}{3}$, say.

An Euler product computation yields
\begin{align*}
\sum_{d=1}^\infty \frac{\mu(d) \rho_\ell(d)}{d^s} &= \zeta(s)^{-1} L(s,\chi)^{-1} H(s) f_s(\ell),
\end{align*}
where
\begin{align*}
H(s) &:= \prod_p \frac{1 - \frac{1 + \chi(p)}{p^s}}{\left(1 - \frac{1}{p^s} \right) \left(1 - \frac{\chi(p)}{p^s} \right)}
\end{align*}
is analytic in $\sigma \geq \frac{2}{3}$, say, and
\begin{align*}
f_s(\ell) &:= \prod_{p \mid \ell} \frac{1 - \frac{1}{p^s}}{1 - \frac{1 + \chi(p)}{p^s}} = \prod_{p \mid \ell} \left( 1 + \frac{\chi(p)}{p^s - 1 - \chi(p)} \right).
\end{align*}

We move the line of integration in \eqref{mellin invert} to $\sigma = 1 + \frac{1}{\log z}$ and estimate trivially the contribution from $|t| \geq T$, with $T$ a parameter to be chosen. This gives
\begin{align*}
\int_{|t| \geq T} &\ll (\log z)^{O(1)} \ \frac{z^3}{y^2T} \ \prod_{p \mid \ell} \left(1 + \frac{\chi^2(p)}{p-1-\chi^2(p)} \right) .
\end{align*}
For $|t| \leq T$ we move the line of integration to $\sigma = 1 - \frac{c}{\log T}$, where $c$ is chosen small enough that $\zeta(s)^{-1} L(s,\chi)^{-1}$ has no zeros in $\sigma \geq 1 - \frac{c}{\log T}, |t| \leq T$, and add in horizontal connecting lines. We estimate everything trivially to arrive at
\begin{align*}
\int_{|t| \leq T} &\ll z (\log z T)^{O(1)} \exp(2b \sqrt{\log z}) \prod_{p \mid \ell} \left(1 + \frac{\chi^2(p)}{p^{2/3} - 1 - \chi^2(p)} \right) \left(\frac{1}{T} + \exp \left(-c \frac{\log z}{\log T} \right) \right).
\end{align*}
We set $T = \exp(\sqrt{\log z})$, and take $b=\frac{c}{3}$. With a small amount of calculation we see that
\begin{align*}
\frac{\chi^2(p)}{p^{2/3} - 1 - \chi^2(p)} < \frac{26}{p^{2/3}},
\end{align*}
and this completes the proof of \eqref{mob rho cancellation}. 

The fact that $\ell$ is coprime to $\Pi$ implies
\begin{align*}
\prod_{p \mid \ell} \left(1 + \frac{26}{p^{2/3}} \right) \ll 1.
\end{align*}
From \eqref{rho par sum} we see that \eqref{mob rho cancellation} and $(\ell,\Pi)=1$ yield
\begin{align}\label{first term of mxuv}
\sum_{d \leq V} \frac{\mu(d) \rho_\ell(d)}{d} \ll \exp(-c \sqrt{\log V}).
\end{align}
This shows that the first term of \eqref{eq: main term MxUV expanded} satisfies the bound
\begin{align*}
\mathop{\sum \sum}_{\substack{g^2 + \ell^2 \leq x \\ (\ell,\Pi) = 1}} \mathbf{1}_\mathcal{A}(\ell)\log(g^2 + \ell^2) \sum_{d \leq V} \frac{\mu(d) \rho_\ell(d)}{d} \ll x^{\frac{1}{2} + \frac{\gamma_0}{2}} \exp (-c' \sqrt{\log x}),
\end{align*}
provided 
\begin{align*}
V \geq x^\delta
\end{align*}
for some absolute constant $\delta > 0$.

We turn to estimating
\begin{align*}
-\sum_{d \leq V} \frac{\mu(d) \rho_\ell(d) \log d}{d}.
\end{align*}
We add and subtract the quantity
\begin{align*}
\log V \sum_{d \leq V} \frac{\mu(d) \rho_\ell(d)}{d},
\end{align*}
which yields
\begin{align*}
-\sum_{d \leq V} \frac{\mu(d) \rho_\ell(d) \log d}{d} &= \sum_{d \leq V} \frac{\mu(d) \rho_\ell(d)}{d} \log (V/d) + O(\exp(-c \sqrt{\log V}))
\end{align*}
by \eqref{first term of mxuv}. From Perron's formula we obtain
\begin{align}\label{eq: mt log smoothing}
\sum_{d \leq V} \frac{\mu(d) \rho_\ell(d)}{d} \log (V/d) &= \frac{1}{2\pi i}\int_{(1)} \frac{x^s}{s^2} \sum_{d=1}^\infty \frac{\mu(d) \rho_\ell(d)}{d^{1+s}} ds.
\end{align}
An Euler product computation reveals
\begin{align*}
\sum_{d=1}^\infty \frac{\mu(d) \rho_\ell(d)}{d^{1+s}} &= \zeta(1+s)^{-1} L(1+s,\chi)^{-1} H(1+s) \prod_{p \mid \ell} \left( 1 + \frac{\chi(p)}{p^{1+s} - 1 - \chi(p)} \right).
\end{align*}
We proceed in nearly identical fashion to the proof of \eqref{mob rho cancellation}, but here there is a main term coming from the simple pole of the integrand in \eqref{eq: mt log smoothing} at $s = 0$. Since $L(1,\chi) = \frac{\pi}{4}$, we deduce
\begin{align}\label{partial main term}
-\sum_{d \leq V} \frac{\mu(d) \rho_\ell(d) \log d}{d} &= \frac{4}{\pi} \prod_{p \mid \ell} \left(1 + \frac{\chi(p)}{p - 1 - \chi(p)} \right) \prod_p \left(1 - \frac{\chi(p)}{(p-1) (p- \chi(p))} \right) \\ 
&+ O(\exp(-c \sqrt{\log V})). \nonumber
\end{align}
The expression in \eqref{partial main term} gives rise to the main term in Theorem \ref{main theorem}.

The last term of $M(x;U,V)$ we estimate similarly to the first. We aim to show that
\begin{align}\label{third term of mxuv}
\sum_{m \leq U} \frac{\Lambda(m)}{m} \sum_{d \leq V} \frac{\mu(d) \rho_\ell(dm)}{d} \ll (\log \ell V)^3 P^{-1/2}.
\end{align}
It is convenient to distinguish two cases for $d$: those $d$ that are coprime to $m$, and those that are not. If $d$ is not coprime to $m = p^k$, then the presence of the M\"obius function implies $d = ep$ with $(e,p)=1$. Therefore
\begin{align}\label{eq: MxUV mangoldt and mobius}
\sum_{m \leq U} \frac{\Lambda(m)}{m} \sum_{d \leq V} \frac{\mu(d) \rho_\ell(dm)}{d} &= \sum_{m \leq U} \frac{\Lambda(m)\rho_\ell(m)}{m} \sum_{\substack{d \leq V \\ (d,m)=1}} \frac{\mu(d) \rho_\ell(d)}{d} \\
&- \sum_{p^k \leq U} \frac{(\log p) \rho_\ell(p^{k+1})}{p^{k+1}} \sum_{\substack{e \leq V/p \\ (e,p)=1}} \frac{\mu(e) \rho_\ell(e)}{e}. \nonumber
\end{align}
It is not difficult to deal with the sum over $d$ in the first term of \eqref{eq: MxUV mangoldt and mobius} using an argument analogous to that which gave \eqref{first term of mxuv}, as the condition $(d,m) = 1$ causes no great complications. To bound the sum over $m$ we use Lemma \ref{lem: sums of rho ell}, obtaining
\begin{align*}
\sum_{m \leq U} \frac{\Lambda(m)\rho_\ell(m)}{m} &\leq \sum_{\substack{m \leq U \\ (m,\ell)=1}} \frac{\Lambda(m)\rho_\ell(m)}{m} + (\log U)\sum_{\substack{p^k \\ p \mid \ell}} \frac{\rho_\ell(p^k)}{p^k} \\
&\ll \log U + (\log U)\sum_{\substack{p^k \\ p \mid \ell}} \frac{p^{k/2}}{p^k} \ll \log U. \nonumber
\end{align*}
The last inequality follows since $p \mid \ell$ implies $p >  P$. Therefore
\begin{align}\label{third term of mxuv first part}
\sum_{m \leq U} \frac{\Lambda(m)\rho_\ell(m)}{m} \sum_{\substack{d \leq V \\ (d,m)=1}} \frac{\mu(d) \rho_\ell(d)}{d} \ll (\log U) \exp (-c \sqrt{\log V}).
\end{align}

We turn our attention to the second term of \eqref{eq: MxUV mangoldt and mobius}. We first remove those $p$ that are not coprime to $\ell$. By trivial estimation
\begin{align}\label{second term p divs ell}
\sum_{\substack{p^k \leq U \\ p \mid \ell}} \frac{(\log p) \rho_\ell(p^{k+1})}{p^{k+1}} \sum_{\substack{e \leq V/p \\ (e,p)=1}} \frac{\mu(e) \rho_\ell(e)}{e} &\ll (\log V)^2 \sum_{p \mid \ell} (\log p) \sum_{k=1}^\infty \frac{1}{p^{k/2}} \ll (\log \ell V)^3 P^{-1/2}
\end{align}
Here we have again used the fact that $P^-(\ell) > P$.

To handle those $p$ that are coprime to $\ell$, we assume that
\begin{align*}
U \geq x^\delta
\end{align*}
for some absolute constant $\delta > 0$. We then estimate trivially the contribution from $p > R = \exp(\sqrt{\log V})$. Observe that $R < U$. Then
\begin{align}\label{large p}
\sum_{\substack{p^k \leq U \\  p > R \\ (p,\ell)=1}} \frac{(\log p) \rho_\ell(p^{k+1})}{p^{k+1}} \sum_{\substack{e \leq V/p \\ (e,p)=1}} \frac{\mu(e) \rho_\ell(e)}{e} &\ll (\log V)^2 \sum_{p > R} \log p \sum_{k=2}^\infty \frac{k}{p^k} \\
&\ll (\log V)^2 \sum_{p > R} \frac{\log p}{p^2} \ll \frac{(\log V)^2}{R},\nonumber
\end{align}
and this is an acceptably small error. We may then show
\begin{align}\label{last slice of cake}
\sum_{\substack{p^k \leq U \\ p \leq R \\ (p,\ell)=1}} \frac{(\log p) \rho_\ell(p^{k+1})}{p^{k+1}} \sum_{\substack{e \leq V/p \\ (e,p)=1}} \frac{\mu(e) \rho_\ell(e)}{e} \ll \exp(-c \sqrt{\log V})
\end{align}
by arguing as before, since $V/p$ is close to $V$ in the logarithmic scale. Taking \eqref{third term of mxuv first part}, \eqref{second term p divs ell}, \eqref{large p}, and \eqref{last slice of cake} together gives \eqref{third term of mxuv}. We combine \eqref{first term of mxuv}, \eqref{partial main term}, and \eqref{third term of mxuv first part} to derive
\begin{align}\label{mxuv main term, before fund lemma}
M(x;U,V) &= \frac{4}{\pi}C \mathop{\sum \sum}_{\substack{m^2 + \ell^2 \leq x \\ (\ell,\Pi) = 1}} \mathbf{1}_\mathcal{A}(\ell) \prod_{p \mid \ell}\left(1 + \frac{\chi(p)}{p - 1 - \chi(p)} \right) + O \left((\log x)^3 x^{\frac{1}{2} + \frac{\gamma_0}{2}} P^{-1/2} \right),
\end{align}
provided $U,V \geq x^\delta$ for some absolute constant $\delta > 0$. Here
\begin{align*}
C &= \prod_p \left(1 - \frac{\chi(p)}{(p-1) (p- \chi(p))} \right)
\end{align*}
is the constant in Theorem \ref{main theorem}. Since $P^-(\ell) > P$ we have
\begin{align*}
\prod_{p \mid \ell}\left(1 + \frac{\chi(p)}{p - 1 - \chi(p)} \right) &= 1 + O \left(\frac{\log \ell}{P} \right),
\end{align*}
and so \eqref{mxuv main term, before fund lemma} becomes
\begin{align}\label{mxuv mt, removed f ell, before fund}
M(x;U,V) &= \frac{4}{\pi} C \mathop{\sum \sum}_{\substack{m^2 + \ell^2 \leq x \\ (\ell,\Pi) = 1}} \mathbf{1}_\mathcal{A}(\ell) + O \left((\log x)^3 x^{\frac{1}{2} + \frac{\gamma_0}{2}} P^{-1/2} \right).
\end{align}

\section{The sieve main term: fundamental lemma}\label{sec: sieve main, using fund}

We wish to simplify the main term of \eqref{mxuv mt, removed f ell, before fund} by removing the condition $(\ell,\Pi) = 1$, which we accomplish with the fundamental lemma of sieve theory. 

In order to apply the sieve we require information about the elements of $\mathcal{A}$ in arithmetic progressions. We invariably detect congruence conditions on elements of $\mathcal{A}$ via additive characters, so we require information on exponential sums over $\mathcal{A}$. It is convenient to normalize these exponential sums so that we may study them at different scales. For $Y$ an integral power of 10, we define
\begin{align}\label{defn of F}
F_Y(\theta) &:= Y^{-\log 9/\log 10} \left|\sum_{0 \leq n < Y} \mathbf{1}_\mathcal{A}(n) e(n\theta) \right|,
\end{align}
so $F_Y(\theta) \ll 1$ for all $Y$ and real numbers $\theta$. Observe that $F_Y$ is a periodic function with period 1. We emphasize that $Y$ is always a power of 10 when it appears in a subscript.

Let $U$ and $V$ be two integral powers of ten (here $U$ and $V$ have nothing to do with the $U$ and $V$ from Vaughan's identity \eqref{vaughans identity}). From the definition \eqref{defn of F} it is not difficult to derive (see \cite[p. 6]{maynard}) the identity
\begin{align}\label{product formula for F}
F_{UV}(\theta) = F_U(\theta) F_V(U\theta).
\end{align}

We take the opportunity to collect in one place the lemmas we need to estimate $F_Y$ and various averages of $F_Y$. 

The first result is a sort of Siegel-Walfisz result for $F_Y$.

\begin{lemma}\label{maynard small moduli}
Let $q < Y^{1/3}$ be of the form $q = q_1q_2$ with $(q_1,10) = 1$ and $q_1 > 1$. Then for any integer $a$ coprime to $q$ we have
\begin{align*}
F_Y \left( \frac{a}{q} \right) &\ll \exp \left( - c_0 \frac{\log Y}{\log q} \right)
\end{align*}
for some absolute constant $c_0 > 0$.
\end{lemma}
\begin{proof}
This is a slight weakening of \cite[Proposition 4.1]{maynard}.
\end{proof}

The next two lemmas are results of large sieve type for $F_Y$.

\begin{lemma}\label{maynard single denominator}
For $q \geq 1$ we have
\begin{align*}
\sup_{\beta \in \mathbb{R}} \ \sum_{a \leq q} F_X \left(\frac{a}{q} + \beta \right) \ll q^{27/77} + \frac{q}{X^{50/77}}.
\end{align*}
\end{lemma}
\begin{proof}
This is a slight weakening of the first part of \cite[Lemma 4.5]{maynard}.
\end{proof}

\begin{lemma}\label{maynard type I}
For $Q \geq 1$ we have
\begin{align*}
\sup_{\beta \in \mathbb{R}} \ \sum_{q \leq Q} \sum_{\substack{1 \leq a \leq q \\ (a,q)=1}} F_Y \left(\frac{a}{q} + \beta \right) \ll Q^{54/77} + \frac{Q^2}{Y^{50/77}}.
\end{align*}
\end{lemma}
\begin{proof}
This is a slight weakening of the second part of \cite[Lemma 4.5]{maynard}.
\end{proof}

Now that the necessary results are in place, we proceed with the estimation of the main term in \eqref{mxuv mt, removed f ell, before fund}. We write
\begin{align*}
\mathop{\sum \sum}_{\substack{m^2 + \ell^2 \leq x \\ (\ell,\Pi) = 1}} \mathbf{1}_\mathcal{A}(\ell) &= \sum_{|m| \leq x^{1/2}} \sum_{\substack{\ell \leq \sqrt{x - m^2} \\ (\ell,\Pi)=1}} \mathbf{1}_\mathcal{A}(\ell) \\
&= \sum_{|m| \leq \sqrt{1 - P^{-2}} x^{1/2}} \ \sum_{\substack{\ell \leq \sqrt{x - m^2} \\ (\ell,\Pi)=1}} \mathbf{1}_\mathcal{A}(\ell) + O(x^{\frac{1}{2} + \frac{\gamma_0}{2}} P^{-1}),
\end{align*}
the second equality following by trivial estimation.

With the restriction $|m| \leq \sqrt{1 - P^{-2}} x^{1/2}$ on $m$ we estimate each sum over $\ell$ individually. Set $z = z(m) = \sqrt{x - m^2}$. We apply upper- and lower-bound linear sieves of level 
\begin{align*}
D = z^{1/5}
\end{align*}
(see \cite[Chapter 5]{opera} for terminology). Therefore
\begin{align}\label{upp and low bd sieve weights}
\sum_{\substack{d \leq D \\ d \mid \Pi \\ (d,10)=1}} \lambda_d^- \sum_{\substack{\ell \leq z \\ \ell \equiv 0 (d) \\ (\ell,10)=1}} \mathbf{1}_\mathcal{A}(\ell) &\leq \sum_{\substack{\ell \leq z \\ (\ell,\Pi)=1}} \mathbf{1}_\mathcal{A}(\ell) \leq \sum_{\substack{d \leq D \\ d \mid \Pi \\ (d,10)=1}} \lambda_d^+ \sum_{\substack{\ell \leq z \\ \ell \equiv 0 (d) \\ (\ell,10)=1}} \mathbf{1}_\mathcal{A}(\ell).
\end{align}
The upper and lower bounds turn out to be asymptotically equal, and we write $\lambda_d$ for $\lambda_d^+$ or $\lambda_d^-$.

It is difficult to work with elements of $\mathcal{A}$ over intervals whose lengths are not a power of 10. We put ourselves in this situation with a short interval decomposition (a similar technique is applied in \cite{bour}). Let $Y$ be the largest power of 10 that satisfies $Y \leq z P^{-1}$. We break the summation over $\ell$ into intervals of the form $[nY,(n+1)Y)$, where $n$ is a nonnegative integer. This gives
\begin{align}\label{fund lem short interval}
\sum_{\substack{\ell \leq z \\ \ell \equiv 0 (d) \\ (\ell,10)=1}} \mathbf{1}_\mathcal{A}(\ell) &= \sum_{n \in S(z)} \ \sum_{\substack{nY \leq \ell < (n+1)Y \\ \ell \equiv 0 (d) \\ (\ell,10)=1}} \mathbf{1}_\mathcal{A}(\ell) + O \left(\sum_{\substack{z - Y \leq \ell \leq z + Y \\ \ell \equiv 0 (d)}} \mathbf{1}_\mathcal{A}(\ell) \right).
\end{align}
Here $S(z)$ is some set of size $\ll P$. We remark that we will repeatedly see this technique of breaking an interval into shorter subintervals, with each subinterval having length a power of 10, in the estimation of the bilinear sum $B(x;U,V)$.

We first illustrate how to use Lemma \ref{maynard type I} to handle the error term in \eqref{fund lem short interval}. On summing over $d$, we must estimate
\begin{align*}
\mathcal{E} &:= \sum_{d \leq D} \sum_{\substack{z - Y \leq \ell \leq z + Y \\ \ell \equiv 0 (d)}} \mathbf{1}_\mathcal{A}(\ell).
\end{align*}
Because the estimation of $\mathcal{E}$ introduces a number of important ideas that we use throughout the proof of Theorem \ref{main theorem}, we encapsulate the estimation in a lemma.
\begin{lemma}\label{prototypical estimation lemma one}
With the notation as above,
\begin{align*}
\mathcal{E} &\ll (\log D)^2 \ Y^{\gamma_0}.
\end{align*}
\end{lemma}
\begin{proof}
For $X$ some power of 10 with $Y \leq X \ll Y$ and some integer $k$ depending only on $z,Y$, and $X$, we have
\begin{align*}
\mathcal{E} &\leq \sum_{d \leq D} \ \sum_{\substack{kX \leq \ell < (k+1)X \\ \ell \equiv 0 (d)}}\mathbf{1}_\mathcal{A}(\ell).
\end{align*}
If $\mathbf{1}_\mathcal{A}(k) = 0$ then the sum over $\ell$ is empty and $\mathcal{E} = 0$. Suppose then that $\mathbf{1}_\mathcal{A}(k) = 1$. We write $\ell = kX + t$, where $0 \leq t < X$. There are now two subcases to consider, depending on whether the missing $a_0$ is equal to 0 or not. If $a_0 \neq 0$ then $\mathbf{1}_\mathcal{A}(kX + t) = \mathbf{1}_\mathcal{A}(t)$ for $0 \leq t < X$. If $a_0 = 0$ then $\mathbf{1}_\mathcal{A}(kX + t) = 0$ for $0 \leq t < X/10$ and $\mathbf{1}_\mathcal{A}(kX + t) = \mathbf{1}_\mathcal{A}(t)$ for $X/10 \leq t < X$. We can unite the two subcases by writing
\begin{align*}
\mathcal{E} &\leq \sum_{d \leq D} \ \sum_{\substack{\delta(a_0)X/10 \leq t < X \\ t \equiv -kX (d)}} \mathbf{1}_\mathcal{A}(t),
\end{align*}
where
\begin{align*}
\delta(n) = 
\begin{cases}
1, \ \ \ n = 0, \\
0, \ \ \ n \neq 0.
\end{cases}
\end{align*}
By inclusion-exclusion and the triangle inequality we find
\begin{align*}
\mathcal{E} &\ll \sum_{d \leq D}\sum_{\substack{t < X \\ t \equiv -kX (d)}} \mathbf{1}_\mathcal{A}(t).
\end{align*}
We detect the congruence via the orthogonality of additive characters, which yields
\begin{align*}
\mathcal{E} &\ll \sum_{d \leq D} \frac{1}{d} \sum_{r=1}^d e \left(\frac{rkX}{d} \right)\sum_{t < X} \mathbf{1}_\mathcal{A}(t) e \left(\frac{rt}{d} \right).
\end{align*}
By the triangle inequality,
\begin{align*}
\mathcal{E} &\ll X^{\gamma_0}\sum_{d \leq D}\frac{1}{d} \sum_{r=1}^d F_X \left(\frac{r}{d} \right).
\end{align*}
We remove the terms with $r =d$ (the ``zero'' frequency), which gives
\begin{align*}
\mathcal{E} &\ll (\log D)X^{\gamma_0} + X^{\gamma_0}\sum_{1 <d \leq D}\frac{1}{d} \sum_{r=1}^{d-1} F_X \left(\frac{r}{d} \right).
\end{align*}
For the ``non-zero'' frequencies we reduce to primitive fractions and obtain
\begin{align*}
\sum_{1 <d \leq D}\frac{1}{d} \sum_{r=1}^{d-1} F_X \left(\frac{r}{d} \right) &= \sum_{1 < d \leq D} \frac{1}{d}\sum_{\substack{q \mid d \\ q > 1}} \sum_{\substack{1 \leq b \leq q \\ (b,q)=1}} F_X \left( \frac{b}{q} \right) \ll (\log D) \sum_{1<q \leq D} \frac{1}{q} \sum_{\substack{1 \leq b \leq q \\ (b,q)=1}} F_X \left( \frac{b}{q} \right).
\end{align*}
We perform a dyadic decomposition on the range of $q$ to get
\begin{align*}
\mathcal{E} &\ll (\log D)^2 \ X^{\gamma_0} \ \sup_{Q \leq D} \frac{1}{Q} \sum_{q \leq Q} \sum_{\substack{1 \leq b \leq q \\ (b,q)=1}} F_X \left( \frac{b}{q} \right).
\end{align*}
By Lemma \ref{maynard type I},
\begin{align*}
\mathcal{E} &\ll (\log D)^2 \ X^{\gamma_0} \sup_{Q \leq D} \left(\frac{1}{Q^{23/77}} + \frac{Q}{X^{50/77}}  \right) \ll (\log D)^2 \ X^{\gamma_0} \left(1 + \frac{D}{X^{50/77}}\right) \\
&\ll (\log D)^2 \ X^{\gamma_0},
\end{align*}
and this completes the proof.
\end{proof}

From Lemma \ref{prototypical estimation lemma one} it follows that
\begin{align*}
\sum_{\substack{d \leq D \\ d \mid \Pi \\ (d,10)=1}} \lambda_d \sum_{\substack{\ell \leq z \\ \ell \equiv 0 (d) \\ (\ell,10)=1}} \mathbf{1}_\mathcal{A}(\ell) &= \sum_{n \in S(z)} \mathbf{1}_\mathcal{A}(n) \sum_{\substack{d \leq D \\ d \mid \Pi \\ (d,10)=1}} \lambda_d \sum_{\substack{\delta(a_0)Y/10 \leq t < Y \\ t \equiv -nY (d) \\ (t,10)=1}} \mathbf{1}_\mathcal{A}(t) + O(x^{\gamma_0/2} P^{-1/2}).
\end{align*}
We detect the congruence with additive characters and obtain
\begin{align*}
\sum_{\substack{\delta(a_0)Y/10 \leq t < Y \\ t \equiv -nY (d) \\ (t,10)=1}} \mathbf{1}_\mathcal{A}(t) = \frac{1}{d} \sum_{r=1}^d e \left(\frac{r nY}{d} \right) \sum_{\substack{\delta(a_0)Y/10 \leq t < Y \\ (t,10)=1}} \mathbf{1}_\mathcal{A}(t) e \left(\frac{rt}{d} \right).
\end{align*}
Naturally we extract the main term from $r = d$.

Define
\begin{align*}
\kappa &:=
\begin{cases}
\frac{\varphi(10)}{9}, \ \ \  (a_0,10)\neq 1, \\
\frac{\varphi(10)-1}{9}, \ \ \  (a_0,10) = 1.
\end{cases}
\end{align*}
It is easy to check that
\begin{align*}
\sum_{\substack{t < 10^k \\ (t,10)=1}} \mathbf{1}_\mathcal{A}(t) = \kappa \sum_{t < 10^k} \mathbf{1}_\mathcal{A}(t),
\end{align*}
which implies
\begin{align*}
\sum_{\substack{\delta(a_0)Y/10 \leq t < Y \\ (t,10)=1}} \mathbf{1}_\mathcal{A}(t) &= \kappa \sum_{\substack{\delta(a_0)Y/10 \leq t < Y}} \mathbf{1}_\mathcal{A}(t).
\end{align*}

For $1 \leq r \leq d-1$ we handle the condition $(t,10) = 1$ with M\"obius inversion. We then reverse our short interval decomposition to get
\begin{align}\label{intermed fund lemma}
\sum_{\substack{d \leq D \\ d \mid \Pi \\ (d,10)=1}} \lambda_d \sum_{\substack{\ell \leq z \\ \ell \equiv 0 (d) \\ (\ell,10)=1}} \mathbf{1}_\mathcal{A}(\ell) &= \kappa \sum_{\substack{d \leq D \\ d \mid \Pi \\ (d,10)=1}} \frac{\lambda_d}{d} \sum_{\ell \leq z} \mathbf{1}_\mathcal{A}(\ell) + O \left(x^{\gamma_0/2} \left(P^{-1/2} + E P^{\gamma_0} \right) \right),
\end{align}
where
\begin{align*}
E &:= \sum_{\substack{1<d \leq D \\ (d,10)=1}}\frac{1}{d} \sum_{e \mid 10} \sum_{\substack{r=1}}^{d-1} F_X \left(\frac{r}{d} + \frac{e}{10} \right)
\end{align*}
and $X$ is a power of 10 with $X \asymp Y$. Similarly to the estimation of $\mathcal{E}$ in Lemma \ref{prototypical estimation lemma one} above, we put, for pedagogical reasons, the estimation of $E$ into a lemma.
\begin{lemma}\label{prototypical estimation lemma two}
With the notation given above,
\begin{align*}
E \ll \exp(-c \sqrt{\log z})
\end{align*}
for some absolute constant $c > 0$.
\end{lemma}
\begin{proof}
We reduce to primitive fractions to derive
\begin{align*}
E &= \sum_{\substack{1<d \leq D \\ (d,10)=1}}\frac{1}{d} \sum_{e \mid 10} \sum_{\substack{q \mid d \\ q > 1}} \sum_{\substack{a = 1 \\ (a,q)=1}}^q F_X \left(\frac{a}{q} + \frac{e}{10} \right).
\end{align*}
We apply \eqref{product formula for F} with $U= 10, V = X/10$ to obtain
\begin{align*}
F_X \left(\frac{a}{q} + \frac{e}{10} \right) &= F_{10} \left(\frac{a}{q} + \frac{e}{10}  \right) F_V \left(\frac{10a}{q} + \frac{10e}{10}  \right) \\
&= F_{10} \left(\frac{a}{q} + \frac{e}{10}  \right)F_V \left(\frac{10a}{q} \right) \ll F_V \left(\frac{10a}{q} \right).
\end{align*}
Since $(10,q) = 1$, we may change variables $10a \rightarrow a$ to obtain
\begin{align*}
E &\ll \sum_{\substack{1<d \leq D \\ (d,10)=1}}\frac{1}{d} \sum_{\substack{q \mid d \\ q > 1}} \sum_{\substack{a = 1 \\ (a,q)=1}}^q F_V \left(\frac{a}{q} \right) \ll (\log D) \sum_{\substack{1 < q \leq D \\ (q,10)=1}} \frac{1}{q}\sum_{\substack{a = 1 \\ (a,q)=1}}^q F_V \left(\frac{a}{q} \right).
\end{align*}
We perform a dyadic decomposition on the range of $q$ to obtain
\begin{align*}
E &\ll (\log D)^2 \sup_{Q \leq D} \frac{1}{Q} \sum_{\substack{1 < q \asymp Q \\ (q,10)=1}} \sum_{\substack{a = 1 \\ (a,q)=1}}^q F_V \left(\frac{a}{q} \right).
\end{align*}
Set $Q_1 = \exp(\varepsilon\sqrt{\log z})$, where $\varepsilon > 0$ is a small positive constant. If $Q > Q_1$ we use Lemma \ref{maynard type I}, and if $Q \leq Q_1$ we use Lemma \ref{maynard small moduli}. Provided $\varepsilon$ in the definition of $Q_1$ is taken sufficiently small in terms of $c_0$ in Lemma \ref{maynard small moduli}, we obtain
\begin{align*}
E &\ll \exp (-c \sqrt{\log z}),
\end{align*}
where $c > 0$ is some absolute constant.
\end{proof}

We take \eqref{intermed fund lemma} with Lemma \ref{prototypical estimation lemma two}, along with the fact that $\log P = o(\sqrt{\log z})$, to get
\begin{align}\label{intermed fund lemma errors dispatched}
\sum_{\substack{d \leq D \\ d \mid \Pi \\ (d,10)=1}} \lambda_d \sum_{\substack{\ell \leq z \\ \ell \equiv 0 (d) \\ (\ell,10)=1}} \mathbf{1}_\mathcal{A}(\ell) &= \kappa \sum_{\substack{d \leq D \\ d \mid \Pi \\ (d,10)=1}} \frac{\lambda_d}{d} \sum_{\ell \leq z} \mathbf{1}_\mathcal{A}(\ell) + O \left(x^{\gamma_0/2} P^{-1/2} \right).
\end{align}
By the fundamental lemma of the linear sieve (see \cite[Lemma 6.11]{opera})
\begin{align}\label{fund lem ref}
\sum_{\substack{d \leq D \\ d \mid \Pi \\ (d,10)=1}} \frac{\lambda_d}{d} &= \left(1 + O \left(\exp \left( - \frac{1}{2}s \log s \right) \right) \right) \prod_{\substack{p \leq P \\ p \nmid 10}} \left(1 - \frac{1}{p} \right),
\end{align}
where
\begin{align*}
s &= \frac{\log D}{\log P} \gg (\log \log x) \sqrt{\log x}.
\end{align*}
It follows that
\begin{align*}
\sum_{\substack{d \leq D \\ d \mid \Pi \\ (d,10)=1}} \lambda_d \sum_{\substack{\ell \leq z \\ \ell \equiv 0 (d) \\ (\ell,10)=1}} \mathbf{1}_\mathcal{A}(\ell) &= \frac{10}{\varphi(10)} \kappa \prod_{\substack{p \leq P}} \left(1 - \frac{1}{p} \right) \sum_{\ell \leq z} \mathbf{1}_\mathcal{A}(\ell) + O \left(x^{\gamma_0/2} P^{-1/2} \right),
\end{align*}
where $\varphi$ is the Euler totient function. We use Mertens' theorem with prime number theorem error term to get
\begin{align*}
\prod_{p \leq P} \left(1 - \frac{1}{p} \right) &= \frac{e^{-\gamma}}{\log P} \left(1 + O \left(\exp \left(-c \sqrt{\log P} \right) \right) \right),
\end{align*}
for some constant $c > 0$. Observe that our lower bound for $\log P$ implies
\begin{align*}
\exp \left(-c \sqrt{\log P} \right) \leq \exp \left(-c (\log \log x)^2 \right),
\end{align*}
so this error term is acceptable for Theorem \ref{main theorem}. Therefore,
\begin{align}\label{mt after fund lemma}
\mathop{\sum \sum}_{\substack{m^2 + \ell^2 \leq x \\ (\ell,\Pi) = 1}} \mathbf{1}_\mathcal{A}(\ell) &= \frac{10}{\varphi(10)} \kappa \frac{e^{-\gamma}}{\log P} \mathop{\sum \sum}_{\substack{m^2 + \ell^2 \leq x}} \mathbf{1}_\mathcal{A}(\ell) + O \left(x^{\frac{1}{2} + \frac{\gamma_0}{2}} \exp \left(-c \sqrt{\log P} \right) \right).
\end{align}
Combining \eqref{mt after fund lemma} with \eqref{mxuv mt, removed f ell, before fund} yields the main term of Theorem \ref{main theorem}.

\section{Bilinear form in the sieve: first steps}\label{sec: first bilinear section}

Let us summarize what we have accomplished thus far. We take \eqref{vaughan decomposition}, \eqref{type i approx main and error}, \eqref{bound for sieve remainder}, \eqref{mxuv mt, removed f ell, before fund}, and \eqref{mt after fund lemma} to obtain
\begin{align}\label{intermed summary}
S(x) &= \frac{4 C \kappa_1}{\pi} \frac{e^{-\gamma}}{\log P} \mathop{\sum \sum}_{m^2 + \ell^2 \leq x} \mathbf{1}_\mathcal{A}(\ell) + B(x;U,V) + O \left(x^{\frac{1}{2} + \frac{\gamma_0}{2}} \exp \left(-c \sqrt{\log P} \right) \right).
\end{align}
This holds provided our parameters $U,V$ in Vaughan's identity \eqref{vaughans identity} satisfy
\begin{align}\label{intermediate restrictions on U,V}
UV \leq x^{\gamma_0 - \delta},  \ \ \ \ \ \ \ \ \ U,V \geq x^\delta,
\end{align}
for some absolute $\delta > 0$. The task now is to show that the Type II sum $B(x;U,V)$ contributes only to the error term of \eqref{intermed summary}. (We note that the implied constant in the error term of \eqref{intermed summary} is effectively computable.)

In the course of our estimations we encounter more severe restrictions on $U,V$ than those in \eqref{intermediate restrictions on U,V}, and we note these as we go along. It transpires that there is some flexibility in choosing $U$ and $V$. For those unwilling to wait in suspense, we mention that the choices
\begin{align}\label{sample choices of U,V}
U = x^{7/10}, \ \ V = x^{1/5},
\end{align}
say, are acceptable for Theorem \ref{main theorem} (see \eqref{restrictions on N} and \eqref{admissible choices for U,V}).

Recall that
\begin{align*}
B(x;U,V) &= \sum_{U < m \leq x/V} (\Lambda_{>U}\star 1)(m) \sum_{V < n \leq x/m} \mu(n) a(mn).
\end{align*}
We shall prove the following proposition.
\begin{proposition}\label{main proposition}
Let $x$ be large, and let $L > 0$ be fixed. There exist absolute constants $C,\omega > 0$ such that
\begin{align*}
B(x;U,V) &\ll_L (\log x)^{- \omega L+C} \ x^{\frac{1}{2} + \frac{\gamma_0}{2}}.
\end{align*}
The implied constant is ineffective.
\end{proposition}
Observe that Proposition \ref{main proposition} implies Theorem \ref{main theorem} by taking $L = \omega^{-1}(A+C)$. To prove Theorem \ref{main theorem} it therefore suffices to prove Proposition \ref{main proposition}, and this task occupies the remainder of the paper. In what follows we allow implied constants to depend on $L$ without indicating it in the notation.

In this section we perform technical reductions that reduce the estimation of $B(x;U,V)$ to the estimation of sums of the form
\begin{align}\label{eq: goal for first bilinear section}
\sum_m \left|\sum_{\substack{(n,m\Pi)=1}} \mu(n) a(mn) \right|.
\end{align}
Here the intervals of summation of $m$ and $n$ are independent of one another. This separation of variables is accomplished by a short interval decomposition. Once $m$ and $n$ are separated, we remove the small prime factors of $n$ and transfer them to $m$.  This has the immediate benefit of insuring that $m$ and $n$ are almost always coprime, but also confers substantial technical advantages in later calculations. The contribution from those $m$ and $n$ satisfying $(m,n) > 1$ is then trivially estimated and shown to be negligible, which gives the reduction to \eqref{eq: goal for first bilinear section} (see \eqref{bound for B1MN}).

We mention that the arguments in this section have some similarity to those in \cite[Section 10]{friediw1} and \cite[Section 4]{friediw2}.

Since $(\Lambda_{>U} \star 1)(m) \leq (\Lambda \star 1)(m) = \log m$, we see
\begin{align}\label{bxuv remove m weight}
B(x;U,V) &\leq (\log x) \sum_{U < m \leq x/V} \left| \sum_{V < n \leq x/m} \mu(n) a(mn) \right|.
\end{align}
It is easy to obtain a trivial bound for $B(x;U,V)$ that is not far from the bound of Proposition \ref{main proposition}.
\begin{lemma}\label{trivial bound for bxuv}
Let $x \geq 2$. Then
\begin{align*}
B(x;U,V) &\ll (\log x)^3 x^{\frac{1}{2} + \frac{\gamma_0}{2}}.
\end{align*}
\end{lemma}
\begin{proof}
We change variables in \eqref{bxuv remove m weight} and deduce
\begin{align*}
B(x;U,V) &\ll (\log x)\sum_{k \leq x} a(k) \tau(k) \ll (\log x) \sum_{d \leq x^{1/2}} \sum_{\substack{k \leq x \\ d \mid k}} a(k) = (\log x)\sum_{d \leq x^{1/2}}A_d(x).
\end{align*}
We write $A_d(x) = M_d(x) + R_d(x)$ and use Lemma \ref{Type I remainder} to bound the sum of $R_d(x)$, giving
\begin{align}\label{rho ell difficulties}
B(x;U,V) &\ll (\log x)\mathop{\sum \sum}_{\substack{m^2 + \ell^2 \leq x \\ (\ell,\Pi)=1}} \mathbf{1}_\mathcal{A}(\ell) \sum_{d \leq x^{1/2}} \frac{\rho_\ell(d)}{d}.
\end{align}
By Lemma \ref{lem: sums of rho ell} we have
\begin{align*}
B(x;U,V) &\ll (\log x)^3 \mathop{\sum \sum}_{\substack{m^2 + \ell^2 \leq x \\ (\ell,\Pi)=1}} \mathbf{1}_\mathcal{A}(\ell) \prod_{p \mid \ell} \left(1 + \frac{7}{p^{1/2}}\right) \ll (\log x)^3 x^{\frac{1}{2} + \frac{\gamma_0}{2}},
\end{align*}
the last inequality following since $P^-(\ell) > P$.
\end{proof}

In the proof of Lemma \ref{trivial bound for bxuv} we used Lemma \ref{lem: sums of rho ell} to control averages of $\rho_\ell$. We shall need more elaborate versions of this argument in several of our reductions of $B(x;U,V)$.

Our first step is to separate the variables $m$ and $n$ so that they run independently over intervals of summation. We accomplish this with a short interval decomposition. Set 
\begin{align}\label{defn of theta}
\theta = (\log x)^{-L},
\end{align}
where $L>0$ is fixed as in Proposition \ref{main proposition}. We break the summation over $n$ into subintervals $N < n \leq (1+\theta) N$. By the triangle inequality,
\begin{align*}
B(x;U,V) &\leq (\log x)\sum_{\substack{N = (1+\theta)^j V \\ j \geq 0 \\ N \leq x/U}} \ \sum_{U < m \leq x/V} \left|\sum_{\substack{N < n \leq (1+\theta) N \\ mn \leq x}} \mu(n) a(mn) \right|.
\end{align*}
We wish to replace the condition $mn\leq x$ with $mN \leq x$. Clearly $mn \leq x$ implies $mN \leq x$ since $n > N$. Thus, suppose $mN \leq x$ but $mn > x$. Then
\begin{align*}
x < mn \leq (1+\theta)mN \leq (1+\theta) x,
\end{align*}
and so for fixed $N$ we have
\begin{align*}
\sum_{U < m \leq x/V} \left|\sum_{\substack{N < n \leq (1+\theta) N \\ mN \leq x \\ mn > x}} \mu(n) a(mn) \right| &\leq \sum_{N < n \leq (1+\theta)N} \mu^2(n) \sum_{\substack{x < k \leq (1+\theta)x \\ n \mid k}} a(k).
\end{align*}
In order to apply Lemma \ref{Type I remainder} we require $N \leq x/U \ll x^{\gamma_0 - \epsilon}$, which holds provided
\begin{align}\label{first bound for U}
U \geq x^{1-\gamma_0 + \epsilon}.
\end{align}
This supersedes the lower bound for $U$ in \eqref{intermediate restrictions on U,V}. Assuming \eqref{first bound for U} we obtain by Lemma \ref{Type I remainder}
\begin{align*}
&\sum_{N < n \leq (1+\theta)N} \mu^2(n) \sum_{\substack{x < k \leq (1+\theta)x \\ n \mid k}} a(k) \ll \mathop{\sum \sum}_{x < m^2 + \ell^2 \leq (1+\theta)x}\mathbf{1}_\mathcal{A}(\ell) \sum_{N < n \leq (1+\theta)N} \frac{\mu^2(n)\rho_\ell(n)}{n} \\
&\ll \mathop{\sum \sum}_{x < m^2 + \ell^2 \leq (1+\theta)x}\mathbf{1}_\mathcal{A}(\ell) \sum_{N < n \leq (1+\theta)N} \frac{\tau(n)}{n} \ll \theta (\log N)\mathop{\sum \sum}_{x < m^2 + \ell^2 \leq (1+\theta)x}\mathbf{1}_\mathcal{A}(\ell).
\end{align*}
We have
\begin{align*}
\mathop{\sum \sum}_{x < m^2 + \ell^2 \leq (1+\theta)x}\mathbf{1}_\mathcal{A}(\ell) &= \sum_{\ell \leq \sqrt{(1+\theta)x}} \mathbf{1}_\mathcal{A}(\ell) \ \sum_{\sqrt{x - \ell^2} < |m| \leq \sqrt{(1+\theta)x - \ell^2}}1 \\
&\ll \theta^{1/2} x^{1/2} \sum_{\ell \ll x^{1/2}} \mathbf{1}_\mathcal{A}(\ell) \ll \theta^{1/2}x^{\frac{1}{2} + \frac{\gamma_0}{2}}.
\end{align*}
Since the number of intervals $N < n \leq (1+\theta)N$ is $\ll (\log x) \theta^{-1}$ we see
\begin{align}\label{bxuv after sep vars}
B(x;U,V) &\ll (\log x)\sum_{\substack{N = (1+\theta)^j V \\ j \geq 0 \\ N \leq x/U}} \sum_{\substack{U < m \leq x/N}} \left|\sum_{\substack{N < n \leq (1+\theta) N}} \mu(n) a(mn) \right| + O((\log x)^3 \theta^{1/2} x^{\frac{1}{2} + \frac{\gamma_0}{2}}).
\end{align}

We now fix one such $N$ with $V \leq N \leq x/U$, and perform a dyadic decomposition on the range of $m$, which yields
\begin{align*}
\sum_{\substack{U < m \leq x/N}} \left|\sum_{\substack{N < n \leq (1+\theta) N}} \mu(n) a(mn) \right| &\leq \sum_{\substack{M = 2^j U \\ j \geq 0 \\ MN \leq x}} \ \sum_{M < m \leq 2M} \left|\sum_{\substack{N < n \leq (1+\theta) N}} \mu(n) a(mn) \right|.
\end{align*}
We define
\begin{align*}
B_1(M,N) &:= \sum_{M < m \leq 2M} \left|\sum_{\substack{N < n \leq (1+\theta) N}} \mu(n) a(mn) \right|,
\end{align*}
and note that the variables $m$ and $n$ are separated in $B_1(M,N)$.

Observe that if $MN \leq \theta x$ then
\begin{align}\label{bound on BMN for small MN}
B_1(M,N) &\ll \sum_{N < n \leq (1+\theta)N} \mu^2(n) \sum_{\substack{k \ll \theta x \\ n \mid k}} a(k) \ll (\log N) \theta^{1 + \gamma_0} x^{\frac{1}{2} + \frac{\gamma_0}{2}},
\end{align}
the latter inequality following essentially by the argument that gave \eqref{bxuv after sep vars}. In order to prove Proposition \ref{main proposition} it suffices by virtue of \eqref{bxuv after sep vars} and \eqref{bound on BMN for small MN} to prove the following result.
\begin{proposition}\label{prop for BMN}
Let $x$ be large, and for $L > 0$ fixed set $\theta = (\log x)^{-L}$. We then have
\begin{align*}
B_1(M,N) &\ll_{\epsilon,L} (\log MN)^{O(1)} \theta^{5/2} (MN)^{\frac{1}{2} + \frac{\gamma_0}{2}}
\end{align*}
uniformly in
\begin{align}\label{restrictions on N}
x^{1/2 - \gamma_0/2 + \epsilon} \leq N \leq x^{25/77 - \epsilon}, \ \ \ \ \ \  \theta x < MN \leq x.
\end{align}
The implied constant is ineffective.
\end{proposition}
It is not yet apparent why $N$ must be of the size given in \eqref{restrictions on N}. We gradually introduce stronger conditions on $N$ as the proof requires, and find in the last instance that \eqref{restrictions on N} is sufficient.

Now that the variables $m$ and $n$ are separated from one another, we wish to remove the small prime factors from $n$. We write $n = n_0 n_1$, where $(n_0,\Pi) = 1$ and $n_1 \mid \Pi$, then set
\begin{align*}
C &= \exp ((\log P)^2)
\end{align*}
(there should be no cause to confuse the $C$ given here with the absolute constant $C$ in Proposition \ref{main proposition}). Observe that $C>P$, and that $C = x^{o(1)}$ by our upper bound for $\log P$ in Theorem \ref{main theorem}.

We first show that the contribution from $n_1 > C$ to $B_1(M,N)$ is negligible. If $n_1$ divides $\Pi$ and $n_1 > C$, then there is a divisor $d$ of $n_1$ that satisfies $C < d \leq CP$. Indeed, writing $n_1 = p_1 \cdots p_r$ where $p_1 < \cdots < p_r$, we see there is a minimal $j$ such that $p_1 \cdots p_j \leq C$ but $p_1 \cdots p_{j+1} > C$. The desired divisor is $d = p_1 \cdots p_{j+1}$. The contribution to $B_1(M,N)$ from $n_1 > C$ is
\begin{align*}
\sum_{M < m \leq 2M} &\left|\mathop{\sum \sum}_{\substack{N < n_0n_1 \leq (1+\theta)N \\ (n_0,\Pi)=1 \\ n_1 \mid \Pi \\ n_1 > C}} \mu(n_0) \mu(n_1) a(mn_0n_1) \right| \leq \sum_{\substack{C < d \leq CP \\ d \mid \Pi}} \sum_{\substack{n \ll MN \\ d | n}} a(n) \tau_3(n) \\
&\leq \sum_{\substack{C < d \leq CP \\ d \mid \Pi}} \sum_{\substack{n \ll MN \\ d | n}} a(n) \tau(n)^2 =: B_1',
\end{align*}
say. We utilize the following lemma to handle the divisor function.
\begin{lemma}\label{div reduc lemma}
For any $n,k \geq 1$ there exists a divisor $d \mid n$ such that $d \leq n^{1/2^k}$ and
\begin{align*}
\tau(n) \leq 2^{2^k-1} \tau(d)^{2^k}.
\end{align*}
\end{lemma}
\begin{proof}
This is \cite[Lemma 4]{hbli}.
\end{proof}
Applying Lemma \ref{div reduc lemma} with $k=2$ yields
\begin{align*}
B_1' &\ll \sum_{\substack{C < d \leq CP \\ d \mid \Pi}} \ \sum_{e \ll (MN)^{1/4}}\tau(e)^8 \sum_{\substack{n \ll MN \\ [d,e] \mid n}} a(n),
\end{align*}
where $[d,e]$ is the least common multiple of $d$ and $e$. By trivial estimation (i.e. no need for recourse to Lemma \ref{Type I remainder} since $[d,e] \leq (MN)^{1/4 + o(1)}$ is so small) we find that
\begin{align*}
\sum_{\substack{n \ll MN \\ [d,e] \mid n}} a(n) &\ll (MN)^{1/2} \sum_{\substack{\ell \leq (MN)^{1/2} \\ (\ell,\Pi)=1}} \mathbf{1}_\mathcal{A}(\ell) \frac{\rho_\ell([d,e])}{[d,e]}.
\end{align*}

Recall that $P^+(n)$ and $P^-(n)$ denote the greatest and least prime factors of $n$, respectively. We factor $e$ uniquely as $e = rs$, where $P^+(r) \leq P$ and $P^-(s) > P$. Thus
\begin{align*}
&B_1' \ll (MN)^{1/2} \sum_{\substack{\ell \ll (MN)^{1/2} \\ (\ell,\Pi)=1}} \mathbf{1}_\mathcal{A}(\ell) \sum_{\substack{C < d \leq CP \\ d \mid \Pi}} \sum_{e \ll (MN)^{1/4}} \frac{\tau(e)^8 \rho_\ell([d,e])}{[d,e]} \\
&\ll (MN)^{1/2} \sum_{\substack{\ell \ll (MN)^{1/2} \\ (\ell,\Pi)=1}} \mathbf{1}_\mathcal{A}(\ell) \sum_{\substack{C < d \leq CP \\ d \mid \Pi}} \sum_{\substack{r \ll (MN)^{1/4} \\ P^+(r) \leq P}} \frac{\tau(r)^8 \rho_{\ell}([d,r])}{[d,r]} \sum_{\substack{s \ll (MN)^{1/4} \\ P^-(s) > P}} \frac{\tau(s)^8\rho_\ell(s)}{s}.
\end{align*}
We bound the sum over $s$ by working as in Lemma \ref{lem: sums of rho ell}. We have
\begin{align}\label{separate ell and coprime to ell}
&\sum_{\substack{s \ll (MN)^{1/4} \\ P^-(s) > P}} \frac{\tau(s)^8\rho_\ell(s)}{s} \leq \sum_{s \ll (MN)^{1/4}} \frac{\tau(s)^8\rho_\ell(s)}{s} \leq \sum_{d \mid \ell^\infty} \frac{\tau(d)^8 \rho_\ell(d)}{d} \sum_{\substack{t \ll (MN)^{1/4} \\ (t,\ell)=1}} \frac{\tau(t)^9}{t} \\
&\ll (\log MN)^{2^9} \prod_{p \mid \ell} \left(1 + \frac{2^{9}}{p^{1/2}} \right)\ll (\log MN)^{2^9}. \nonumber
\end{align}
By \eqref{separate ell and coprime to ell} and the change of variables $n = [d,r]$, we obtain
\begin{align*}
B_1' &\ll (\log MN)^{2^9} \sum_{\substack{\ell \ll (MN)^{1/2} \\ (\ell,\Pi)=1}} \mathbf{1}_\mathcal{A}(\ell) \sum_{\substack{n > C \\ P^+(n) \leq P}} \frac{\tau(n)^8 \tau_3(n) \rho_\ell(n)}{n}.
\end{align*}
Since $P^-(\ell) > P$ we see that $(n,\ell) = 1$, and therefore $\rho_\ell(n) \leq \tau(n)$. Set $\varepsilon = (\log P)^{-1}$. By Rankin's trick and the inequality $\tau_3(n) \leq \tau(n)^2$ we obtain
\begin{align*}
\sum_{\substack{n > C \\ P^+(n) \leq P}} \frac{\tau(n)^8 \tau_3(n) \rho_\ell(n)}{n} &\leq C^{-\varepsilon} \sum_{P^+(n) \leq P} \frac{\tau(n)^{11}}{n^{1-\varepsilon}} = C^{-\varepsilon} \prod_{p \leq P} \left(1 + \sum_{k=1}^\infty \frac{\tau(p^k)^{11}}{p^{k(1-\varepsilon)}} \right) \\
&\ll C^{-\varepsilon} \prod_{p \leq P} \left(1 + \frac{2^{11}}{p^{1-\varepsilon}}\right)\leq C^{-\varepsilon} \prod_{p \leq P} \left(1 + \frac{2^{14}}{p^{1+\varepsilon}}\right).
\end{align*}
The last inequality follows since $p^{2\varepsilon} \leq e^2 < 8$. We finish by observing that
\begin{align*}
C^{-\varepsilon} \prod_{p \leq P} \left(1 + \frac{2^{14}}{p^{1+\varepsilon}}\right) &\leq C^{-\varepsilon} \zeta(1+\varepsilon)^{2^{14}} \ll C^{-\varepsilon} \varepsilon^{-2^{14}} \leq (\log MN)^{2^{13}} P^{-1}.
\end{align*}
We deduce
\begin{align}\label{n1 bigger than C}
B_1' &\ll (\log MN)^{2^{14}} (MN)^{\frac{1}{2}+\frac{\gamma_0}{2}} P^{-1}.
\end{align}
By \eqref{n1 bigger than C} and our lower bound for $P$ the contribution from $n_1 > C$ is acceptable for Proposition \ref{prop for BMN}. It follows that
\begin{align}\label{intermediate BMN}
B_1(M,N) &\ll \theta^{5/2}(MN)^{\frac{1}{2}+\frac{\gamma_0}{2}} + \sum_{M < m \leq 2M} \sum_{\substack{n_1 \leq C \\ n_1 \mid \Pi}} \left|\sum_{\substack{N < n_0n_1 \leq (1+\theta)N \\ (n_0,\Pi)=1}} \mu(n_0) a(mn_0n_1) \right|.
\end{align}

We wish to make $mn_1$ into a single variable, but before we can do this we need to separate the variables $n_0$ and $n_1$. We achieve this with another short interval decomposition. We are reduced to studying
\begin{align*}
\sum_{\substack{G = (1+\theta^{5/2})^j \\ j \geq -1 \\ G \leq C}} \sum_{M < m \leq 2M} \sum_{\substack{G < n_1 \leq (1+\theta^{5/2})G \\ n_1 \mid \Pi}} \left|\sum_{\substack{N < n_0n_1 \leq (1+\theta)N \\ (n_0,\Pi)=1}} \mu(n_0) a(mn_0n_1) \right|.
\end{align*}
In the sum over $n_0$ we wish to replace the conditions $N < n_0n_1$ and $n_0n_1 \leq (1+\theta)N$ by the conditions $N < n_0G$ and $n_0 G \leq (1+\theta)N$, respectively. If $n_0n_1 > N$ but $n_0 G \leq N$, then
\begin{align*}
N < n_0n_1 \leq (1+\theta^{5/2})n_0 G \leq (1+\theta^{5/2}) N,
\end{align*}
and the error in replacing the condition $n_0n_1 > N$ by $n_0 G > N$ is
\begin{align*}
\ll (\log C)\theta^{-5/2} \sup_{G \leq C} \sum_{\substack{G < n_1 \leq (1+\theta^{5/2})G}} \mu^2(n_1) \ \ \sum_{(1+\theta^{5/2})^{-1}N/G < n_0 \leq (1+\theta^{5/2})N/G} \mu^2(n_0) \sum_{\substack{n \leq 3MN \\ n_0n_1 | n}} a(n).
\end{align*}
We write these three sums as
\begin{align*}
\sum_{n_1}\sum_{n_0} A_{n_0n_1}(3MN)= \sum_{n_1} \sum_{n_0}\left(M_{n_0n_1}(3MN) + R_{n_0n_1}(3MN) \right).
\end{align*}
To estimate the remainder term we change variables
\begin{align*}
\sum_{n_1} \sum_{n_0}|R_{n_0n_1}(3MN)| \leq \sum_d \tau(d) |R_d(3MN)|,
\end{align*}
then apply the divisor bound and Lemma \ref{Type I remainder}. We estimate the main term as we have before, and find that
\begin{align*}
\sum_{n_1}\sum_{n_0} A_{n_0n_1}(3MN) \ll (\log MN)^{O(1)}\theta^{5/2} (MN)^{\frac{1}{2}+\frac{\gamma_0}{2}}. 
\end{align*}
We similarly acquire the condition $n_0 G \leq (1+\theta) N$. We then change variables $mn_1 \rightarrow m, n_0 \rightarrow n$ to obtain
\begin{align}\label{BMN with small primes removed}
B_1(M,N) &\ll (\log MN)^{O(1)}\theta^{5/2} (MN)^{\frac{1}{2}+\frac{\gamma_0}{2}} \\ 
&+ (\log MN)\theta^{-5/2} \sup_{G \leq C} \sum_{MG < m \leq 2 (1+\theta^{5/2})MG} \tau(m) \left|\sum_{\substack{N/G < n_0 \leq (1+\theta)N/G \\ (n,\Pi)=1}} \mu(n) a(mn) \right|. \nonumber
\end{align}
In order to prove Proposition \ref{prop for BMN} it therefore suffices to show that
\begin{align}\label{bound for B prime MN}
B_2(M,N) &:= \sum_{M < m \leq 2M} \tau(m) \left|\sum_{\substack{N < n \leq (1+\theta)N \\ (n,\Pi)=1}} \mu(n) a(mn) \right| \ll \theta^5 (\log MN)^{O(1)} (MN)^{\frac{1}{2}+\frac{\gamma_0}{2}}
\end{align}
uniformly in
\begin{align}\label{restrictions on N after remove small primes}
x^{1/2 - \gamma_0/2 + \epsilon} \leq N \leq x^{25/77 - \epsilon}, \ \ \ \ \ \theta x \ll MN \ll x.
\end{align}
Note the slight (inconsequential) difference between \eqref{restrictions on N after remove small primes} and \eqref{restrictions on N}.

We have removed the small prime factors from $n$. This will aid us in making $m$ and $n$ coprime, which in turn will allow us to perform a factorization of our bilinear form over $\mathbb{Z}[i]$. Before estimating the contribution of those $m$ and $n$ which are not coprime, however, it is useful to remove the divisor function on $m$, as it is more difficult to deal with later. By the Cauchy-Schwarz inequality
\begin{align*}
B_2(M,N) &\leq B_3(M,N)^{1/2} B_2'(M,N)^{1/2},
\end{align*}
where
\begin{align*}
B_3(M,N) &:= \sum_{M < m \leq 2M} \left|\sum_{\substack{N < n \leq (1+\theta)N \\ (n,\Pi)=1}} \mu(n) a(mn) \right|, \\
B_2'(M,N) &:= \sum_{M < m \leq 2M} \tau(m)^2 \left|\sum_{\substack{N < n \leq (1+\theta)N \\ (n,\Pi)=1}} \mu(n) a(mn) \right|.
\end{align*}
We bound $B_2'(M,N)$ trivially.

\begin{lemma}\label{trivial bound for B '' MN}
We have $B_2'(M,N) \ll \theta (\log MN)^{2^{23}} (MN)^{\frac{1}{2}+\frac{\gamma_0}{2}}.$
\end{lemma}
\begin{proof}
We have the trivial bound
\begin{align*}
B_2'(M,N) &\ll \sum_{N < n \leq (1+\theta)N} \mu^2(n) \sum_{\substack{k \leq 3MN \\ d|k}} a(k) \tau(k)^2.
\end{align*}
We impose here a more severe condition on $U$, and therefore $N$, than \eqref{first bound for U}. We require
\begin{align}\label{second bound for U}
U \geq x^{1/2},
\end{align}
which implies $N \ll x^{1/2}$. Stricter conditions than \eqref{second bound for U} are imposed later, so there is no loss in imposing this condition now. We apply Lemma \ref{div reduc lemma} with $k = 2$ to arrive at
\begin{align*}
B_2'(M,N) &\ll \sum_{N < n \leq (1+\theta)N} \sum_{e \ll (MN)^{1/4}} \tau(e)^8\sum_{\substack{k \leq 3MN \\ [n,e]|k}} a(k) \\ 
&= \sum_{N < n \leq (1+\theta)N} \sum_{e \ll (MN)^{1/4}} \tau(e)^8 \left(M_{[n,e]}(3MN) + R_{[n,e]}(3MN) \right).
\end{align*}
The contribution from the remainder terms is
\begin{align*}
\ll \sum_{d \ll N(MN)^{1/4}} \left(\sum_{\substack{n_1,n_2 \\ [n_1,n_2]=d}} 1 \right) |R_d(3MN)| \ll (MN)^\epsilon \sum_{d \ll N(MN)^{1/4}}|R_d(3MN)|,
\end{align*}
and since $N x^{1/4} \leq x^{\gamma_0 - \epsilon}$ we may bound the remainder terms with Lemma \ref{Type I remainder}. 

We estimate the main term using the same types of arguments that gave \eqref{n1 bigger than C}. We factor $n = bd$ and $e = rs$ to bound the main term by
\begin{align*}
\ll (MN)^{1/2} \sum_{\substack{\ell \ll (MN)^{1/2} \\ (\ell,\Pi)=1}} \mathbf{1}_\mathcal{A}(\ell) &\sum_{\substack{b \leq (1+\theta)N \\ b \mid \ell^\infty}} \ \sum_{\substack{N/b < d \leq (1+\theta)N/b \\ (d,\ell)=1}}  \\ 
&\times \sum_{\substack{r \ll (MN)^{1/4} \\ r \mid \ell^\infty}} \frac{\tau(r)^8 \rho_\ell([b,r])}{[b,r]} \sum_{\substack{s \ll (MN)^{1/4} \\ (s,\ell)=1}} \frac{\tau(s)^8 \rho_\ell([d,s])}{[d,s]}.
\end{align*}
Since $(ds,\ell) = 1$ we have $\rho_\ell([d,s]) \leq \tau([d,s]) \leq \tau(ds) \leq \tau(d)\tau(s)$. We write
\begin{align*}
\frac{1}{[d,s]} = \frac{1}{ds} (d,s) \leq \frac{1}{ds} \sum_{\substack{f \mid d \\ f \mid s}} f,
\end{align*}
which yields that the main term is
\begin{align*}
&\ll (MN)^{1/2} \sum_{\substack{\ell \ll (MN)^{1/2} \\ (\ell,\Pi)=1}} \mathbf{1}_\mathcal{A}(\ell) \sum_{\substack{b \leq (1+\theta)N \\ b \mid \ell^\infty}} \sum_{\substack{r \ll (MN)^{1/4} \\ r \mid \ell^\infty}} \frac{\tau(r)^8 \rho_\ell([b,r])}{[b,r]} \\ 
&\ \ \ \ \ \ \ \ \ \ \ \ \ \ \times\sum_{\substack{N/b < d \leq (1+\theta)N/b}} \frac{\tau(d)}{d} \sum_{f \mid d} f \sum_{\substack{s \ll (MN)^{1/4} \\ f \mid s}} \frac{\tau(s)^9}{s} \\
&\ll (\log MN)^{2^9} (MN)^{1/2} \sum_{\substack{\ell \ll (MN)^{1/2} \\ (\ell,\Pi)=1}} \mathbf{1}_\mathcal{A}(\ell) \sum_{\substack{b \leq (1+\theta)N \\ b \mid \ell^\infty}} \sum_{\substack{r \ll (MN)^{1/4} \\ r \mid \ell^\infty}} \frac{\tau(r)^8 \rho_\ell([b,r])}{[b,r]} \sum_{\substack{N/b < d \leq (1+\theta)N/b}} \frac{\tau(d)^{11}}{d}.
\end{align*}
If $b \leq N^{1/2}$ we use Lemma \ref{div reduc lemma} with $k = 1$ to deduce
\begin{align*}
\sum_{\substack{N/b < d \leq (1+\theta)N/b}} \frac{\tau(d)^{11}}{d} &\ll \frac{b}{N} \sum_{k \ll (N/b)^{1/2}} \tau(k)^{22} \sum_{\substack{N/b < d \leq (1+\theta)N/b \\ k \mid d}} 1 \ll (\log N)^{2^{22}} \left(\theta + (b/N)^{-1/2} \right) \\
&\ll (\log N)^{2^{22}} \left( \theta + N^{-1/4} \right) \ll (\log N)^{2^{22}} \theta,
\end{align*}
the last inequality following from the lower bound \eqref{intermediate restrictions on U,V}. For $b > N^{1/2}$ we estimate the sum over $d$ trivially and change variables $n = [b,r]$ to get
\begin{align*}
\ll (\log MN)^{2^{12}} (MN)^{1/2} \sum_{\substack{\ell \ll (MN)^{1/2} \\ (\ell,\Pi)=1}} \mathbf{1}_\mathcal{A}(\ell) \sum_{\substack{n > N^{1/2} \\ n \mid \ell^\infty}} \frac{\tau(n)^8 \tau_3(n) \rho_\ell(n)}{n}.
\end{align*}
By Rankin's trick
\begin{align*}
\sum_{\substack{n > N^{1/2} \\ n \mid \ell^\infty}} \frac{\tau(n)^{8} \tau_3(n) \rho_\ell(n)}{n} &\ll N^{-1/4} \prod_{p \mid \ell} \left(1 + \sum_{k=1}^\infty \frac{2(k+1)^{10} p^{k/2}}{p^{3k/4}} \right) \ll N^{-1/4} \prod_{p \mid \ell} \left( 1 + \frac{2^{12}}{p^{1/4}} \right),
\end{align*}
Since $\ell$ has no small prime factors this last quantity is $\ll N^{-1/4}$. We deduce that
\begin{align*}
B_2'(M,N) &\ll (\log MN)^{2^{23}} \theta (MN)^{\frac{1}{2} + \frac{\gamma_0}{2}},
\end{align*}
as desired.
\end{proof}
By Lemma \ref{trivial bound for B '' MN} we see that in order to prove \eqref{bound for B prime MN} it suffices to show that
\begin{align}\label{bound for B0MN}
B_3(M,N) &\ll \theta^9 (\log MN)^{O(1)} (MN)^{\frac{1}{2}+\frac{\gamma_0}{2}}.
\end{align}

We are finally in a position where we can make our variables $m$ and $n$ coprime to one another. Since $n$ is only divisible by primes $p > P$, if $(m,n) \neq 1$ it follows that there exists a prime $p > P$ with $p \mid m$ and $p \mid n$. Therefore the contribution from those $m$ and $n$ that are not coprime is bounded by
\begin{align*}
B_3'(M,N) := \sum_{M < m \leq 2M} \left|\sum_{\substack{N < n \leq (1+\theta)N \\ (n,\Pi)=1 \\ (n,m) \neq 1}} \mu(n) a(mn) \right| &\ll \sum_{P < p \ll (MN)^{1/2}} \sum_{\substack{k \ll MN \\ p^2 \mid k}} a(k) \tau(k).
\end{align*}
We trivially estimate the contribution from $p > (MN)^{1/10}$ using the bound $a(k) \tau(k) \ll_\epsilon (MN)^\epsilon$. Thus
\begin{align*}
B_3'(M,N) &\ll (MN)^{9/10 + \epsilon} + \sum_{P < p \leq (MN)^{1/10}}\sum_{\substack{k \ll MN \\ p^2 \mid k}} a(k) \tau(k) \ll \sum_{P < p \leq (MN)^{1/10}} \sum_{d \ll (MN)^{1/2}} \sum_{\substack{k \ll MN \\ [d,p^2] \mid k}} a(k).
\end{align*}
Considering separately three cases ($d$ and $p$ are coprime, $p$ divides $d$ but $p^2$ does not, $p^2$ divides $d$), we find that
\begin{align*}
B_3'(M,N) &\ll \sum_{P < p \leq (MN)^{1/10}} \sum_{d \ll (MN)^{1/2}} \sum_{\substack{k \ll MN \\ dp^2 \mid k}} a(k).
\end{align*}
We apply Lemma \ref{Type I remainder} to deduce
\begin{align*}
B_3'(M,N) &\ll (MN)^{1/2}\sum_{\substack{\ell \ll (MN)^{1/2} \\ (\ell,\Pi)=1}} \mathbf{1}_\mathcal{A}(\ell) \sum_{P < p \leq (MN)^{1/10}} \frac{1}{p^2} \sum_{d \ll (MN)^{1/2}} \frac{\rho_\ell(dp^2)}{d} \\ 
&\ll (MN)^{1/2}\sum_{\substack{\ell \ll (MN)^{1/2} \\ (\ell,\Pi)=1}} \mathbf{1}_\mathcal{A}(\ell)\sum_{p > P} \frac{1}{p^2} \sum_{k=0}^\infty \sum_{\substack{d \ll (MN)^{1/2}/p^k \\ (d,p)=1}} \frac{\rho_\ell(dp^{k+2})}{dp^k} \\
&\ll (\log MN)^2(MN)^{1/2}\sum_{\substack{\ell \ll (MN)^{1/2} \\ (\ell,\Pi)=1}} \mathbf{1}_\mathcal{A}(\ell)\sum_{p > P} \frac{1}{p^2} \sum_{k=0}^\infty \frac{\rho_\ell(p^{k+2})}{p^k}.
\end{align*}
In going from the second line to the third line we have used Lemma \ref{lem: sums of rho ell} to bound the sum over $d$. 

We consider separately the cases $(p,\ell)=1$ and $p \mid \ell$:
\begin{align*}
\sum_{\substack{p > P \\ (p,\ell)=1}} \frac{1}{p^2} \sum_{k=0}^\infty \frac{\rho_\ell(p^{k+2})}{p^k} \ll \sum_{\substack{p > P}} \frac{1}{p^2} \ll P^{-1}
\end{align*}
and
\begin{align*}
\sum_{\substack{p > P \\ p \mid \ell}} \frac{1}{p^2} \sum_{k=0}^\infty \frac{\rho_\ell(p^{k+2})}{p^k} &\ll \sum_{\substack{p > P \\ p \mid \ell}} \frac{1}{p} \ll (\log \ell) P^{-1},
\end{align*}
where we have used Lemma \ref{lemma about rho ell d} to control the behavior of $\rho_\ell(p^{k+2})$. It follows that
\begin{align}\label{B0MN with non coprime removed}
B_3(M,N) &\ll (\log MN)^{O(1)} (MN)^{\frac{1}{2}+\frac{\gamma_0}{2}}P^{-1} + \sum_{M < m \leq 2M} \left|\sum_{\substack{N < n \leq (1+\theta)N \\ (n,m\Pi)=1}} \mu(n) a(mn) \right|.
\end{align}
In order to prove \eqref{bound for B0MN} it therefore suffices to show that
\begin{align}\label{bound for B1MN}
B_4(M,N) &:= \sum_{M < m \leq 2M} \left|\sum_{\substack{N < n \leq (1+\theta)N \\ (n,m\Pi)=1}} \mu(n) a(mn) \right| \ll \theta^9 (\log MN)^{O(1)} (MN)^{\frac{1}{2}+\frac{\gamma_0}{2}}
\end{align}
for $M$ and $N$ satisfying \eqref{restrictions on N after remove small primes}.

\section{Bilinear form in the sieve: transformations}\label{sec: bilinear form transformations}

Now that $m$ and $n$ are coprime we are able to enter the realm of the Gaussian integers. This is the key step that allows us to estimate successfully the bilinear form $B_4(M,N)$ (see the discussion in \cite[section 5]{friediw2} for more insight on this). Since $m$ and $n$ are coprime the unique factorization in $\mathbb{Z}[i]$ gives
\begin{align*}
a(mn) &= \frac{1}{4} \mathop{\sum \sum}_{\substack{|w|^2 = m, \ |z|^2 = n \\ (\text{Re}(z \overline{w}),\Pi) = 1}} \mathbf{1}_\mathcal{A}(\text{Re}(z \overline{w})).
\end{align*}
Since $(n,\Pi) = 1$ we have $(z \overline{z},\Pi) = 1$, so in particular $z$ is odd. Multiplying $w$ and $z$ by a unit we can rotate $z$ to a number satisfying
\begin{align*}
z \equiv 1 \pmod{2(1+i)}.
\end{align*}
Such a number is called primary, and is determined uniquely by its ideal. In rectangular coordinates $z = r+is$ being primary is equivalent to
\begin{align}\label{z primary}
r \equiv 1 \pmod{2}, \ \ \ \ \ \ \ \ \ s \equiv r-1 \pmod{4},
\end{align}
so that $r$ is odd and $s$ is even. We therefore obtain
\begin{align*}
B_4(M,N) &\leq \mathcal{B}_1(M,N) := \sum_{M < |w|^2 \leq 2M} \left|\sum_{\substack{N < |z|^2 \leq (1+\theta)N \\ (z \overline{z}, w \overline{w} \Pi)=1 \\ (\text{Re}(z \overline{w}),\Pi)=1}} \mu(|z|^2) \mathbf{1}_\mathcal{A}(\text{Re}(z \overline{w})) \right|.
\end{align*}
Here we assume that $z$ runs over primary numbers, so that the factor of $\frac{1}{4}$ does not occur. Further, the presence of the M\"obius function implies we may take $z$ to be primitive, that is, $z = r+is$ with $(r,s) = 1$. Henceforth a summation over Gaussian integers $z$ is always assumed to be over primary, primitive Gaussian integers.

The condition $(m,n) = 1$ was crucial in obtaining a factorization of our bilinear form over $\mathbb{Z}[i]$, but now this condition has become $(w \overline{w},z \overline{z})=1$ which is a nuisance since we wish for $w$ and $z$ to range independently of one another. Because $z \overline{z}$ has no small prime factors, it suffices to estimate trivially the complimentary sum in which $(w \overline{w},z \overline{z})\neq 1$. 

The arguments of this section bear some semblance to those in \cite[Section 5]{friediw2} and \cite[Section 20.4]{opera}. The plan of this section is as follows. We remove the condition $(w \overline{w},z \overline{z})= 1$ in order to make $w$ and $z$ more independent. With this condition gone we apply the Cauchy-Schwarz inequality to arrive at sums of the form
\begin{align*}
\sum_{w} \left|\sum_{\substack{z}} \mu(|z|^2) \mathbf{1}_\mathcal{A}(\text{Re}(z \overline{w})) \right|^2 = \sum_w \mathop{\sum \sum}_{z_1,z_2}\mu(|z_1|^2)\mu(|z_2|^2) \mathbf{1}_\mathcal{A}(\text{Re}(z_1 \overline{w})) \mathbf{1}_\mathcal{A}(\text{Re}(z_2 \overline{w})).
\end{align*}
For technical reasons it is convenient to impose the condition that $z_1$ and $z_2$ are coprime to each other. The key is again the fact that $|z_i|^2$ has no small prime factors. Once this is accomplished, we change variables to arrive at sums of the form
\begin{align*}
\mathop{\sum \sum}_{z_1,z_2} \mu(|z_1|^2) \mu(|z_2|^2) \mathop{\sum \sum}_{\ell_1,\ell_2} \mathbf{1}_\mathcal{A}(\ell_1) \mathbf{1}_\mathcal{A}(\ell_2),
\end{align*}
where $\ell_1,\ell_2$ are rational integers. The variable $w$ has disappeared, but now there are numerous conditions entangling $z_1,z_2$ and the $\ell_i$. Foremost among these conditions is a congruence to modulus $\Delta$, which is the imaginary part of $\overline{z_1} z_2$. The contribution from $\Delta = 0$ is easily dispatched, but the estimation of the terms with $\Delta \neq 0$ is much more involved and is handled in future sections.

Let $\mathcal{B}_1'(M,N)$ denote the contribution to $\mathcal{B}_1(M,N)$ from those $w$ and $z$ with $(w \overline{w},z \overline{z}) \neq 1$. We estimate $\mathcal{B}_1'(M,N)$ trivially and show that it is sufficiently small.

\begin{lemma}\label{lem: triv estimation of tilde cal B}
With the notation as above, we have
\begin{align*}
\mathcal{B}_1'(M,N) \ll (\log MN)^2 (MN)^{\frac{1}{2}+\frac{\gamma_0}{2}} P^{-1}.
\end{align*}
\end{lemma}
\begin{proof}
We find
\begin{align*}
\mathcal{B}_1'(M,N) &\ll \sum_{\ell \ll (MN)^{1/2}} \mathbf{1}_\mathcal{A}(\ell) \sum_{p > P} \ \mathop{\sum \sum}_{\substack{N < r^2 + s^2 \leq 2N \\ r^2+s^2 \equiv 0 (p) \\ (r,s)=1}} \ \mathop{\sum \sum}_{\substack{M < u^2 + v^2 \leq 2M \\ u^2 + v^2 \equiv 0 (p) \\ ru + sv = \ell}} 1.
\end{align*}
Observe that $p \nmid rs$ since $r^2+s^2 \equiv 0 \pmod{p}$ and $(r,s)=1$.

Given fixed $\ell,r$, and $s$, we claim that the residue class of $u$ is fixed modulo $ps/(\ell,p)$. Indeed, we see that $u$ is in a fixed residue class modulo $s$, since $ru+sv = \ell$ implies $u \equiv \overline{r} \ell \pmod{s}$. If $p \mid \ell$ this gives the claim, so assume $p \nmid \ell$. Then $v \equiv \overline{s}(\ell - ru) \pmod{p}$, which yields
\begin{align*}
0 \equiv u^2+v^2 \equiv u^2 + (\overline{s})^2 (\ell - ru)^2 \pmod{p}.
\end{align*}
We multiply both sides of the congruence by $s^2$, expand out $(\ell - ru)^2$, and use the fact that $r^2 + s^2 \equiv 0 (p)$. This gives
\begin{align*}
2\ell ru \equiv \ell^2 \pmod{p}.
\end{align*}
Since $\ell$ is coprime to $p$ we can divide both sides by $\ell$, and we can divide by $2r$ since $p \nmid 2r$. Thus the class of $u$ is fixed modulo $p$. Since the class of $u$ is fixed modulo $p$ and modulo $s$, and since $(p,s)=1$, the Chinese remainder theorem gives that the class of $u$ is fixed modulo $ps$. This completes the proof of the claim.

If $\ell,r,s$, and $u$ are given, then $v$ is determined. The sum over $u,v$ is then bounded by
\begin{align*}
\ll \frac{M^{1/2} (\ell,p)}{ps} + 1.
\end{align*}
By the symmetry of $u$ and $v$ we also have that the sum over $u,v$ is bounded by
\begin{align*}
\ll \frac{M^{1/2} (\ell,p)}{pr} + 1.
\end{align*}
Since $r^2 + s^2 > N$, either $r \gg N^{1/2}$ or $s \gg N^{1/2}$, so we may bound the sum over $u,v$ by
\begin{align*}
\ll \frac{M^{1/2}}{N^{1/2}} \frac{(\ell,p)}{p} + 1.
\end{align*}
We also note that
\begin{align*}
\mathop{\sum \sum}_{\substack{N < r^2 + s^2 \leq 2N \\ r^2+s^2 \equiv 0 (p) \\ (r,s)=1}} 1 \ll \sum_{\substack{n \ll N \\ p \mid n}} \tau(n) \ll \frac{N \log N}{p}.
\end{align*}
Therefore
\begin{align}\label{tilde BMN similar argum later}
\mathcal{B}_1'(M,N) &\ll (\log N)(MN)^{1/2} \sum_{\ell \ll (MN)^{1/2}} \mathbf{1}_\mathcal{A}(\ell) \sum_{P < p \ll N} \frac{(\ell,p)}{p^2}  + (\log N)^2 (MN)^{\gamma_0/2}N.
\end{align}
The second term is sufficiently small if $N \leq x^{1/2 - \epsilon}$, which is satisfied if
\begin{align}\label{third bound for U}
U \geq x^{1/2 + \epsilon}.
\end{align}
This lower bound supersedes \eqref{second bound for U}, and implies $M > N$ since $MN \gg \theta x$. 

To bound the first term we note that
\begin{align*}
\sum_{P < p \ll N} \frac{(\ell,p)}{p^2} \leq \sum_{\substack{p > P \\ p \nmid \ell}} \frac{1}{p^2} + \sum_{\substack{p > P \\ p \mid \ell}} \frac{1}{p} \ll (\log \ell) P^{-1},
\end{align*} 
and this gives the bound
\begin{align*}
\mathcal{B}_1'(M,N) &\ll (\log MN)^2 (MN)^{\frac{1}{2}+\frac{\gamma_0}{2}} P^{-1}.
\end{align*}
\end{proof}

Lemma \ref{lem: triv estimation of tilde cal B} proves that \eqref{bound for B1MN} follows from the bound
\begin{align*}
\mathcal{B}_2(M,N) &:= \sum_{M < |w|^2 \leq 2M} \left|\sum_{\substack{N < |z|^2 \leq (1+\theta)N \\ (z \overline{z},\Pi)=1 \\ (\text{Re}(z \overline{w}),\Pi)=1}} \mu(|z|^2) \mathbf{1}_\mathcal{A}(\text{Re}(z \overline{w})) \right| \ll \theta^9 (\log MN)^{O(1)}  (MN)^{1/2 + \gamma_0/2}.
\end{align*}

We now apply the Cauchy-Schwarz inequality, obtaining
\begin{align*}
\mathcal{B}_2(M,N)^2 &\ll M \mathcal{D}_1(M,N),
\end{align*}
where
\begin{align*}
\mathcal{D}_1(M,N) &:=  \sum_{|w|^2 \leq 2M} \left|\sum_{\substack{N < |z|^2 \leq (1+\theta)N \\ (z \overline{z},\Pi)=1 \\ (\text{Re}(z \overline{w}),\Pi)=1}} \mu(|z|^2) \mathbf{1}_\mathcal{A}(\text{Re}(z \overline{w})) \right|^2.
\end{align*}
Note that we have used positivity to extend the sum over $w$. It therefore suffices to show that
\begin{align}\label{bound for cal DMN}
\mathcal{D}_1(M,N) &\ll \theta^{18} (\log MN)^{O(1)} (MN)^{\gamma_0} N.
\end{align}

Expanding the square in $\mathcal{D}_2(M,N)$ gives a sum over $w,z_1$, and $z_2$, say. As mentioned above, we wish to impose the condition that $z_1$ and $z_2$ are coprime. To do so we first require a trivial bound. Observe that
\begin{align*}
\mathcal{D}_1(M,N) &\leq D_1'(M,N) :=\sum_{|w|^2 \leq 2M} \mathop{\sum \sum}_{N < |z_1|^2,|z_2|^2 \leq 2N} \mathbf{1}_\mathcal{A}(\text{Re}(z_1 \overline{w})) \mathbf{1}_\mathcal{A}(\text{Re}(z_2 \overline{w})).
\end{align*}

\begin{lemma}\label{trivial bound on DMN}
For $M \geq N \geq 2$ we have
\begin{align*}
D_1'(M,N) &\ll \left( (MN)^{\frac{1}{2}+\frac{\gamma_0}{2}} + (MN)^{\gamma_0} N  \right) (\log MN)^{36}.
\end{align*}
\end{lemma}
\begin{proof}
We consider separately the diagonal $|z_1| = |z_2|$ and the off-diagonal $|z_1| \neq |z_2|$ cases.

We can bound the diagonal terms by
\begin{align*}
D_{=}'(M,N) &:= \sum_{\ell \ll (MN)^{1/2}} \mathbf{1}_\mathcal{A}(\ell) \mathop{\sum \sum}_{\substack{N < r^2 + s^2 \leq 2N \\ (r,s)=1}} \tau(r^2+s^2) \mathop{\sum \sum}_{\substack{u^2 + v^2 \leq 2M \\ ru + sv = \ell}} 1.
\end{align*}
By an argument similar to that which yielded \eqref{tilde BMN similar argum later}, we bound the sum over $u,v$ by
\begin{align*}
\ll \min \left(\frac{M^{1/2}}{r} + 1, \frac{M^{1/2}}{s} + 1 \right) \ll \frac{M^{1/2}}{N^{1/2}} + 1 \ll \frac{M^{1/2}}{N^{1/2}}.
\end{align*}
The sum over $r$ and $s$ is bounded by
\begin{align*}
\mathop{\sum \sum}_{\substack{N < r^2 + s^2 \leq 2N \\ (r,s)=1}} \tau(r^2+s^2) &\ll \sum_{n \leq 2N} \tau(n)^2 \ll N (\log N)^3,
\end{align*}
and we deduce that
\begin{align*}
D_{=}'(M,N) \ll (\log N)^3 (MN)^{\frac{1}{2}+\frac{\gamma_0}{2}}.
\end{align*}

We turn now to bounding the off-diagonal terms with $|z_1| \neq |z_2|$. Observe that
\begin{align*}
\Delta = \Delta(z_1,z_2) = \frac{1}{2i} \left(\overline{z_1}z_2 - z_1 \overline{z_2} \right) \neq 0,
\end{align*}
since $(z_1,\overline{z_1})=(z_2,\overline{z_2})=1$. The off-diagonal terms therefore contribute
\begin{align*}
D_{\neq}'(M,N) &\ll \mathop{\sum \sum}_{\substack{\ell_1,\ell_2 \ll (MN)^{1/2}}} \mathbf{1}_\mathcal{A}(\ell_1) \mathbf{1}_\mathcal{A}(\ell_2) \mathop{\sum\sum}_{\substack{N < |z_1|^2,|z_2|^2 \leq 2N \\ |z_1| \neq |z_2| \\ \Delta \mid (\ell_1 z_2 - \ell_2z_1)}} 1.
\end{align*}
We note that the division takes place in the Gaussian integers, and that $\ell_1 z_2 - \ell_2 z_1 \neq 0$ (see \eqref{relation btwn w delta ells and zs} below). Using rectangular coordinates $z_1 = r_1 + is_1, z_2 = r_2 + is_2$, we see that $\Delta = r_1 s_2 - r_2 s_1$ and
\begin{align*}
\ell_1 r_2 &\equiv \ell_2 r_1 \pmod{\Delta}, \\
\ell_1 s_2 &\equiv \ell_2 s_1 \pmod{\Delta},
\end{align*}
where now the congruences are congruences of rational integers. By symmetry we may assume that $\ell_1s_2 - \ell_2s_1 \neq 0$. Given $\ell_1,\ell_2,s_1,s_2$, and $\Delta \neq 0$, we see that the residue class of $r_1$ modulo $s_1/(s_1,s_2)$ is fixed, and then $r_2$ is determined by the relation $\Delta = r_1s_2 - r_2s_1$. The number of pairs $r_1,r_2$ is then bounded by
\begin{align*}
\ll \sqrt{N} \frac{(s_1,s_2)}{s_1}.
\end{align*}
Letting $\delta = (s_1,s_2)$ and $s_1 = \delta s_1^*, s_2 = \delta s_2^*$ (so that $(s_1^*,s_2^*)=1$), we see that
\begin{align*}
D_{\neq}'(M,N) &\ll \sqrt{N} \sum_{\delta \ll N^{1/2}} \tau(\delta) \mathop{\sum \sum}_{\substack{s_1^*,s_2^* \ll N^{1/2}/\delta \\ (s_1^*,s_2^*)=1}} \frac{1}{s_1^*} \mathop{\sum \sum}_{\substack{\ell_1,\ell_2 \ll (MN)^{1/2} \\ \ell_1s_2^* - \ell_2s_1^* \neq 0}} \mathbf{1}_\mathcal{A}(\ell_1)\mathbf{1}_\mathcal{A}(\ell_2) \tau(\ell_1s_2^* - \ell_2s_1^*).
\end{align*}
Observe that $|\ell_1 s_2^* - \ell_2s_1^*| \ll \sqrt{M} N$. We apply Lemma \ref{div reduc lemma} with $k = 2$ to get
\begin{align*}
D_{\neq}'(M,N) &\ll \sqrt{N} \sum_{\delta \ll N^{1/2}} \tau(\delta) \mathop{\sum \sum}_{\substack{s_1^*,s_2^* \ll N^{1/2}/\delta \\ (s_1^*,s_2^*)=1}} \frac{1}{s_1^*} \sum_{f \ll F} \tau(f)^4 \mathop{\sum \sum}_{\substack{\ell_1,\ell_2 \ll (MN)^{1/2} \\ \ell_1s_2^* \equiv \ell_2s_1^* (f)}} \mathbf{1}_\mathcal{A}(\ell_1)\mathbf{1}_\mathcal{A}(\ell_2),
\end{align*}
where $F = (\sqrt{M} N)^{1/4}$. Taking the supremum over $s_2^*$ and $\delta$ gives
\begin{align*}
D_{\neq}'(M,N) &\ll (\log N)^2 N \sum_{\substack{s_1^* \leq N' \\ (s_1^*,s_2^*)=1}} \frac{1}{s_1^*} \sum_{f \ll F} \tau(f)^4 \mathop{\sum \sum}_{\substack{\ell_1,\ell_2 \ll (MN)^{1/2} \\ \ell_1s_2^* \equiv \ell_2s_1^* (f)}} \mathbf{1}_\mathcal{A}(\ell_1)\mathbf{1}_\mathcal{A}(\ell_2)
\end{align*}
for some $N',s_2^* \ll N^{1/2}$. We now write $f = gh, s_1^* = hs$ with $(g,s) = 1$. Observe that $(h,s_2^*) = 1$. Then the congruence $\ell_1s_2^* \equiv \ell_2s_1^* (f)$ yields the congruences
\begin{align*}
\ell_1 &\equiv 0 \pmod{h}, \\
\ell_2 &\equiv \overline{s} s_2^* (\ell_1/h) \pmod{g},
\end{align*}
where $\overline{s}$ is the inverse of $s$ modulo $g$. We deduce that
\begin{align*}
D_{\neq}'(M,N) &\ll (\log N)^2 N \mathop{\sum \sum}_{gh \leq F} \frac{\tau(g)^4 \tau(h)^4}{h} \sum_{s \leq N'/h} \frac{1}{s} \sum_{\substack{\ell_1 < X \\ h \mid \ell_1}} \mathbf{1}_\mathcal{A}(\ell_1) \sum_{\substack{\ell_2 < X \\ \ell_2 \equiv \nu (g)}} \mathbf{1}_\mathcal{A}(\ell_2),
\end{align*}
where $X$ is a power of 10 with $X \asymp (MN)^{1/2}$ and $\nu = \nu(h,\ell_1,s,s_2^*)$ is a residue class. We detect the congruence on $\ell_2$ with additive characters, and then apply the triangle inequality to eliminate $\nu$ (we have already seen this technique in the proof of Lemma \ref{prototypical estimation lemma one}). We then drop the divisibility condition on $\ell_1$, obtaining
\begin{align*}
D_{\neq}'(M,N) &\ll (\log MN)^{19} (MN)^{\gamma_0} N \sum_{g \leq F} \frac{\tau(g)^4}{g} \sum_{r=1}^g F_X \left(\frac{r}{g} \right).
\end{align*}
Reducing to primitive fractions gives
\begin{align*}
\sum_{g \leq F} \frac{\tau(g)^4}{g} \sum_{r=1}^g F_X \left(\frac{r}{g} \right)&\ll (\log MN)^{16} \sum_{1 \leq q \leq F} \frac{\tau(q)^4}{q} \sum_{\substack{1 \leq a \leq q \\ (a,q)=1}} F_X \left(\frac{a}{q}\right).
\end{align*}
By the divisor bound, a dyadic division, and Lemma \ref{maynard type I}, we find this last quantity is
\begin{align*}
\ll (\log MN)^{17} \sup_{Q \leq F} \left(\frac{1}{Q^{23/77 - \epsilon}} + \frac{Q^{1+\epsilon}}{X^{50/77}} \right) \ll (\log MN)^{17} \left(1 + \frac{F^{1+\epsilon}}{X^{50/77}} \right).
\end{align*}
Observe that
\begin{align*}
F &= (\sqrt{M}N)^{1/4} \leq (\sqrt{MN})^{1/2} \ll (\sqrt{MN})^{50/77-\epsilon} \asymp X^{50/77-\epsilon},
\end{align*}
which yields
\begin{align*}
D_{\neq}'(M,N) &\ll (\log MN)^{36} (MN)^{\gamma_0} N.
\end{align*}
\end{proof}

With Lemma \ref{trivial bound on DMN} in hand, we can show that the contribution from $(z_1,z_2) \neq 1$ in $\mathcal{D}_1(M,N)$ is negligible. This is due to the fact that $(|z_i|^2,\Pi) = 1$. Denoting by $\pi$ a Gaussian prime, the contribution from $(z_1,z_2) \neq 1$ is bounded by
\begin{align*}
\sum_{P < |\pi|^2 \ll N} \sum_{|w|^2 \ll M} \mathop{\sum \sum }_{\frac{N}{|\pi|^2} < |z_1|^2,|z_2|^2 \leq \frac{2N}{|\pi|^2}} \mathbf{1}_\mathcal{A}(\text{Re}(z_1 \pi \overline{w})) \mathbf{1}_\mathcal{A}(\text{Re}(z_2 \pi \overline{w})).
\end{align*}
We break the range of $|\pi|^2$ into dyadic intervals $P_1 < |\pi|^2 \leq 2P_1$, and put $w' = w \overline{\pi}$. We observe that
\begin{align*}
\sum_{\pi \mid \mathfrak{z}} 1 \ll \log |\mathfrak{z}|
\end{align*}
for any Gaussian integer $\mathfrak{z}$, so the contribution from the pairs $z_1,z_2$ that are not coprime is bounded by
\begin{align*}
\ll (\log MN)^2 D_1'(MP_1,NP_1^{-1}),
\end{align*}
for some $P < P_1 \ll N$. By Lemma \ref{trivial bound on DMN} this bound becomes
\begin{align*}
\ll (\log MN)^{38} \left((MN)^{1/2 + \gamma_0/2} + (MN)^{\gamma_0} N P^{-1} \right).
\end{align*}
The second term is satisfactorily small, and the first is sufficiently small provided
\begin{align}\label{lower bound for V}
V \geq x^{1/2 - \gamma_0/2 + \epsilon}.
\end{align}
This lower bound for $V$ supersedes the one in \eqref{intermediate restrictions on U,V}. In order to show \eqref{bound for cal DMN} it then suffices to show that
\begin{align}\label{bound for cal D star MN}
\mathcal{D}_2(M,N) &:= \sum_{|w|^2 \leq 2M} \mathop{\sum \sum}_{\substack{N < |z_1|^2,|z_2|^2 \leq (1+\theta)N \\ (|z_1|^2|z_2|^2,\Pi)=1 \\ (z_1,z_2)=1 \\ (\text{Re}(z_1 \overline{w}) \text{Re}(z_2 \overline{w}),\Pi) = 1}} \mu(|z_1|^2) \mu(|z_2|^2) \mathbf{1}_\mathcal{A}(\text{Re}(z_1 \overline{w})) \mathbf{1}_\mathcal{A}(\text{Re}(z_2 \overline{w})) \\ 
&\ll \theta^{18}(\log MN)^{O(1)}  (MN)^{\gamma_0} N \nonumber
\end{align}
for $M$ and $N$ satisfying \eqref{restrictions on N after remove small primes}. Since $Vx^{-o(1)} \ll N \ll x/U$ we see that \eqref{lower bound for V} yields the lower bound on $N$ in \eqref{restrictions on N after remove small primes}.

Now that $z_1$ and $z_2$ are coprime we change variables in order to rid ourselves of the variable $w$. We put $\ell_1 = \text{Re}(z_1 \overline{w})$ and $\ell_2 = \text{Re}(z_2 \overline{w})$, that is,
\begin{align*}
&z_1 \overline{w} + \overline{z_1}w= 2\ell_1, \\
&z_2 \overline{w} + \overline{z_2}w= 2\ell_2.
\end{align*}
We set $\Delta = \Delta(z_1,z_2) = \text{Im}(\overline{z_1}z_2) = \frac{1}{2i}(\overline{z_1}z_2 - z_1 \overline{z_2})$, and note that
\begin{align}\label{relation btwn w delta ells and zs}
i w \Delta = \ell_1 z_2 - \ell_2 z_1.
\end{align}
It follows that
\begin{align*}
\mathcal{D}_2 (M,N) &= \mathop{\sum \sum}_{\substack{N < |z_1|^2,|z_2|^2 \leq (1+\theta)N \\ (|z_1|^2 |z_2|^2,\Pi)=1 \\ (z_1,z_2)=1}} \mu(|z_1|^2) \mu(|z_2|^2) \mathop{\sum \sum}_{\substack{\ell_1,\ell_2 \leq \sqrt{2(1+\theta)} (MN)^{1/2} \\ \ell_1 z_2 \equiv \ell_2 z_1 (|\Delta|) \\ |\ell_1 z_2 - \ell_2 z_1|^2 \leq 2 \Delta^2 M \\ (\ell_1\ell_2,\Pi)=1}}  \mathbf{1}_\mathcal{A}(\ell_1)\mathbf{1}_\mathcal{A}(\ell_2).
\end{align*}
Observe that the congruence is a congruence of Gaussian integers.

The contribution from $\Delta = 0$ is bounded by
\begin{align*}
\mathcal{D}_2' &:= \mathop{\sum \sum}_{\substack{|z_1|^2, |z_2|^2\ll N \\ \text{Im}(\overline{z_1}z_2)=0}} \sum_{\ell \ll (MN)^{1/2}} \mathbf{1}_\mathcal{A}(\ell),
\end{align*}
since if $\Delta = 0$ the triple $(z_1,z_2,\ell_1)$ determines $\ell_2$. The summation over $\ell$ is bounded by $O((MN)^{\gamma_0/2})$. Writing $z_1 = r+is$ and $z_2 = u+iv$, we may bound the sum over $z_1,z_2$ by
\begin{align*}
\mathop{\sum \sum}_{r,v \ll N^{1/2}} \mathop{\sum \sum}_{\substack{s,u \ll N^{1/2} \\ su = rv}} 1 \leq \mathop{\sum \sum}_{r,v \ll N^{1/2}} \tau(rv) \ll N \log^2 N.
\end{align*}
Thus
\begin{align*}
\mathcal{D}_2' &\ll (\log N)^{2} (MN)^{\gamma_0/2} N,
\end{align*}
and this is acceptable for \eqref{bound for cal D star MN} since $MN \gg \theta x$. It therefore suffices to show that
\begin{align}\label{bound for cal D 1 MN}
\mathcal{D}_3(M,N) \ll \theta^{18}(\log MN)^{O(1)}  (MN)^{\gamma_0} N,
\end{align}
where $\mathcal{D}_3$ is $\mathcal{D}_2$ with the additional condition that $\Delta \neq 0$.

\section{Congruence exercises}\label{sec: congruence exercises}

Our next major task, which requires much preparatory work, is to simplify $\mathcal{D}_3$ by removing the congruence condition entangling $z_1,z_2,\ell_1$, and $\ell_2$. To handle the condition $\ell_1 z_2 \equiv \ell_2 z_1 \pmod{|\Delta|}$, we sum over all residue classes $b$ modulo $|\Delta|$ such that $bz_1 \equiv z_2 \pmod{|\Delta|}$. Then, with $\ell_1$ fixed, we sum over $\ell_2 \equiv b \ell_1 \pmod{|\Delta|}$. 

A key point is that since $z_1$ and $z_2$ are coprime, $b$ is uniquely determined modulo $|\Delta|$. That is,
\begin{align*}
\sum_{\substack{b (|\Delta|) \\ bz_1 \equiv z_2 (|\Delta|)}} 1 = 1.
\end{align*}
To see this, note that we only require $(z_1,|\Delta|) = 1$. Now, if $\pi$ is a Gaussian prime dividing $z_1$ and $|\Delta|$, then the congruence condition on $b$ implies $\pi \mid z_2$, which contradicts the fact that $z_1$ and $z_2$ are coprime Gaussian integers.

One problem we face is that the congruence $\ell_2 \equiv b \ell_1 \pmod{|\Delta|}$ is not a congruence of rational integers. If we write $b = r+is$, then we see that the Gaussian congruence $\ell_2 \equiv b \ell_1 \pmod{|\Delta|}$ is equivalent to the rational congruences
\begin{align*}
\ell_2 &\equiv r \pmod{|\Delta|}, \\
s \ell_1 &\equiv 0 \pmod{|\Delta|}.
\end{align*}
If we can take $\ell_1$ to be coprime to $|\Delta|$, then this implies $s \equiv 0 \pmod{\Delta}$. As $s$ is only defined modulo $\Delta$ we may then take $s$ to be zero, which implies $b$ is rational.

Lastly, with a view towards using the fundamental lemma to control the condition $(\ell_2,\Pi) = 1$, we anticipate sums of the form
\begin{align*}
\sum_{\substack{\ell_2 \equiv b\ell_1 \pmod{|\Delta|} \\ \ell_2 \equiv 0 \pmod{d}}} \mathbf{1}_\mathcal{A}(\ell).
\end{align*}
If we can ensure that $b$ is coprime to $|\Delta|$, then the first congruence implies $\ell_2$ is coprime to $|\Delta|$ (recall we are assuming for the moment that $(\ell_1,|\Delta|) = 1$). Taking the first and second congruences together we see that $(d,|\Delta|) = 1$, so that the set of congruences may be combined by the Chinese remainder theorem into a single congruence modulo $d|\Delta|$. We can take $b$ to be coprime to $|\Delta|$ by imposing the condition $(z_1z_2,|\Delta|) = 1$. Actually, we saw above that $z_1$ is already coprime to $|\Delta|$, so we only need to make $z_2$ coprime to $|\Delta|$.

One technical obstacle to overcome is that the set $\mathcal{A}$ is not well-distributed in residue classes to moduli that are not coprime to 10. Since we have essentially no control over the 2- or 5-adic valuation of $|\Delta|$, we need to work around the ``10-adic'' part of $|\Delta|$ somehow.

We begin by removing those $|\Delta|$ that are unusually small (see \cite[(17)]{hbli} and the following discussion for a similar computation). Since $\Delta = \text{Im}(\overline{z_1}z_2)$ and $|z_i| \asymp N^{1/2}$, we expect that typically $|\Delta| \approx N$, and perhaps that those $|\Delta|$ that are much smaller than $N$ should have a negligible contribution. 
\begin{lemma}\label{lem: remove small Delta}
The contribution to $\mathcal{D}_3$ from $|\Delta|\leq \theta^{18}N$ is
\begin{align*}
\ll \theta^{18}(\log N)^{2}  (MN)^{\gamma_0} N.
\end{align*}
\end{lemma}
\begin{proof}
We estimate trivially the contribution from $|\Delta| \leq \theta^{18} N$. By the triangle inequality, this contribution is bounded by
\begin{align*}
D_3'(M,N) &:= \mathop{\sum\sum}_{\substack{N < |z_1|^2,|z_2|^2 \leq (1+\theta)N \\ (z_1,z_2)=1 \\ 0 < |\Delta| \leq \theta^{18} N}} \sum_{\substack{b(|\Delta|) \\ bz_1 \equiv z_2 (|\Delta|)}} \sum_{\ell_1 \ll (MN)^{1/2}} \mathbf{1}_\mathcal{A}(\ell_1) \sum_{\substack{\ell_2 \equiv \text{Re}(b) \ell_1 (|\Delta|) \\ \eqref{star condition}}} \mathbf{1}_\mathcal{A}(\ell_2),
\end{align*}
where \eqref{star condition} denotes the condition
\begin{align}\label{star condition}
\left|\ell_2 - \ell_1 \frac{z_2}{z_1} \right| \ll \theta^{18} (MN)^{1/2}.
\end{align}
Observe that $\eqref{star condition}$ forces $\ell_2$ to lie in an interval $I = I(\ell_1,z_2,z_2)$ of length $\leq c \theta^{18} (MN)^{1/2}$, for some positive constant $c$.

We use the ``intervals of length a power of ten'' technique we deployed in analyzing \eqref{mxuv mt, removed f ell, before fund} (see \eqref{fund lem short interval}). Let $Y$ be the largest power of 10 satisfying $Y \leq \theta^{18} (MN)^{1/2}$, and cover the interval $I$ with subintervals of the form $[nY,(n+1)Y)$, where $n$ is a nonnegative integer (observe that we require only $O(1)$ subintervals to cover $I$). Recalling from the proof of Lemma \ref{prototypical estimation lemma one} that we argue slightly differently depending on whether $a_0$ is zero or not, we have
\begin{align*}
\sum_{\substack{\ell_2 \equiv \text{Re}(b) \ell_1 (|\Delta|) \\ \eqref{star condition}}} \mathbf{1}_\mathcal{A}(\ell_2) &\leq \sum_{n \in S(I)} \mathbf{1}_\mathcal{A}(n) \sum_{\substack{\delta(a_0)Y/10 \leq t < Y \\  t+nY \equiv \text{Re}(b) \ell_1 (|\Delta|)}} \mathbf{1}_\mathcal{A}(t),
\end{align*}
where $S(I)$ is some set of integers depending on $I$. We detect the congruence condition via additive characters, and separate the zero frequency from the nonzero frequencies. On the nonzero frequencies we apply inclusion-exclusion and then the triangle inequality so that $t$ runs over an interval of the form $t < Y/10$ or $t < Y$. This application of the triangle inequality also removes the dependence on $n,b$, and $\ell_1$. It follows that
\begin{align*}
\sum_{\substack{\ell_2 \equiv \text{Re}(b) \ell_1 (|\Delta|) \\ \eqref{star condition}}} \mathbf{1}_\mathcal{A}(\ell_2) &\ll \frac{1}{|\Delta|} (\theta^{18} \sqrt{MN})^{\gamma_0} + \frac{1}{|\Delta|} (\theta^{18} \sqrt{MN})^{\gamma_0} \sum_{r=1}^{|\Delta| - 1} F_X \left(\frac{r}{|\Delta|} \right),
\end{align*}
where $X$ is a power of 10 with $X \asymp Y$.

The contribution $D_{3,0}'(M,N)$ to $D_3'(M,N)$ coming from the first term here is
\begin{align*}
D_{3,0}'(M,N) &\ll \theta^{18\gamma_0} (MN)^{\gamma_0} \sum_{0 < |\Delta| \leq \theta^K N} \frac{1}{|\Delta|} \sum_{N < |z_1|^2 \leq (1+\theta)N} \sum_{\substack{N < |z_2|^2\leq (1+\theta)N \\ \text{Im}(\overline{z_1}z_2) = \Delta}} 1.
\end{align*}
Let $z_1 = r+is$ with $(r,s) = 1$. Since $r^2 + s^2 > N$ this implies $rs \neq 0$. Let $z_2 = u+iv$, and note that $\text{Im}(\overline{z_1}z_2) = rv - su$. Let $(u_0,v_0)$ be a pair such that $rv_0 - su_0 = \Delta$. Then for any other pair $(u_1,v_1)$ such that $rv_1 - su_1 = \Delta$, we have
\begin{align*}
r(v_1 - v_0) - s(u_1-u_0) = 0.
\end{align*}
Since $r$ and $s$ are coprime, we see that $v_1 - v_0 = ks$ for some integer $k$, and $u_1 -u_0 = \ell r$ for some integer $\ell$. As $rs\neq 0$ we find that $k = \ell$, and thus $u_1 + iv_1 = u_0 + iv_0 + kz_1$. Since $|z_1| \asymp |z_2| \asymp N^{1/2}$, it follows that the number of choices for $z_2$, given $\Delta$ and $z_1$, is $O(1)$, and therefore
\begin{align*}
D_{3,0}'(M,N) &\ll \theta^{18\gamma_0 + 1} (\log N) (MN)^{\gamma_0} N \ll \theta^{18}(\log N) (MN)^{\gamma_0} N.
\end{align*}

We now turn to bounding the contribution of the nonzero frequencies $D_{3,*}'(M,N)$. Arguing as with $D_{3,0}'(M,N)$, we deduce that
\begin{align*}
D_{3,*}'(M,N) &\ll \theta^{18\gamma_0 + 1} (MN)^{\gamma_0} N \sum_{d \leq \theta^{18} N} \frac{1}{d} \sum_{r=1}^{d-1} F_X \left(\frac{r}{d} \right).
\end{align*}
We reduce to primitive fractions and perform dyadic decompositions to obtain
\begin{align*}
D_{3,*}'(M,N) &\ll \theta^{18\gamma_0 + 1} (\log N)^2 (MN)^{\gamma_0} N \sup_{Q \ll \theta^{18} N} \frac{1}{Q} \sum_{q \ll Q} \sum_{(b,q)=1} F_X \left(\frac{b}{q} \right).
\end{align*}
By Lemma \ref{maynard type I},
\begin{align}\label{eq: nonzero freqs small Delta}
\frac{1}{Q}\sum_q \sum_{(b,q)=1} F_X \left( \frac{b}{q}\right) \ll \frac{1}{Q^{23/77}} + \frac{Q}{X^{50/77}} \ll 1 + \frac{\theta^{18} N}{X^{50/77}}.
\end{align}
We wish for the quantity in \eqref{eq: nonzero freqs small Delta} to be $\ll 1$, so it suffices to have $N \ll x^{25/77 - \epsilon}$, and this in turn requires
\begin{align}\label{last bound for U}
U \geq x^{52/77 + \epsilon}.
\end{align}
The constraint \eqref{last bound for U} replaces \eqref{third bound for U}, and is the last lower bound condition we need to put on $U$. We deduce that the total contribution from $|\Delta| \leq \theta^{18}N$ is
\begin{align*}
\ll D_3'(M,N) &\ll\theta^{18} (\log N)^{2}  (MN)^{\gamma_0} N,
\end{align*}
as desired.
\end{proof}

We make now a brief detour to discuss our restrictions on $N, U$, and $V$. With the upper bound on $UV$ from \eqref{intermediate restrictions on U,V}, our lower bound for $V$ \eqref{lower bound for V}, and our lower bound for $U$ \eqref{last bound for U} in hand, there are no more conditions to put on $U$ or $V$, and the range of $N$ in \eqref{restrictions on N after remove small primes} is now clear. For these constraints to be consistent with one another it suffices to have
\begin{align}\label{admissible choices for U,V}
&U = x^\alpha , \ \ \ \ \ \ \  \  \ V = x^\beta, \nonumber \\
&\frac{52}{77} = 0.675\ldots < \alpha < \gamma_0 - \left(\frac{1}{2} - \frac{\gamma_0}{2} \right) = 0.931\ldots, \\
&\frac{1}{2} - \frac{\gamma_0}{2} = 0.0228\ldots < \beta < \gamma_0 - \alpha.\nonumber
\end{align}
Note that \eqref{admissible choices for U,V} is consistent with the specific choice we made in \eqref{sample choices of U,V}.

Let us return to estimations. With Lemma \ref{lem: remove small Delta} we have removed those moduli $|\Delta|$ that are substantially smaller than expected, and we now proceed with our task of making $b$ a rational residue class. We saw above that it suffices to impose the condition $(\ell_1,|\Delta|) = 1$. We expect to be able to impose this condition with the cost of only a small error since $(\ell_1,\Pi) = 1$. Indeed, it is for this step alone that we introduced the condition $(\ell,\Pi) = 1$ at the beginning of the proof of Theorem \ref{main theorem} (see the comments after \eqref{eq: naive sum}).

We estimate trivially the contribution from $(\ell_1,|\Delta|) \neq 1$. By the triangle inequality, it suffices to estimate
\begin{align*}
D_3''(M,N) &:= \sum_{p > P} \mathop{\sum \sum}_{\substack{N < |z_1|^2,|z_2|^2 \leq (1+\theta)N \\ (z_1,z_2)=1 \\ p \mid |\Delta| }} \sum_{bz_1 \equiv z_2 (|\Delta|)} \sum_{\substack{\ell_1 < X \\ p \mid \ell_1}} \mathbf{1}_\mathcal{A}(\ell_1) \sum_{\substack{\ell_2 < X \\ \ell_2 \equiv \text{Re}(b) \ell_1 (|\Delta|)}} \mathbf{1}_\mathcal{A}(\ell_2),
\end{align*}
where $X$ is a power of 10 with $X \asymp (MN)^{1/2}$. As has become typical, we introduce characters to detect the congruence on $\ell_2$ and then apply the triangle inequality to eliminate the dependence on $b,\ell_1$. We also apply additive characters to detect the congruence on $\ell_1$, obtaining
\begin{align*}
D_3''(M,N) &\ll (MN)^{\gamma_0} \sum_{P < p \ll N} \mathop{\sum \sum}_{\substack{N < |z_1|^2,|z_2|^2 \leq (1+\theta)N \\ (z_1,z_2)=1 \\ p \mid |\Delta|}} \frac{1}{p|\Delta|} \sum_{k=1}^p F_X \left(\frac{k}{p} \right) \sum_{r=1}^{|\Delta|} F_X \left( \frac{r}{|\Delta|} \right).
\end{align*}
By Lemma \ref{maynard single denominator} we find
\begin{align*}
\frac{1}{p}\sum_{k=1}^p F_X \left(\frac{k}{p} \right) \ll p^{-50/77} + X^{-50/77} \ll p^{-50/77},
\end{align*}
the last inequality following since $N < M$. Thus
\begin{align*}
D_3''(M,N) &\ll (MN)^{\gamma_0} N \sum_{P < p \ll N} p^{-50/77} \sum_{\substack{d \ll N \\ p \mid d}} \frac{1}{d} \sum_{r=1}^d F_X \left(\frac{r}{d} \right).
\end{align*}
We separate the contribution of the zero frequency $r = d$, and find that it contributes
\begin{align*}
\ll (\log N) (MN)^{\gamma_0} N P^{-50/77}.
\end{align*}
For the nonzero frequencies we reduce to primitive fractions, obtaining
\begin{align*}
&\sum_{p > P} p^{-50/77} \sum_{\substack{d \ll N \\ p \mid d}} \frac{1}{d} \sum_{r=1}^d F_X \left(\frac{r}{d} \right) \ll \sum_{p > P} p^{-50/77} \sum_{\substack{d \ll N \\ p \mid d}} \frac{1}{d} \sum_{\substack{q \mid d \\ q > 1}} \sum_{(a,q)=1} F_X \left( \frac{a}{q}\right) \\
&\ll \sum_{\substack{q \ll N }} \frac{1}{q}\sum_{(a,q)=1} F_X \left( \frac{a}{q}\right) \sum_{d' \ll N/q} \frac{1}{d'} \sum_{\substack{p > P \\ p \mid q}} p^{-50/77} \\
&+ \sum_{q \ll N} \frac{1}{q}\sum_{(a,q)=1} F_X \left( \frac{a}{q}\right) \sum_{\substack{d' \ll N/q}} \frac{1}{d'} \sum_{\substack{p > P \\ p \mid d'}} p^{-50/77}.
\end{align*}
Here we have written $d = qd'$ and used the fact that $p \mid qd'$ implies $p \mid q$ or $p \mid d'$. We change variables $d' = pd''$, say, and use the bound
\begin{align*}
\sum_{\substack{p > P \\ p \mid k}} \frac{1}{p^{50/77}} \ll (\log k)P^{-50/77},
\end{align*}
to obtain
\begin{align*}
D_3''(M,N) &\ll (\log N)^2 (MN)^{\gamma_0} N P^{-50/77} \sum_{q \ll N} \frac{1}{q}\sum_{(a,q)=1} F_X \left( \frac{a}{q}\right) \\ 
&\ll (\log N)^3 (MN)^{\gamma_0} N P^{-50/77}.
\end{align*}
The second inequality follows by a dyadic decomposition and Lemma \ref{maynard type I}. Thus the contribution from those $\ell_1$ not coprime to $|\Delta|$ is negligible.

Now that $b$ is a rational residue class, it remains only to make $b$ coprime with $|\Delta|$. It suffices to make $z_2$ coprime to $|\Delta|$ (recall that $z_1$ is already coprime to $|\Delta|$ since $z_1,z_2$ are coprime). Since $z_2$ has no small prime factors this condition is easy to impose. The details are by now familiar so we omit them. The error terms involved are of size
\begin{align*}
\ll (\log N)^{3} (MN)^{\gamma_0} N P^{-1}.
\end{align*}

In order to prove our desired bound \eqref{bound for cal D 1 MN}, it therefore suffices to study
\begin{align*}
\mathcal{D}_4(M,N) &:= \mathop{\sum \sum}_{\substack{N < |z_1|^2,|z_2|^2 \leq (1+\theta)N \\ (z_1 \overline{z_1} z_2 \overline{z_2},\Pi)=1 \\ (z_1 |\Delta|,z_2)=1 \\ |\Delta| > \theta^{18}N}} \mu(|z_1|^2) \mu(|z_2|^2) \sum_{bz_1 \equiv z_2 (|\Delta|)} \\ 
&\times\sum_{\substack{\ell_1 \leq \sqrt{2(1+\theta)} (MN)^{1/2} \\ (\ell_1,\Pi |\Delta|)=1}} \mathbf{1}_\mathcal{A}(\ell_1)   \sum_{\substack{\ell_2 \leq \sqrt{2(1+\theta)} (MN)^{1/2} \\ \ell_2\equiv b \ell_1 (|\Delta|) \\ |\ell_1 z_2 - \ell_2 z_1|^2 \leq 2 \Delta^2 M \\ (\ell_2,\Pi)=1}}  \mathbf{1}_\mathcal{A}(\ell_2), \nonumber
\end{align*}
and show
\begin{align}\label{bound for cal D c MN}
\mathcal{D}_4 (M,N) &\ll \theta^{18}(\log MN)^{O(1)}  (MN)^{\gamma_0} N.
\end{align}

Before we proceed further, there is another technical issue to resolve. As mentioned above, the sequence $\mathcal{A}$ is nicely distributed in residue classes to moduli that are coprime to 10, but things become more complicated if the modulus is not coprime to 10. In an effort to isolate this poor behavior at the primes 2 and 5 we write
\begin{align*}
\Delta = \Delta_{10} |\Delta'|,
\end{align*}
where $\Delta_{10}$ is a positive divisor of $10^\infty$ and $(|\Delta'|,10) = 1$. Note that $2 \mid \Delta_{10}$ since $z_1$ and $z_2$ are primary (see \eqref{z primary}). By the Chinese remainder theorem we can think about the congruence $\ell_2 \equiv b \ell_1 (|\Delta|)$ as two separate rational congruences, one to modulus $\Delta_{10}$ and one to modulus $|\Delta'|$. Because integers divisible only by the primes 2 and 5 form a very sparse subset of all the integers, we expect the contribution to $\mathcal{D}_4(M,N)$ from large $\Delta_{10}$ to be negligible. We finish this section with the following result.
\begin{lemma}\label{lem: remove large Delta10}
The contribution to $\mathcal{D}_4(M,N)$ from $\Delta_{10} > \theta^{-28}$ is
\begin{align*}
\ll \theta^{18} (\log MN)^{O(1)} (MN)^{\gamma_0} N.
\end{align*}
\end{lemma}
\begin{proof}
The contribution to $\mathcal{D}_4(M,N)$ from $\Delta_{10} > \theta^{-28}$ is bounded above by a constant multiple of
\begin{align*}
\mathcal{D}_4' &:= \mathop{\sum \sum}_{\substack{|z_1|^2,|z_2|^2 \ll N \\ (z_1,z_2)=1 \\ \Delta_{10} > \theta^{-28}}} \sum_{\substack{b (|\Delta|) \\ bz_1 \equiv z_2 (|\Delta|)}} \sum_{\ell_1 < X} \mathbf{1}_\mathcal{A}(\ell_1) 
\sum_{\substack{\ell_2 < X \\ \ell_2 \equiv b \ell_1 (\Delta_{10}) \\ \ell_2 \equiv b \ell_1 (|\Delta'|)}} \mathbf{1}_\mathcal{A}(\ell_2),
\end{align*}
where $X \asymp (MN)^{1/2}$ is a power of 10. We apply additive characters to detect the congruences, obtaining
\begin{align*}
\sum_{\substack{\ell_2 < X \\ \ell_2 \equiv b \ell_1 (\Delta_{10}) \\ \ell_2 \equiv b \ell_1 (|\Delta'|)}} \mathbf{1}_\mathcal{A}(\ell_2) &= \frac{1}{\Delta_{10}|\Delta'|} \sum_{s=1}^{\Delta_{10}} e \left(-\frac{sb\ell_1}{\Delta_{10}} \right) \sum_{k=1}^{|\Delta'|} e \left( - \frac{kb\ell_1}{|\Delta'|} \right) \sum_{\ell_2 < X} \mathbf{1}_\mathcal{A}(\ell_2) e \left(\frac{k\ell_2}{|\Delta'|}+\frac{s\ell_2}{\Delta_{10}}  \right).
\end{align*}
We first consider the contribution from $k = |\Delta'|$. By the triangle inequality, the contribution from $k = |\Delta'|$ to $\mathcal{D}_4'$ is
\begin{align*}
&\ll (MN)^{\gamma_0} \mathop{\sum \sum}_{\substack{|z_1|^2,|z_2|^2 \ll N \\ (z_1,z_2)=1 \\ \Delta_{10} > \theta^{-28}}} \frac{1}{\Delta_{10}|\Delta'|} \sum_{s=1}^{\Delta_{10}} F_X \left( \frac{s}{\Delta_{10}} \right) \\ 
&\ll (MN)^{\gamma_0} N \sum_{\substack{\theta^{-28}<d \ll N \\ d \mid 10^\infty}}\frac{1}{d} \sum_{s=1}^{d}F_X \left( \frac{s}{d} \right) \sum_{|\Delta'| \ll N} \frac{1}{|\Delta'|} \\
&\ll (\log N) (MN)^{\gamma_0} N \sum_{\substack{d>\theta^{-28} \\ d \mid 10^\infty}} \frac{1}{d^{50/77}},
\end{align*}
the last inequality following by Lemma \ref{maynard single denominator}. By Rankin's trick and an Euler product computation,
\begin{align*}
\sum_{\substack{d > \theta^{-28} \\ d \mid 10^\infty}} \frac{1}{d^{50/77}} &\ll \theta^{28(50/77 - 1/\log \log N)} \sum_{d \mid 10^\infty} d^{-1/\log \log N} \ll \theta^{18}(\log \log N)^2 .
\end{align*}

Let us now turn to the case in which $1 \leq k \leq |\Delta'| - 1$. The argument is a more elaborate version of the proof of Lemma \ref{prototypical estimation lemma two}. Arguing as in the case $k = |\Delta'|$ and changing variables, this contribution is bounded by
\begin{align*}
\ll (MN)^{\gamma_0} N \sum_{\substack{t \ll N \\ t \mid 10^\infty}} \frac{1}{t}\sum_{s=1}^t \sum_{\substack{1 < e \ll N \\ (e,10)=1}} \frac{1}{e} \sum_{k=1}^{e-1} F_X \left(\frac{k}{e}+\frac{s}{t}  \right).
\end{align*}
Reducing from fractions with denominator $e$ to primitive fractions gives that this last quantity is bounded by
\begin{align*}
\ll \log N (MN)^{\gamma_0} N \sum_{\substack{t \ll N \\ t \mid 10^\infty}} \frac{1}{t} \sum_{s = 1}^{t} \sum_{\substack{1 < q \ll N \\ (q,10)=1}}\frac{1}{q} \sum_{(r,q)=1} F_X \left(\frac{r}{q} + \frac{s}{t} \right).
\end{align*}
We break the sum over $q$ into $q \leq Q$ and $q > Q$, where $Q = \exp(\varepsilon \sqrt{\log N})$ with $\varepsilon > 0$ sufficiently small. We first handle $q > Q$. Taking the supremum over $s$ and $t$, the contribution from $q > Q$ is
\begin{align*}
\ll (\log N)^3 \ \sup_{\beta \in \mathbb{R}} \ \sum_{\substack{Q < q \ll N \\ (q,10) = 1}} \frac{1}{q} \sum_{(r,q)=1} F_X \left(\frac{r}{q} + \beta \right).
\end{align*}
We break the range of $q$ into dyadic segments and apply Lemma \ref{maynard type I}, which gives that the contribution from $q > Q$ is
\begin{align*}
\ll \frac{(\log N)^4}{Q^{23/77}}.
\end{align*}

Now we turn to $q \leq Q$. We first show that the contribution from $t > T = Q^4$, say, is negligible. Interchanging the order of summation,
\begin{align*}
&\sum_{\substack{1 < q \leq Q \\ (q,10)=1}} \frac{1}{q} \sum_{(r,q)=1} \sum_{\substack{t > T \\ t \mid 10^\infty}} \frac{1}{t} \sum_{r=1}^t F_X \left(\frac{s}{t} + \frac{r}{q} \right) \ll Q \ \sup_{\beta \in \mathbb{R}} \sum_{\substack{t > T \\ t \mid 10^\infty}} \frac{1}{t} \sum_{r=1}^t F_X \left(\frac{s}{t} + \beta \right) \\
&\ll Q \sum_{\substack{t > T \\ t \mid 10^\infty}} \frac{1}{t^{50/77}},
\end{align*}
the last inequality following from Lemma \ref{maynard single denominator}. By Rankin's trick, we obtain the bound
\begin{align*}
Q \sum_{\substack{t > T \\ t \mid 10^\infty}} \frac{1}{t^{50/77}} \ll \frac{Q}{T^{1/2}} = \frac{1}{Q}.
\end{align*}
It therefore suffices to bound
\begin{align*}
\sum_{\substack{t \leq Q^4 \\ t \mid 10^\infty}} \frac{1}{t} \sum_{s=1}^t\sum_{\substack{1 < q \leq Q \\ (q,10)=1}}\frac{1}{q} \sum_{(r,q)=1} F_X \left(\frac{r}{q} + \frac{s}{t} \right).
\end{align*}
At this point we avail ourselves of the product formula \eqref{product formula for F} for $F$. We take $U$ to be a power of 10 such that $t$ divides $U$ for every $t \mid 10^\infty$ with $t \leq Q^4$, and set $V = X/U$. Any such $t \leq Q^4$ may be written as $t = 2^a 5^b$ with $2^a 5^b \leq Q^4$. Clearly $a,b \leq \frac{4}{\log 2} \log Q$. We take $U = 10^c$ with $c = \frac{4}{\log 2}\log Q + O(1)$, so that
\begin{align*}
U \asymp Q^{\frac{4}{\log 2} \log 10} \ll Q^{14},
\end{align*}
say. Since $F$ is 1-periodic we obtain
\begin{align*}
F_X \left(\frac{r}{q} + \frac{s}{t} \right) &= F_U \left(\frac{r}{q} + \frac{s}{t} \right) F_V \left(\frac{Ur}{q} + \frac{Us}{t} \right) = F_U \left(\frac{r}{q} + \frac{s}{t} \right) F_V \left(\frac{Ur}{q} \right) \ll F_V \left(\frac{Ur}{q} \right).
\end{align*}
Observe that $V$ and $X$ are asymptotically equal in the logarithmic scale since $Q = N^{o(1)}$. We then apply Lemma \ref{maynard small moduli} to bound each $F_V$ individually, and find
\begin{align}\label{elaboration of prev proto lemma}
\mathcal{D}_4' &\ll (\log N)^4 (MN)^{\gamma_0} N \left(\theta^{18} + \exp (-c' \sqrt{\log MN}) \right).
\end{align}
\end{proof}

In light of Lemma \ref{lem: remove large Delta10} it suffices to show that
\begin{align}\label{bound for cal D s MN}
\mathcal{D}_5 (M,N) &\ll \theta^{18}(\log MN)^{O(1)}  (MN)^{\gamma_0} N,
\end{align}
where $\mathcal{D}_5(M,N)$ is the same as $\mathcal{D}_4(M,N)$, but with the additional condition that $\Delta_{10} \leq \theta^{-28}$.

\section{Polar boxes and the fundamental lemma}\label{sec: polar boxes}

In this section we remove the congruence condition modulo $|\Delta'|$, which will simplify the situation considerably. Not surprisingly, there are several technical barriers to overcome before this can be accomplished. For instance, the condition
\begin{align*}
|\ell_1z_2 - \ell_2 z_1|^2 \leq 2 \Delta^2 M
\end{align*}
entangles the four variables $z_1,z_2,\ell_1$, and $\ell_2$. We put $z_1$ and $z_2$ into polar boxes in order to reduce some of this dependence. After restricting to ``generic'' boxes and removing as much $z_1$ and $z_2$ dependence as we can, we break the sum over $\ell_2$ into short intervals in preparation for applying additive characters. We employ the fundamental lemma to handle the condition $(\ell_2,\Pi) = 1$. The error term is estimated as we have done before, using distribution results for $F_Y$. With the congruence condition modulo $|\Delta'|$ removed, we can make some simplifications and adjustments in the main term. The last task is then to get cancellation from the M\"obius function in the main term, which we do in Section \ref{sec: endgame}.

We begin with a preparatory lemma.
\begin{lemma}\label{delta in short interval}
Let $\exp(- (\log MN)^{1/3}) < \delta < \frac{1}{2}$, and let $C' > 0$ be an absolute constant. Then
\begin{align*}
\mathcal{D}_5'(M,N) &:= \mathop{\sum \sum}_{\substack{N < |z_1|^2,|z_2|^2 \leq 2N \\ (z_1,z_2)=1 \\ (1-C'\delta)\theta^{18}N < |\Delta| \leq (1+C'\delta)\theta^{18}N \\ \Delta_{10} \leq \theta^{-28}}} \mathop{\sum \sum}_{\substack{\ell_1,\ell_2 \ll (MN)^{1/2} \\ \ell_1 z_2 \equiv \ell_2 z_1 (|\Delta|)}} \mathbf{1}_\mathcal{A}(\ell_1) \mathbf{1}_\mathcal{A}(\ell_2) \ll \delta \theta^{-28} (MN)^{\gamma_0} N.
\end{align*}
\end{lemma}
\begin{proof}
We begin by handling the congruence condition as we did in imposing the condition $\Delta_{10} \leq \theta^{-28}$ in Lemma \ref{lem: remove large Delta10}. We detect the congruences modulo $\Delta_{10}$ and $|\Delta'|$ with additive characters. The nonzero frequencies modulo $|\Delta'|$ contribute an acceptably small error term, by the argument that led to \eqref{elaboration of prev proto lemma}. For the zero frequency modulo $|\Delta'|$ we apply orthogonality of additive characters to reintroduce the congruence modulo $\Delta_{10}$. We find
\begin{align*}
\mathcal{D}_5'(M,N) &= (1+o(1))\mathop{\sum \sum}_{\substack{N < |z_1|^2,|z_2|^2 \leq 2N \\ (z_1,z_2)=1 \\ (1-C'\delta)\theta^{18}N < |\Delta| \leq (1+C'\delta)\theta^{18}N \\ \Delta_{10} \leq \theta^{-28}}} \frac{1}{|\Delta'|} \mathop{\sum \sum}_{\substack{\ell_1,\ell_2 \ll (MN)^{1/2} \\ \ell_1 z_2 \equiv \ell_2 z_1 (\Delta_{10})}} \mathbf{1}_\mathcal{A}(\ell_1) \mathbf{1}_\mathcal{A}(\ell_2).
\end{align*}
Since
\begin{align*}
\theta^{18} N \ll |\Delta| = \Delta_{10} |\Delta'|,
\end{align*}
we see that
\begin{align*}
|\Delta'|^{-1} \ll \theta^{-28} J^{-1},
\end{align*}
where $J = \theta^{18} N$. Dropping the congruence condition modulo $\Delta_{10}$, it follows that
\begin{align*}
\mathcal{D}_5'(M,N) &\ll \frac{\theta^{-28}}{J} (MN)^{\gamma_0} \mathop{\sum \sum}_{\substack{N < |z_1|^2,|z_2|^2 \leq 2N \\ (1-C'\delta)J < |\Delta| \leq (1+C'\delta)J}} 1 \\ 
&\ll \frac{\theta^{-28}}{J} (MN)^{\gamma_0} \mathop{\sum \sum}_{\substack{N < r^2 + s^2 \leq 2N \\ (r,s)=1}} \mathop{\sum \sum}_{\substack{u,v \ll N^{1/2} \\ (1-C'\delta)J < |rv-su| \leq (1+C'\delta)J}} 1.
\end{align*}
Observe that the conditions on $r$ and $s$ imply $rs \neq 0$. Given $r,s$, and $u$ the number of $v$ is $\ll \delta J/|r|$, and given $r,s$, and $v$ the number of $u$ is $\ll \delta J/|s|$. Since $\max(|r|,|s|) \gg N^{1/2}$, we see that
\begin{align*}
\mathop{\sum \sum}_{\substack{u,v \ll N^{1/2} \\ (1-C'\delta)J < |rv-su| \leq (1+C'\delta)J}} 1 \ll \delta J.
\end{align*}
Summing over $r$ and $s$ then completes the proof.
\end{proof}

We now introduce a parameter $\lambda = k^{-1}$, for some $k \in \mathbb{N}$ to be chosen. We break the sums over $z_1,z_2$ into polar boxes, so that
\begin{align*}
z \in \mathfrak{B} = \left\{w \in \mathbb{C} : R_i < |w|^2 \leq (1+\lambda)R_i, \ \delta_j \leq \text{arg}(w) \leq \delta_j + 2\pi \lambda \right\}.
\end{align*}
Note that $N < R_i \leq (1+\theta) N$ and $\delta_j = 2\pi j \lambda$ for $0 \leq j \leq \lambda^{-1}-1$ an integer. For such a polar box, let $z (\mathfrak{B}) := R_i e^{i\delta_j}$. The number of polar boxes for $z_1,z_2$ is $O \left(\lambda^{-4} \right)$, and we have the trivial bound
\begin{align*}
\sum_{\substack{z \in \mathfrak{B}}} 1 \ll \lambda^2 N.
\end{align*}
Set $\Delta(\mathfrak{B}_1,\mathfrak{B}_2) := \Delta (z (\mathfrak{B}_1),z (\mathfrak{B}_2))$. From the lower bound for $|\Delta|$, we see the polar boxes $\mathfrak{B}_1$, $\mathfrak{B}_2$ cannot be too close to one another, in a sense. Writing $z_i = r_i e^{i\theta_i}$, we see
\begin{align*}
|\Delta(z_1,z_2)| &= r_1r_2 |\sin(\theta_2 - \theta_1)| > \theta^{18} N,
\end{align*}
after using the fact that $e^{i\theta} = \cos \theta + i \sin \theta$. Since $r_1,r_2 \asymp N^{1/2}$, we have
\begin{align*}
|\sin(\theta_2 - \theta_1)| \gg \theta^{18}.
\end{align*}
Recall that $\theta_i = \delta_i + O(\lambda)$. If we assume that $\lambda \leq \varepsilon \theta^{18}$ for some sufficiently small $\varepsilon > 0$, then the sine angle addition formula and the triangle inequality imply
\begin{align*}
|\sin(\delta_2 - \delta_1)| \gg \theta^{18}.
\end{align*}
Thus the angles $\delta_1,\delta_2$ cannot be too close to each other. Given this fact, we may show in the same manner that
\begin{align}\label{box delta close to z delta}
\Delta(z_1,z_2) = (1+O(\lambda'))\Delta(\mathfrak{B}_1,\mathfrak{B}_2),
\end{align}
where $\lambda' = \theta^{-18} \lambda$.

We claim it suffices to sum over polar boxes $\mathfrak{B}_1,\mathfrak{B}_2$ such that $|\Delta(\mathfrak{B}_1,\mathfrak{B}_2)| > (1+\lambda') \theta^{18} N$. Indeed, the sum over polar boxes not satisfying this condition is bounded by
\begin{align*}
\mathop{\sum \sum}_{\substack{N < |z_1|^2,|z_2|^2 \leq 2N \\ (z_1,z_2)=1 \\ (1-C'\lambda')\theta^{18}N < |\Delta| \leq (1+C'\lambda')\theta^{18}N \\ \Delta_{10} \leq \theta^{-28}}} \mathop{\sum \sum}_{\substack{\ell_1,\ell_2 \ll (MN)^{1/2} \\ \ell_1 z_2 \equiv \ell_2 z_1 (|\Delta|)}} \mathbf{1}_\mathcal{A}(\ell_1) \mathbf{1}_\mathcal{A}(\ell_2)
\end{align*}
for some absolute constant $C' > 0$, and by Lemma \ref{delta in short interval} this quantity is $\ll\theta^{-28} \lambda' (MN)^{\gamma_0} N$. This bound is acceptable for \eqref{bound for cal D s MN} provided 
\begin{align*}
\lambda \leq \theta^{64},
\end{align*}
which we now assume.

The number of boxes intersecting the boundary of $\{z : N < |z|^2 \leq (1+\theta)N\}$ is $O(\lambda^{-2})$. Handling the congruences modulo $|\Delta'|$ and $\Delta_{10}$ as in Lemma \ref{delta in short interval}, we find the error made by this approximation is
\begin{align*}
\ll \theta^{-46} \lambda^2 (MN)^{\gamma_0} N,
\end{align*}
and this error is acceptable for \eqref{bound for cal D s MN} since we have already imposed the condition $\lambda \leq \theta^{64}$. We therefore have
\begin{align*}
\mathcal{D}_5(M,N) &= O \left(\theta^{18}(MN)^{\gamma_0}N \right) + \mathop{\sum \sum}_{\substack{\mathfrak{B}_1, \mathfrak{B}_2 \\ |\Delta(\mathfrak{B}_1,\mathfrak{B}_2)| > (1+\lambda')\theta^{18}N}} \mathcal{D}_1(\mathfrak{B}_1,\mathfrak{B}_2),
\end{align*}
where
\begin{align*}
\mathcal{D}_1(\mathfrak{B}_1,\mathfrak{B}_2) &:= \mathop{\sum \sum}_{\substack{z_1 \in \mathfrak{B}_1, z_2 \in \mathfrak{B}_2 \\ (z_1 \overline{z_1} z_2 \overline{z_2},\Pi)=1 \\ (z_1 |\Delta|,z_2)=1 \\ \Delta_{10} \leq \theta^{-28}}} \mu(|z_1|^2) \mu(|z_2|^2) \sum_{bz_1 \equiv z_2 (|\Delta|)} \\ 
&\times\sum_{\substack{\ell_1 \leq \sqrt{2(1+\theta)} (MN)^{1/2} \\ (\ell_1,\Pi |\Delta|)=1}} \mathbf{1}_\mathcal{A}(\ell_1)   \sum_{\substack{\ell_2 \leq \sqrt{2(1+\theta)} (MN)^{1/2} \\ \ell_2\equiv b \ell_1 (|\Delta|) \\ |\ell_1 z_2 - \ell_2 z_1|^2 \leq 2 \Delta^2 M \\ (\ell_2,\Pi)=1}}  \mathbf{1}_\mathcal{A}(\ell_2). \nonumber
\end{align*}
Observe that $\mathcal{D}_1(\mathfrak{B}_1,\mathfrak{B}_2)$ depends on $M$ and $N$, but we have suppressed this in the notation. It therefore suffices to show that
\begin{align}\label{bound for cal D B1 B2}
\mathcal{D}_1(\mathfrak{B}_1,\mathfrak{B}_2) &\ll \theta^{18}(\log MN)^{O(1)}  \lambda^4 (MN)^{\gamma_0} N
\end{align}
uniformly in $\mathfrak{B}_1$ and $\mathfrak{B}_2$.

We now work to make the condition $|\ell_1 z_2 - \ell_2 z_1|^2 \leq 2\Delta^2 M$ less dependent on $z_1$ and $z_2$. We can rearrange to get the condition
\begin{align}\label{not indep of z1 z2}
\left|\ell_2 - \frac{\ell_1 z_2}{z_1} \right| &\leq \sqrt{2} \frac{|\Delta(z_1,z_2)| M^{1/2}}{|z_1|}.
\end{align}
We wish to replace \eqref{not indep of z1 z2} by
\begin{align}\label{indep of z1 z2}
\left|\ell_2 - \frac{\ell_1 z_2}{z_1} \right| \leq \sqrt{2} \frac{|\Delta(\mathfrak{B}_1,\mathfrak{B}_2)| M^{1/2}}{|z(\mathfrak{B}_1)|}.
\end{align}
Since
\begin{align*}
\frac{|\Delta(z_1,z_2)| M^{1/2}}{|z_1|} &= (1+O(\lambda')) \frac{|\Delta(\mathfrak{B}_1,\mathfrak{B}_2)| M^{1/2}}{|z(\mathfrak{B}_1)|},
\end{align*}
we see it suffices to bound the contribution from those $\ell_2$ that satisfy
\begin{align}\label{eq: refining upper bound on difference}
(1-C\lambda') K \leq \left|\ell_2 - \frac{\ell_1 z_2}{z_1} \right| \leq (1+C\lambda') K,
\end{align}
where
\begin{align*}
K &:= \sqrt{2}\frac{|\Delta(\mathfrak{B}_1,\mathfrak{B}_2)| M^{1/2}}{|z(\mathfrak{B}_1)|}
\end{align*}
and $C>0$ is a sufficiently large absolute constant. 

We claim that \eqref{eq: refining upper bound on difference} places $\ell_2$ in a bounded number of intervals (depending on $\ell_1,z_1,z_2$) of length $\ll (\lambda')^{1/2} J$. For notational simplicity, write $A = (1-C\lambda') K$ and $B = (1+C\lambda') K$. Then \eqref{eq: refining upper bound on difference} gives
\begin{align*}
A \leq |\ell_2 - (u+iv)| \leq B
\end{align*}
for some real numbers $u,v$. Since $\ell_2$ is real, we obtain by squaring and rearranging
\begin{align*}
A^2-v^2 \leq (\ell_2-u)^2 \leq B^2 - v^2.
\end{align*}
There are two cases now to consider: $v \geq A$ and $v < A$. If $v \geq A$ then $A^2-v^2 \leq 0$, and the lower bound is therefore automatically satisfied. We therefore obtain
\begin{align*}
|\ell_2-u| \leq \sqrt{B^2-v^2} \leq \sqrt{B^2-A^2} \ll (\lambda')^{1/2} K.
\end{align*}
Now suppose that $v < A$. Then
\begin{align*}
\sqrt{A^2-v^2} \leq |\ell_2-u| \leq \sqrt{B^2-v^2},
\end{align*}
and thus $\ell_2$ is in two intervals of length $\leq \sqrt{B^2-v^2} - \sqrt{A^2-v^2}+2$, say. We then have
\begin{align*}
\sqrt{B^2-v^2} - \sqrt{A^2-v^2} &= \frac{B^2-A^2}{\sqrt{B^2-v^2} + \sqrt{A^2-v^2}} \leq \frac{B^2-A^2}{\sqrt{B^2-v^2}} \leq \sqrt{B^2-A^2},
\end{align*}
and this completes the proof of the claim.

We now bound the contribution of those $\ell_2$ satisfying \eqref{eq: refining upper bound on difference}. At this point we should have enough experience to see how we should proceed. We let $Y$ be the largest power of 10 satisfying $Y \leq (\lambda')^{1/2} K$, and cover the intervals \eqref{eq: refining upper bound on difference} with subintervals of the form $[nY,nY + Y)$, $n \geq 0$ an integer. The number of subintervals is $O(1)$. We can reduce to summing the indicator function $\mathbf{1}_\mathcal{A}(t)$ over $0 \leq t < Y$, and then deal with the congruence modulo $|\Delta|$ by considering it as a congruence modulo $|\Delta'|$ and $\Delta_{10}$. We obtain a bound of
\begin{align*}
\ll \theta^{-46} (\lambda')^{\gamma_0/2} \lambda^4 (MN)^{\gamma_0} N,
\end{align*}
and this is acceptable for \eqref{bound for cal D B1 B2} provided $\lambda \ll \theta^{153}$. We set
\begin{align*}
\lambda \asymp \theta^{153}.
\end{align*}

We now have the conditions
\begin{align}\label{eq: three conditions on ell2}
\ell_2 &\leq \sqrt{2 (1+\theta)} (MN)^{1/2}, \nonumber \\
\left|\ell_2 - \frac{\ell_1 z_2}{z_1} \right| &\leq \sqrt{2} \frac{|\Delta(\mathfrak{B}_1,\mathfrak{B}_2)| M^{1/2}}{|z(\mathfrak{B}_1)|}, \\
\ell_2 &\equiv b \ell_1 (|\Delta|) \nonumber
\end{align}
on $\ell_2$. Recall that the congruence is a congruence of rational integers. To handle the first two conditions we perform a short interval decomposition. Let $Y$ be the largest power of 10 which satisfies
\begin{align*}
Y \leq \lambda\frac{|\Delta(\mathfrak{B}_1,\mathfrak{B}_2)| M^{1/2}}{|z(\mathfrak{B}_1)|}.
\end{align*}
We cover the interval $\ell_2 \leq \sqrt{2 (1+\theta)} (MN)^{1/2}$ with subintervals of the form $[nY,nY + Y)$, as we have done many times before. For the subintervals that intersect the boundary of the second condition of \eqref{eq: three conditions on ell2} we obtain acceptable contributions. The sum over $\ell_2$ has therefore become
\begin{align*}
\sum_{\substack{n \in \mathbb{Z} \\ n \in S(\ell_1,z_1,z_2)}} \sum_{\substack{nY < \ell_2 \leq nY + Y \\ \ell_2 \equiv b \ell_1 (|\Delta|) \\ (\ell_2,\Pi)=1}} \mathbf{1}_\mathcal{A}(\ell_2),
\end{align*}
for some set $S(\ell_1,z_1,z_2)$ of size $O( \lambda^{-1})$.

We handle the condition $(\ell_2,\Pi) = 1$ using the fundamental lemma. Let
\begin{align*}
\Sigma &:= \mathop{\sum \sum}_{z_i \in \mathfrak{B}_i} \sum_{b(|\Delta|)} \sum_{\ell_1} \sum_n \sum_{\ell_2}
\end{align*}
be the sum we wish to bound (up to acceptable errors, $\Sigma$ is $\mathcal{D}(\mathfrak{B_1},\mathfrak{B}_2)$ with the condition \eqref{not indep of z1 z2} replaced by \eqref{indep of z1 z2}). We partition $\Sigma$ as
\begin{align*}
\Sigma = \Sigma_+ + \Sigma_-,
\end{align*}
where in $\Sigma_+$ we sum over those $z_1,z_2$ such that $\mu(|z_1|^2) \mu(|z_2|^2) > 0$, and in $\Sigma_-$ we sum over those $z_1,z_2$ such that $\mu(|z_1|^2) \mu(|z_2|^2) < 0$. We get an upper bound on $\Sigma_+$ using an upper-bound linear sieve of level $D$
\begin{align*}
\mathbf{1}_{(\ell_2,\Pi)=1} &\leq \mathbf{1}_{(\ell_2,10)=1} \sum_{\substack{d \leq D \\ d \mid \Pi/10 \\ d \mid \ell_2}} \lambda_d^{+},
\end{align*}
and a lower bound on $\Sigma_-$ using a lower-bound linear sieve of level $D$
\begin{align*}
\mathbf{1}_{(\ell_2,\Pi)=1} &\geq \mathbf{1}_{(\ell_2,10)=1} \sum_{\substack{d \leq D \\ d \mid \Pi/10 \\ d \mid \ell_2}} \lambda_d^{-},
\end{align*}
where $D$ is chosen shortly (see \eqref{choice for D}). This yields an upper bound on $\Sigma$. Reversing $\lambda^+$ and $\lambda^-$ we get a lower bound on $\Sigma$, and we show that these bounds are the same asymptotically. Thus, for some sign $\epsilon$, it suffices to study
\begin{align}\label{eq: Sigma epsilon}
\Sigma_{\epsilon}^{\pm} := \mathop{\sum \sum}_{\substack{z_i \in \mathfrak{B}_i \\ \mu(|z_1|^2)\mu(|z_2|^2) = \epsilon}}  \sum_{\substack{d \leq D \\ d \mid \Pi \\ (d,10|\Delta|)=1}} \lambda_d^{\pm}  \sum_{b{ (|\Delta|)}} \sum_{\ell_1} \sum_{n} \sum_{\substack{nY < \ell_2 \leq nY + Y \\ \ell_2 \equiv b \ell_1 (|\Delta|) \\ \ell_2 \equiv 0 (d) \\ (\ell_2,10)=1}} \mathbf{1}_\mathcal{A}(\ell_2).
\end{align}
Observe that we have suppressed several conditions in the notation, but these conditions are not to be forgotten.

We write the congruence modulo $|\Delta|$ as two congruences modulo $\Delta_{10}$ and $|\Delta'|$, and then use the Chinese remainder theorem to combine the congruences modulo $d$ and $|\Delta'|$ into a congruence modulo $d|\Delta'|$. Considering separately the cases $a_0 \neq 0$ and $a_0 = 0$ and then applying inclusion-exclusion if necessary, we can reduce to having the sum over $\ell_2$ be a sum over $0 \leq t < Y'$, where $Y' = Y$ or $Y' = Y/10$. Applying additive characters, the sum over $\ell_2$ in \eqref{eq: Sigma epsilon} becomes a linear combination of a bounded number of quantities of the form
\begin{align*}
\frac{1}{d|\Delta'|} \sum_{f=1}^{d|\Delta'|} e \left(\frac{-f \nu}{d|\Delta'|} \right) \sum_{\substack{t < Y' \\ t + nY \equiv b \ell_1 (\Delta_{10}) \\ (t,10)=1)}} \mathbf{1}_\mathcal{A}(t) e \left(\frac{ft}{d|\Delta'|} \right),
\end{align*}
where $\nu = \nu(z_1,z_2,\ell_1,n)$ is some residue class. The term $f = d|\Delta'|$ supplies the main term, which we discuss later. For now we turn our attention to the error term $\Sigma_{\epsilon,E}^{\pm}$, which comes from $1 \leq f \leq d|\Delta'| - 1$. The argument is similar to that which gave \eqref{elaboration of prev proto lemma} in Lemma \ref{lem: remove large Delta10}. 

We apply additive characters to detect the congruence modulo $\Delta_{10}$, apply M\"obius inversion to trade the condition $(t,10) = 1$ for congruence conditions, and then apply additive characters again to detect these latter congruence conditions. We then apply the triangle inequality to eliminate the dependencies on $\ell_1,b$, and $n$. We obtain
\begin{align*}
\Sigma_{\epsilon,E}^{\pm} &\ll \lambda^{-1} (MN)^{\gamma_0} \sum_{\substack{d \leq D \\ (d,10) = 1}} \sum_{\substack{|\Delta'| \ll N \\ (|\Delta'|,10)=1 \\ d|\Delta'| > 1}} \frac{1}{d|\Delta'|} \sum_{f=1}^{d|\Delta'|-1}\sum_{\substack{\Delta_{10} \leq \theta^{-22} \\ \Delta_{10} \mid 10^\infty}} \frac{1}{\Delta_{10}} \sum_{g=1}^{\Delta_{10}} \sum_{h \mid 10} \sum_{k=1}^h \\ 
&\times F_{Y'} \left(\frac{f}{d|\Delta'|} + \frac{g}{\Delta_{10}} + \frac{k}{h} \right) \sum_{z_1 \in \mathfrak{B}_1} \sum_{\substack{z_2 \in \mathfrak{B}_2 \\ \text{Im}(\overline{z_1}z_2) = \Delta_{10}\Delta'}} 1 \\
&\ll \lambda (MN)^{\gamma_0} N \sum_{\substack{1 < d \ll DN \\ (d,10)=1}} \frac{\tau(d)}{d} \sum_{f = 1}^{d-1} \sum_{\substack{\Delta_{10} \leq \theta^{-22} \\ \Delta_{10} \mid 10^\infty}} \frac{1}{\Delta_{10}} \sum_{g=1}^{\Delta_{10}} \sum_{h \mid 10} \sum_{k=1}^h F_{Y'} \left(\frac{f}{d} + \frac{g}{\Delta_{10}} + \frac{k}{h} \right).
\end{align*}
The second inequality follows, among other things, by changing variables $d |\Delta'| \rightarrow d$. We reduce to primitive fractions to obtain
\begin{align*}
\Sigma_{\epsilon,E}^{\pm}  &\ll (\log N)^2 \lambda (MN)^{\gamma_0} N \sum_{\substack{1 < q \ll DN \\ (q,10)=1}} \frac{\tau(q)}{q} \sum_{\substack{r = 1 \\ (r,q)=1}}^{q} \sum_{\substack{\Delta_{10} \leq \theta^{-22} \\ \Delta_{10} \mid 10^\infty}} \frac{1}{\Delta_{10}} \sum_{g=1}^{\Delta_{10}} \sum_{h \mid 10} \sum_{k=1}^h F_{Y'} \left(\frac{r}{q} + \frac{g}{\Delta_{10}} + \frac{k}{h} \right).
\end{align*}
We choose 
\begin{align}\label{choice for D}
D := x^{1/\log \log x},
\end{align}
so that $DN \ll x^{25/77 - \epsilon}$. We estimate the contribution from $q > Q = \exp((\log MN)^{1/3})$ using the divisor bound, dyadic decomposition, and Lemma \ref{maynard type I}. For $q \leq Q$ we first use the product formula for $F_{Y'}$ to eliminate $\frac{g}{\Delta_{10}} + \frac{k}{h}$, and then use Lemma \ref{maynard small moduli}.

Let us now turn to the main term we alluded to above. We reverse the transition from $\ell_2$ to $t$, and then undo our short interval decomposition. Up to acceptable error terms, the main term is then given by
\begin{align*}
\Sigma_{\epsilon,0}^{\pm} &:= \mathop{\sum \sum}_{\substack{z_i \in \mathfrak{B}_i \\ \mu(|z_1|^2)\mu(|z_2|^2) = \epsilon}} \frac{1}{|\Delta'|}  \sum_{\substack{d \leq D \\ d \mid \Pi \\ (d,10|\Delta|)=1}} \frac{\lambda_d^{\pm}}{d} \sum_{\substack{\ell_1\leq \sqrt{2(1+\theta)}(MN)^{1/2} \\ (\ell,\Pi |\Delta|)=1}} \mathbf{1}_\mathcal{A}(\ell_1) \sum_{\substack{\ell_2 \leq \sqrt{2(1+\theta)}(MN)^{1/2} \\ \ell_2z_1 \equiv \ell_1 z_2 (\Delta_{10}) \\ \eqref{indep of z1 z2} \\ (\ell_2,10)=1}} \mathbf{1}_\mathcal{A}(\ell_2).
\end{align*}
From the fundamental lemma of sieve theory (see \eqref{fund lem ref}, for example) we have
\begin{align}\label{eq: last fund lem}
\sum_{\substack{d \leq D \\ d \mid \Pi \\ (d,10|\Delta|)=1}} \frac{\lambda_d^{\pm}}{d} &= \left(1 + \exp \left(-\frac{1}{2}s \log s \right) \right) \prod_{\substack{p \leq P \\ p \nmid 10 |\Delta|}} \left(1 - \frac{1}{p} \right),
\end{align}
where
\begin{align*}
s = \frac{\log D}{\log P} \geq \sqrt{\log x} \gg \sqrt{\log MN}.
\end{align*}
The error term of \eqref{eq: last fund lem} is therefore acceptably small for \eqref{bound for cal D B1 B2}. We write
\begin{align*}
\prod_{\substack{p \leq P \\ p \nmid 10 |\Delta|}} \left(1 - \frac{1}{p} \right) &= \prod_{p \leq P} \left(1 - \frac{1}{p} \right) \prod_{\substack{p \leq P \\ p \mid 10 |\Delta|}} \left(1 - \frac{1}{p} \right)^{-1} \\ 
&= \left(1 + O \left( \frac{\log N}{P} \right) \right) \prod_{p \leq P} \left(1 - \frac{1}{p} \right) \prod_{p \mid 10 |\Delta|} \left(1 - \frac{1}{p} \right)^{-1} \\
&=\left(1 + O \left( \frac{\log N}{P} \right) \right) \prod_{p \leq P} \left(1 - \frac{1}{p} \right) \frac{10|\Delta|}{\varphi(10|\Delta|)},
\end{align*}
and observe that the error term is again acceptable by our lower bound for $P$. Thus $\Sigma_{\epsilon,0}^{+}$ and $\Sigma_{\epsilon,0}^{-}$ are asymptotically equal, and up to acceptable error terms we have
\begin{align*}
\Sigma &= \mathop{\sum \sum}_{\substack{z_1 \in \mathfrak{B}_1, z_2 \in \mathfrak{B}_2 \\ (z_1 \overline{z_1} z_2 \overline{z_2},\Pi)=1 \\ (z_1 |\Delta|,z_2)=1 \\ \Delta_{10} \leq \theta^{-28}}} \mu(|z_1|^2) \mu(|z_2|^2) \frac{1}{|\Delta'|} \frac{10|\Delta|}{\varphi(10|\Delta|)} \\
&\times\sum_{\substack{\ell_1\leq \sqrt{2(1+\theta)}(MN)^{1/2} \\ (\ell,\Pi |\Delta|)=1}} \mathbf{1}_\mathcal{A}(\ell_1) \sum_{\substack{\ell_2 \leq \sqrt{2(1+\theta)}(MN)^{1/2} \\ \ell_2z_1 \equiv \ell_1 z_2 (\Delta_{10}) \\ \eqref{indep of z1 z2} \\ (\ell_2,10)=1}} \mathbf{1}_\mathcal{A}(\ell_2).
\end{align*}

We may use trivial estimations to replace condition \eqref{indep of z1 z2} by
\begin{align*}
|\ell_1z_2 - \ell_2 z_1|^2 \leq 2 \Delta(\mathfrak{B}_1,\mathfrak{B}_2)^2 M.
\end{align*}
Further, by trivial estimation we may also remove the conditions $(z_2,|\Delta|)=1$, $(z_1,z_2) = 1$, and $(\ell_1,|\Delta|) = 1$ at the cost of an acceptable error. Having removed these conditions, we then write
\begin{align*}
\frac{1}{|\Delta'|} = \frac{\Delta_{10}}{|\Delta|} = (1+O(\lambda')) \frac{\Delta_{10}}{|\Delta(\mathfrak{B}_1,\mathfrak{B}_2)|}
\end{align*}

It follows that
\begin{align*}
\mathcal{D}_1(\mathfrak{B}_1,\mathfrak{B}_2) &= |\Delta(\mathfrak{B}_1,\mathfrak{B}_2)|^{-1} \prod_{p \leq P} \left(1 - \frac{1}{P} \right) \mathcal{D}_2 (\mathfrak{B}_1,\mathfrak{B}_2) + O \left( \theta^{18}(\log MN)^{O(1)}  \lambda^4 (MN)^{\gamma_0} N \right),
\end{align*}
where
\begin{align*}
\mathcal{D}_2 (\mathfrak{B}_1,\mathfrak{B}_2) &:= \mathop{\sum \sum}_{\substack{z_1 \in \mathfrak{B}_1, z_2 \in \mathfrak{B}_2 \\ (z_1 \overline{z_1} z_2 \overline{z_2},\Pi)=1 \\ \Delta_{10} \leq \theta^{-28}}} \mu(|z_1|^2) \mu(|z_2|^2) \Delta_{10} \frac{10 |\Delta|}{\varphi(10|\Delta|)} \\ 
&\times \mathop{\sum \sum}_{\substack{\ell_1,\ell_2 \leq \sqrt{2(1+\theta)}(MN)^{1/2} \\ \ell_2z_1 \equiv \ell_1 z_2 (\Delta_{10}) \\|\ell_1z_2 - \ell_2 z_1|^2 \leq 2 \Delta(\mathfrak{B}_1,\mathfrak{B}_2)^2 M \\ (\ell_1, \Pi) = 1, \ (\ell_2,10)=1 }} \mathbf{1}_\mathcal{A}(\ell_1) \mathbf{1}_\mathcal{A}(\ell_2).
\end{align*}
Recall the lower bound $|\Delta(\mathfrak{B}_1,\mathfrak{B}_2)| \gg \theta^{18}N$. In order to prove \eqref{bound for cal D B1 B2} it therefore suffices to show that
\begin{align}\label{bound for cal D 0 B1 B2}
\mathcal{D}_2(\mathfrak{B}_1,\mathfrak{B}_2) &\ll \theta^{36}\lambda^4(\log MN)^{O(1)}   (MN)^{\gamma_0} N^2.
\end{align}

\section{Simplifications and endgame}\label{sec: endgame}

We have removed the congruence condition to modulus $|\Delta'|$. From this point onwards our estimates are more straightforward, since we do not have to work with congruence conditions on elements of $\mathcal{A}$ to large moduli.

Recall that our goal is to use the cancellation induced by the M\"obius function to show that $\mathcal{D}_2$ is small. We do not need to perform any averaging over $\ell_1$ and $\ell_2$, so we reduce to considering a sum over $z_1$ and $z_2$. After some manipulations, including splitting into more polar boxes to separate $z_1$ and $z_2$, we reduce to finding cancellation when $z_1$ and $z_2$ are summed over arithmetic progressions whose moduli are bounded by a fixed (but large) power of a logarithm. We detect these congruences with multiplicative characters. We can then get cancellation from the zero-free region for Hecke $L$-functions, even in the presence of an exceptional zero.

We interchange the order of summation in $\mathcal{D}_2(\mathfrak{B}_1,\mathfrak{B}_2)$, putting the sums over $\ell_1$ and $\ell_2$ on the outside and the sums over $z_1$ and $z_2$ on the inside. With $\ell_1$ and $\ell_2$ fixed, we then write
\begin{align*}
\mathop{\sum \sum}_{\substack{z_1 \in \mathfrak{B}_1, z_2 \in \mathfrak{B}_2 \\ \Delta_{10} \leq \theta^{-28}}} \Delta_{10} \leq \sum_{\substack{f \leq \theta^{-28} \\ f \mid 10^\infty}} f  \left|\mathop{\sum \sum}_{\substack{z_1 \in \mathfrak{B}_1, z_2 \in \mathfrak{B}_2 \\ \Delta_{10} =f}}\right|.
\end{align*}
We can exchange $10|\Delta|/\varphi(10|\Delta|)$ for $|\Delta|/\varphi(|\Delta|)$ by considering separately those $f$ divisible by 5 and those $f$ not divisible by 5, and pulling out potential factors of $5/\varphi(5)$ (recall that $|\Delta|$ is always divisible by 2). To show \eqref{bound for cal D 0 B1 B2} it therefore suffices to prove
\begin{align}\label{eq: sum of two mobius}
\mathcal{C} := \mathop{\sum \sum}_{\substack{z_i \in \mathfrak{B}_i \\ (|z_i|^2, \Pi)=1 \\ \Delta_{10}=f \\ \ell_1 z_2 \equiv \ell_2 z_2 (f) \\ |\ell_1z_2 - \ell_2 z_1|^2 \leq 2 \Delta(\mathfrak{B}_1,\mathfrak{B}_2)^2 M}} \mu(|z_1|^2)\mu(|z_2|^2) \frac{|\Delta|}{\varphi(|\Delta|)} \ll \theta^{92}(\log MN)^{O(1)}  \lambda^4 N^2
\end{align}
uniformly in $f \leq \theta^{-28}$ with $f \mid 10^\infty$, and $\ell_1,\ell_2 \ll (MN)^{1/2}$ with $(\ell_1\ell_2,10) = 1$. Note that $\mathcal{C}$ depends on $\mathfrak{B}_1,\mathfrak{B}_2, \ell_1$, $\ell_2$, and $f$, but we have suppressed this dependence for notational convenience.

If $n$ is a positive integer, then
\begin{align*}
\frac{n}{\varphi(n)} =\sum_{d \mid n} \frac{\mu^2(d)}{\varphi(d)},
\end{align*}
and therefore
\begin{align}\label{eq: expanded cal C f}
\mathcal{C} = \sum_{d \ll N} \frac{\mu^2(d)}{\varphi(d)}\mathop{\sum \sum}_{\substack{z_i \in \mathfrak{B}_i \\ (|z_i|^2, \Pi)=1 \\ \Delta_{10}=f \\ \ell_1 z_2 \equiv \ell_2 z_2 (f) \\ \Delta \equiv 0 (d) \\ |\ell_1z_2 - \ell_2 z_1|^2 \leq 2 \Delta(\mathfrak{B}_1,\mathfrak{B}_2)^2 M}} \mu(|z_1|^2)\mu(|z_2|^2).
\end{align}
We introduce a parameter $W$, and estimate trivially the contribution from $d > W$ in \eqref{eq: expanded cal C f}. Writing $z_1$ and $z_2$ in rectangular coordinates, we see the contribution from $d > W$ is bounded by
\begin{align*}
E_W := \sum_{W < d \ll N} \frac{\mu^2(d)}{\varphi(d)} \mathop{\sum \sum \sum \sum}_{\substack{r,s,u,v \ll N^{1/2} \\ rv \equiv su (d)}} 1.
\end{align*}
If $d,r,s$, and $u$ are fixed, then $v$ is fixed modulo $d/(d,r)$, which yields
\begin{align*}
E_W &\ll  N^{3/2} \sum_{W < d \ll N} \frac{\mu^2(d)}{\varphi(d)d} \sum_{r \ll N^{1/2}} (r,d)+(\log N) N^{3/2} \\
&\ll N^2 \sum_{d > W} \frac{\mu^2(d) \tau(d)}{\varphi(d) d} + (\log N)^2 N^{3/2} \\
&\ll (\log W) W^{-1} N^2 + (\log N)^2 N^{3/2}.
\end{align*}
Setting 
\begin{align*}
W = \theta^{-92} \lambda^{-4} \asymp \theta^{-704}
\end{align*}
then gives an acceptable contribution for \eqref{eq: sum of two mobius}.

The rational congruence $\Delta \equiv 0 (d)$ is equivalent to the Gaussian congruence $\overline{z_1} z_2 \equiv z_1 \overline{z_2} (2d)$. Since $(z_i \overline{z_i},\Pi) = 1$ and $2d \ll W$, we see that $(z_1\overline{z_1} z_2 \overline{z_2},2d) = 1$. We detect this congruence with multiplicative characters modulo $2d$. Since $(\ell_1\ell_2,10) = 1$, we may also detect the congruence $z_1 \ell_2 \equiv z_2 \ell_1 (f)$ with multiplicative characters.

We handle the condition $\Delta_{10} = f$ as follows. Write $f = 2^a5^b$. Then $\Delta_{10} = f$ if and only if
\begin{align*}
\Delta &\equiv 0 \pmod{2^a}, \ \ \ \ \Delta \not \equiv 0 \pmod{2^{a+1}}, \\
\Delta &\equiv 0 \pmod{5^b}, \ \ \ \ \Delta \not \equiv 0 \pmod{5^{b+1}}.
\end{align*}
These congruences are equivalent to
\begin{align*}
\Delta &\equiv 2^a \pmod{2^{a+1}}, \\
\Delta &\equiv 5^b,2 \cdot 5^b, 3 \cdot 5^b, \text{ or } 4 \cdot 5^b \pmod{5^{b+1}},
\end{align*}
and by the Chinese remainder theorem these are equivalent to
\begin{align*}
\Delta &\equiv \nu_1, \nu_2, \nu_3, \text{ or } \nu_4 \pmod{10f},
\end{align*}
for some residue classes $\nu_i$. We therefore write our sum over $z_1$ and $z_2$ as
\begin{align*}
\mathop{\sum \sum }_{\substack{z_1,z_2 \\ \Delta_{10} = f}} =  \mathop{\sum \sum}_{\substack{\mathfrak{m}, \mathfrak{n} (10f) \\ \text{Im}(\overline{\mathfrak{m}}\mathfrak{n}) \equiv \nu_i (10f)}} \sum_{z_1 \equiv \mathfrak{m} (10f)} \sum_{z_2 \equiv \mathfrak{n} (10f)}.
\end{align*}
Observe that the residue classes $\mathfrak{m}, \mathfrak{n}$ are primitive since $(z_i,\Pi) = 1$. We trivially have
\begin{align*}
\mathop{\sum \sum}_{\substack{\mathfrak{m}, \mathfrak{n} (10f) \\ \text{Im}(\overline{\mathfrak{m}}\mathfrak{n}) \equiv \nu_i (10f)}} 1 \ll f^4,
\end{align*}
so to prove \eqref{eq: sum of two mobius} it suffices to show that
\begin{align}\label{simplif unlabd bound b4 S}
\mathop{\sum \sum}_{\substack{z_i \in \mathfrak{B}_i \\ (|z_i|^2,\Pi)=1 \\ |\ell_1z_2 - \ell_2 z_1|^2 \leq 2 \Delta(\mathfrak{B}_1,\mathfrak{B}_2)^2 M}} \mu(|z_1|^2) \psi'(z_1) \mu(|z_2|^2) \psi(z_2) \ \ll \theta^{204} \lambda^4(\log MN)^{O(1)} N^2
\end{align}
uniformly in characters $\psi'$ and $\psi$. Here
\begin{align*}
\psi(\mathfrak{m}) = \chi(\mathfrak{m}) \overline{\chi}(\overline{\mathfrak{m}}) \zeta(\mathfrak{m}) \phi(\mathfrak{m}),
\end{align*}
where $\chi$ is a character modulo $2d$, $\zeta$ is a character modulo $f$, and $\phi$ is a character modulo $10f$. The character $\psi'$ is given similarly. The bar denotes complex conjugation and not multiplicative inversion. Observe that $\psi,\psi'$ are characters with moduli at most $O(d^2 f^2) = O(\theta^{-1464})$. Taking the supremum over $z_1$, it suffices to show that
\begin{align}\label{bound for cal S}
\mathcal{S} &:= \sum_{\substack{z_2 \in \mathfrak{B}_2 \\ (z_2 \overline{z_2},\Pi)=1 \\ \eqref{disc int pol box}}} \mu(|z_2|^2) \psi(z_2) \ll \theta^{204} \lambda^2 (\log MN)^{O(1)} N,
\end{align}
uniformly in $\psi, z_1, \ell_1$, and $\ell_2$. The last condition in the summation conditions for $\mathcal{S}$ is
\begin{align}\label{disc int pol box}
\left|z_2 - z_1 \frac{\ell_2}{\ell_1} \right| \leq \sqrt{2} \frac{|\Delta(\mathfrak{B}_1,\mathfrak{B}_2)| M^{1/2}}{\ell_1}.
\end{align}
We see that \eqref{disc int pol box} forces $z_2$ to lie in some disc in the Gaussian integers. Since $z_2$ already lies in a polar box, we need to understand the intersection of a polar box with a disc. 

We introduce a parameter $\varpi$. We cover $\mathfrak{B}_2$ in polar boxes, which we call $\varpi$-polar boxes, of the form 
\begin{align*}
&R \leq |z_2|^2 \leq (1+\varpi)R, \\
&\vartheta \leq \text{arg}(z_2) \leq \vartheta + 2\pi \varpi.
\end{align*}
For technical convenience we use smooth partitions of unity to accomplish this. This amounts to attaching smooth functions $g(|z_2|^2)$ and $q(\text{arg} (z_2))$, where $g(n)$ is a smooth function supported on an interval
\begin{align}\label{g interval}
R < n \leq (1+O(\varpi))R, \ \ \ \ R \asymp N,
\end{align}
and which satisfies
\begin{align}\label{g deriv bounds}
g^{(j)}(n) \ll_j (\varpi N)^{-j}, \ \ \ \ j \geq 0.
\end{align}
Further, $q$ is a smooth, $2\pi$-periodic function supported on an interval of length $O(\varpi)$ which satisfies
\begin{align}\label{q deriv bounds}
q^{(j)}(\alpha) \ll_j \varpi^{-j}, \ \ \ \ j \geq 0.
\end{align}

We observe that the boundary of the intersection between $\mathfrak{B}_2$ and the disc \eqref{disc int pol box} is a finite union of circular arcs and line segments. It is straightforward to check that the boundary has length $\ll \lambda N^{1/2}$. Any $\varpi$-polar box that intersects the boundary is contained in a $O(\varpi N^{1/2})$-neighborhood of the boundary. We deduce that the total contribution from those boxes not strictly contained in the intersection is
\begin{align*}
\ll \varpi \lambda N,
\end{align*}
and this is acceptable if we set 
\begin{align*}
\varpi = \theta^{204} \lambda \asymp \theta^{357}.
\end{align*}
It follows that
\begin{align*}
\mathcal{S} &= O(\theta^{204} \lambda^2 N) + \mathop{\sum \sum}_{(g,q) \in S(z_1,\ell_1,\ell_2)} \sum_{(z_2 \overline{z_2},\Pi)=1} \mu(|z_2|^2) \psi(z_2) g(|z_2|^2) q(\text{arg}(z_2)).
\end{align*}
The number of pairs $(g,q) \in S(z_1,\ell_1,\ell_2)$ is $\ll (\log N)^2 \varpi^{-2} \lambda^2$, so to prove \eqref{bound for cal S} it suffices to show that
\begin{align}\label{bound for cal S g q}
\mathcal{S}_{g,q} &:= \sum_{(z_2 \overline{z_2},\Pi)=1} \mu(|z_2|^2) \psi(z_2) g(|z_2|^2) q(\text{arg}(z_2)) \ll \theta^{204} \varpi^{2}(\log N)^{O(1)}  N
\end{align}
uniformly in $g$ and $q$.

Our sum $\mathcal{S}_{g,q}$ is very similar to the sum $S_\chi^k (\beta)$ treated by Friedlander and Iwaniec (see \cite[(16.14)]{friediw2}). Our treatment of $\mathcal{S}_{g,q}$ follows their treatment of $S_\chi^k(\beta)$ quite closely, and we quote the relevant statements and results of \cite[section 16]{friediw2} as necessary. Friedlander and Iwaniec work with characters having moduli divisible by 4, but this is a distinction without material consequence.

We expand $q(\alpha)$ in its Fourier series. From the derivative bounds \eqref{q deriv bounds} we see the Fourier coefficients satisfy
\begin{align}\label{fourier coeff bound}
\widehat{q}(h) \ll \frac{\varpi}{1 + \varpi^2 h^2}.
\end{align}
By means of \eqref{fourier coeff bound} we obtain the truncated Fourier series
\begin{align}\label{trunc four series}
q(\alpha) &= \sum_{|h| \leq H} \widehat{q}(h) e^{i h \alpha} + O (\varpi^{-1} H^{-1}).
\end{align}
The contribution of the error term in \eqref{trunc four series} to $\mathcal{S}_{g,q}$ is $O(N H^{-1})$.

We next use Mellin inversion to write $g(n)$ as
\begin{align}\label{g as integral}
g(n) &= \frac{1}{2\pi i} \int_{(\sigma)} \widehat{g}(s) n^{-s} ds,  \ \ \ \ \ \ \ s = \sigma + it.
\end{align}
As $g$ is supported in the interval \eqref{g interval} and satisfies \eqref{g deriv bounds}, we find that the Mellin transform $\widehat{g}(s)$ is entire and satisfies
\begin{align}\label{mellin bound}
\widehat{g}(s) &\ll \frac{\varpi N^\sigma}{1 + \varpi^2 t^2}.
\end{align}
Applying \eqref{trunc four series} and \eqref{g as integral} we obtain
\begin{align}\label{exp of cal S g q}
\mathcal{S}_{g,q} &= \sum_{|h| \leq H} \widehat{q}(h) \frac{1}{2\pi i}\int_{(\sigma)} \widehat{g}(s) Z_\psi^h (s) ds + O(N H^{-1}),
\end{align}
where
\begin{align*}
Z_\psi^h(s) &:= \sum_{(z \overline{z},\Pi)=1} \mu(Nz) \psi(z) \left( \frac{z}{|z|}\right)^h (Nz)^{-s}
\end{align*}
and $Nz$ denotes the norm $|z|^2$ of $z$. Call an ideal \emph{odd} if it contains no primes over 2 in its factorization into prime ideals. Since $z$ is odd and primary, there is a one-to-one correspondence between elements $z$ and odd ideals $\mathfrak{a}$, given by $\mathfrak{a} = (z)$. Omitting subscripts and superscripts for simplicity, we then have
\begin{align*}
Z(s) &= \sum_{(\mathfrak{a} \overline{\mathfrak{a}},\Pi)=1} \xi(\mathfrak{a}) \mu(N \mathfrak{a}) (N\mathfrak{a})^{-s},
\end{align*}
where
\begin{align*}
\xi(\mathfrak{a}) := \psi(z) \left( \frac{z}{|z|}\right)^h
\end{align*}
and $z$ is the unique primary generator of $\mathfrak{a}$. From the Euler product it follows that
\begin{align*}
Z(s) &= L(s,\xi)^{-1} P(s) G(s),
\end{align*}
where $L(s,\xi)$ is the Hecke $L$-function
\begin{align*}
L(s,\xi) &:= \sum_{\mathfrak{a}} \frac{\xi(\mathfrak{a})}{(N\mathfrak{a})^s},
\end{align*}
$P(s)$ is the Dirichlet polynomial given by
\begin{align*}
P(s) &:= \prod_{\substack{p \leq P \\ p \equiv 1 (4)}} \left(1 - \frac{\xi(\mathfrak{p})}{p^s} \right)^{-1}\left(1 - \frac{\xi(\overline{\mathfrak{p}})}{p^s} \right)^{-1}
\end{align*}
where $\mathfrak{p} \overline{\mathfrak{p}} = (p)$,
and $G(s)$ is given by an Euler product that converges absolutely and uniformly in $\sigma \geq \frac{1}{2} + \epsilon$. In the region $\sigma \geq 1-\frac{1}{\log P}$ the inequality $p^{-\sigma} < 3p^{-1}$ holds, and this gives the bound
\begin{align}\label{bound for prod P}
P(s) \ll (\log P)^3, \ \ \ \ \ \ \sigma \geq 1 - \frac{1}{\log P}.
\end{align}

Let $k$ be the modulus of $\xi$ (recall that $k \ll \theta^{-1464}$). Then $L(s,\xi)$ is nonzero (see \cite[(16.20)]{friediw2}) in the region
\begin{align*}
\sigma \geq 1- \frac{c}{\log(k + |h| + |t|)},
\end{align*}
except for possibly an exceptional real zero when $\xi$ is real. By applying the method of Siegel (\cite[Lemma 16.1]{friediw2}) one may show that when $\xi$ is real, $L(s,\xi)$ has no zeros in the region
\begin{align}\label{siegel zero free}
\sigma \geq 1 - \frac{c(\varepsilon)}{k^{\varepsilon}}, \ \ \ \ 0 < \varepsilon \leq \frac{1}{4}.
\end{align}
The constant $c(\varepsilon)$ is ineffective, and for this reason the implied constants in Proposition \ref{prop for BMN} and Theorem \ref{main theorem} are ineffective.

The inequality \eqref{siegel zero free} allows one to establish (\cite[(16.23) and (16.24)]{friediw2}) the upper bound
\begin{align*}
L(s,\xi)^{-1} \ll k (\log (|h| + |t| + 3))^2
\end{align*}
in the region
\begin{align*}
\sigma \geq 1 - \frac{c(\varepsilon)}{k^\varepsilon \log(|h| + |t|+3)}.
\end{align*}
For $T \geq |h| + 3$, we set
\begin{align*}
\beta := \text{min} \left( \frac{c(\varepsilon)}{k^\varepsilon \log T}, \frac{1}{\log P} \right),
\end{align*}
so that in the region $\sigma \geq 1 - \beta$ we have the bound
\begin{align}\label{bound on Z}
Z(s) &\ll k (\log (|h| + |t| + N))^5.
\end{align}

We now estimate the integral
\begin{align}\label{defn of int I}
I := \frac{1}{2\pi i}\int_{(\sigma)} \widehat{g}(s) Z(s) ds.
\end{align}
We move the contour of integration to
\begin{align*}
s &= 1+it, \ \ \ \ \ |t| \geq T, \\
s &= 1 - \beta + it, \ \ \ \ |t| \leq T,
\end{align*}
and add in horizontal connecting segments
\begin{align*}
s = \sigma + \pm iT, \ \ \ \ \ 1 - \beta \leq \sigma \leq 1.
\end{align*}
Estimating trivially we find by \eqref{mellin bound} and \eqref{bound on Z} that
\begin{align*}
I &\ll \varpi^{-1} \theta^{-1464} \left(T^{-1} + N^{-\beta} \right) N (\log (N+T))^5.
\end{align*}
We set $T := 3 \exp(\sqrt{\log N})$. Recalling that $\log P \leq \frac{\sqrt{\log x}}{\log \log x} \ll \frac{\sqrt{\log N}}{\log \log N}$, we see that
\begin{align}\label{I bound}
I &\ll (\log N)^5 \varpi^{-1} \theta^{-1464} N \left(\exp \left(- \frac{c(\varepsilon) \sqrt{\log N}}{k^\varepsilon} \right) + \exp \left(-c\sqrt{\log N} \right)\right)
\end{align}
uniformly in $|h| \leq 2\exp(\sqrt{\log N})$. We choose $H := \exp(\sqrt{\log N})$, then take \eqref{exp of cal S g q} together with \eqref{I bound} and sum over $|h| \leq H$ by means of \eqref{fourier coeff bound}. Provided $\varepsilon > 0$ is sufficiently small in terms of $\theta$ (take $\varepsilon = \varepsilon (L)$, compare \eqref{defn of theta}), we obtain the bound
\begin{align}\label{final bound}
\mathcal{S}_{g,q} &\ll_\epsilon N \exp \left(- c(\epsilon) (\log N)^{\frac{1}{2} - \epsilon} \right).
\end{align}
The bound \eqref{final bound} implies \eqref{bound for cal S g q}, and this completes the proof of Proposition \ref{prop for BMN}.

\section{Modifications for Theorem \ref{second theorem}}\label{sec: mods for other theorem}

The proof of Theorem \ref{second theorem} follows the same lines as the proof of Theorem \ref{main theorem}. We provide a sketch of the modified argument, and leave the task of fleshing out complete details to the interested reader.

We let $d \in \{2,3\}$, and let $\{a_1,\ldots,a_d\} \subset \{0,1,2,\ldots,9\}$ be a fixed set. Denote by $\mathcal{A}_d$ the set of nonnegative integers missing the digits $a_1,\ldots,a_d$ in their decimal expansions. Let $\gamma_d := \frac{\log(10-d)}{\log 10}$. For $Y$ a power of 10 we define
\begin{align*}
F_{Y,d} (\theta) := Y^{-\gamma_d} \left|\sum_{n < Y} \mathbf{1}_{\mathcal{B}_d}(n) e(n\theta) \right|.
\end{align*}
We note that if $Y = 10^k$ then
\begin{align*}
F_{Y,d}(\theta) &= \prod_{i=0}^{k-1}\frac{1}{10-d} \left|\sum_{n_i < 10} \mathbf{1}_{\mathcal{B}_d}(n_i) e(n_i10^i\theta) \right| \\
&= \prod_{i=1}^k \frac{1}{10-d} \left|\frac{e(10^i\theta - 1)}{e(10^{i-1}\theta)-1} - \sum_{r=1}^d e(a_r 10^{i-1}\theta) \right|.
\end{align*}
We therefore have the product formula
\begin{align*}
F_{UV}(\theta) = F_U(\theta) F_V (U\theta).
\end{align*}

The most important task is to obtain analogues of Lemmas \ref{maynard small moduli}, \ref{maynard single denominator}, and \ref{maynard type I} for the functions $F_{Y,d}$. By arguing as in the proof of \cite[Lemma 4.1]{maynard} it is not difficult to prove the analogue of Lemma \ref{maynard small moduli}.

\begin{lemma}\label{lem: more digits small moduli}
Let $q < Y^{\frac{1}{3}}$ be of the form $q = q_1q_2$ with $(q_1,10)=1$ and $q_1 > 1$. Then for any integer $a$ coprime with $q$ we have
\begin{align*}
F_{Y,d}\left(\frac{a}{q} \right) \ll \exp \left(-c \frac{\log Y}{\log q} \right)
\end{align*}
for some absolute constant $c > 0$.
\end{lemma}

It is a little more difficult to obtain the analogues of Lemmas \ref{maynard single denominator} and \ref{maynard type I}. They will follow from a good upper bound for
\begin{align*}
\sup_{\beta \in \mathbb{R}} \sum_{a < Y}F_{Y,d} \left(\frac{a}{Y} + \beta \right).
\end{align*}
The key is that we can estimate moments of $F_{Y,d}$ by numerically computing the largest eigenvalue of a certain matrix.

\begin{lemma}\label{lem: matrix moment}
Let $J$ be a positive integer. Let $\lambda_{t,J,d}$ be the largest eigenvalue of the $10^J \times 10^J$ matrix $M_{t,d}$, given by
\begin{align*}
(M_{t,d})_{i,j} := 
\begin{cases}
G_d(a_1,\ldots,a_{J+1})^t, &\text{if } i-1 = \sum_{\ell = 1}^J a_{\ell+1}10^{\ell-1}, j-1 = \sum_{\ell=1}^J a_\ell 10^{\ell-1} \\
&\text{ for some } a_1,\ldots,a_{J+1} \in \{0,\ldots,9\}, \\
0, &\text{otherwise},
\end{cases}
\end{align*}
where
\begin{align*}
G_d(t_0,\ldots,t_J) &:= \sup_{|\gamma| \leq 10^{-J-1}}\frac{1}{10-d} \left|\frac{e \left(\sum_{j=0}^Jt_j 10^{-j}+10\gamma \right)-1}{e \left(\sum_{j=0}^Jt_j 10^{-j-1}+\gamma \right)-1} - \sum_{r=1}^d e \left(\sum_{j=0}^J a_rt_j10^{-j-1} + a_r\gamma \right) \right|.
\end{align*}
Then
\begin{align*}
\sum_{0 \leq a < 10^k} F_{10^k,d}\left(\frac{a}{10^k} \right)^t \ll_{t,J,d} \ \lambda_{t,J,d}^k.
\end{align*}
\end{lemma}
\begin{proof}
Following the proof of \cite[Lemma 4.2]{maynard}, we find that
\begin{align*}
F_{Y,d}\left(\sum_{i=1}^k \frac{t_i}{10^i} \right) \leq \prod_{i=1}^k G_d(t_i,\ldots,t_{i+J}),
\end{align*}
where $t_j = 0$ for $j > k$. Maynard proceeds at this point using a Markov chain argument, but we give here a different argument due to Kevin Ford (private communication).

Write $M_{t,d} = (m_{ij})_{i,j}$ (we suppress the dependence on $d$ for notational convenience), where $m_{ij}$ is zero unless
\begin{align*}
i - 1 &= a_2+a_310 + \cdots + 10^{J-1}a_{J+1}, \\
j-1 &= a_1 + 10a_2 + \cdots + 10^{J-1} a_J
\end{align*}
for some digits $a_1,\ldots,a_J$. Thus
\begin{align*}
(M_{t,d}^k)_{i,j} = \sum_{i_1,\ldots,i_{k-1}} m_{i,i_1} m_{i_1,i_2}\cdots m_{i_{k-1},j},
\end{align*}
where the product is nonzero only if
\begin{align*}
j-1 &= a_1 + 10a_2 + \cdots + 10^{J-1} a_J, \\
i_{k-1}-1 &= a_2 + 10a_3 + \cdots + 10^{J-1} a_{J+1}, \\
&\vdots \\
i_1-1 &= a_k + 10a_{k+1} + \cdots + 10^{J-1} a_{k+J-1}, \\
i-1 &= a_{k+1} + 10a_{k+2} + \cdots + 10^{J-1} a_{k+J}.
\end{align*}
Fixing $i = 1$, so that $a_{k+1} = \cdots = a_{J+k} = 0$, and summing over $j$ we obtain
\begin{align*}
\sum_j (M_{t,d}^k)_{1,j} = \sum_{a_1,\ldots,a_k} G_d(a_1,\ldots,a_{J+1})^t \cdots G(a_k,\ldots,a_{k+J})^t \geq \sum_{0 \leq a < 10^k} F_{10^k,d}\left(\frac{a}{10^k} \right)^t.
\end{align*}
One may then use the Perron-Frobenius theorem to obtain the conclusion of the lemma.
\end{proof}

The following is a consequence of Lemma \ref{lem: matrix moment} and some numerical computation.
\begin{lemma}\label{lem: more digits L1 bound}
We have
\begin{align*}
\sup_{\beta \in \mathbb{R}} \ \sum_{a < Y} F_{Y,d} \left(\beta + \frac{a}{Y} \right) \ll Y^{\alpha_d}
\end{align*}
and
\begin{align*}
\int_0^1 F_{Y,d}(t)dt \ll Y^{-1 + \alpha_d},
\end{align*}
where
\begin{align*}
\alpha_2 &= \frac{54}{125}, \ \ \ \ \ \ \ \alpha_3 = \frac{99}{200}.
\end{align*}
\end{lemma}
\begin{proof}
We use bounds on $\lambda_{1,2,2}$ and $\lambda_{1,2,3}$. By numerical calculation\footnote{Mathematica\textregistered \ files with these computations are included with this work on \text{arxiv.org}.} we find
\begin{align*}
\lambda_{1,2,2} < 10^{\frac{54}{125}}
\end{align*}
for all choices of $\{a_1,a_2\} \subset \{0,\ldots,9\}$, and
\begin{align*}
\lambda_{1,2,3} < 10^{\frac{99}{200}}
\end{align*}
for all choices of $\{a_1,a_2,a_3\} \subset \{0,\ldots,9\}$. By the argument of \cite[Lemma 4.3]{maynard} this then yields
\begin{align*}
\sup_{\beta \in \mathbb{R}} \ \sum_{a < Y} F_{Y,d} \left(\beta + \frac{a}{Y} \right) \ll Y^{\alpha_d}.
\end{align*}
To complete the proof we observe
\begin{align*}
\int_0^1 F_{Y,d}(t) dt &= \sum_{0 \leq a < Y} \int_{\frac{a}{Y}}^{\frac{a}{Y} + \frac{1}{Y}} F_{Y,d}(t)dt = \int_0^{\frac{1}{Y}} \sum_{0 \leq a < Y}F_{Y,d}\left(\frac{a}{Y} + t \right) dt \\ 
&\leq Y^{-1} \sup_{\beta \in \mathbb{R}} \ \sum_{a < Y} F_{Y,d} \left(\frac{a}{Y} + \beta \right) \ll Y^{-1+\alpha_d}.
\end{align*}
\end{proof}

We note it is crucial for the proof of Theorem \ref{second theorem} that $\alpha_d < \frac{1}{2}$. For $d \geq 4$ there exist choices of excluded digits which force $\alpha_d > \frac{1}{2}$.

\begin{lemma}\label{lem: more digits large sieve}
We have
\begin{align*}
\sup_{\beta \in \mathbb{R}} \ \sum_{a \leq q} \left|F_{Y,d} \left(\frac{a}{q} + \beta \right) \right| &\ll q^{\alpha_d} + \frac{q}{Y^{1-\alpha_d}}, \\
\sup_{\beta \in \mathbb{R}} \ \sum_{q\leq Q} \sum_{\substack{1 \leq a \leq q \\ (a,q)=1}} \left|F_{Y,d} \left(\frac{a}{q} + \beta \right) \right| &\ll Q^{2\alpha_d} + \frac{Q^2}{Y^{1-\alpha_d}}
\end{align*}
\end{lemma}
\begin{proof}
We use the large sieve argument of \cite[Lemma 4.5]{maynard} with Lemma \ref{lem: more digits L1 bound}.
\end{proof}

Let us now give a broad sketch of the proof of Theorem \ref{second theorem}. We proceed as in the proof of Theorem \ref{main theorem}, only we use Lemma \ref{lem: more digits large sieve} instead of Lemma \ref{maynard single denominator} or Lemma \ref{maynard type I}.

Our sequence 
\begin{align*}
\mathop{\sum \sum}_{\substack{m^2+\ell^2 = n \\ (\ell,\Pi) = 1}} \mathbf{1}_{\mathcal{A}_d}(\ell)
\end{align*}
has level of distribution
\begin{align*}
D \leq x^{\gamma_d - \epsilon},
\end{align*}
and we have an acceptable Type II bound provided
\begin{align}\label{eq: Type II range more digits}
x^{\frac{1}{2} - \frac{\gamma_d}{2} + \epsilon} \ll N \ll x^{\frac{1}{2}(1-\alpha_d) - \epsilon}.
\end{align}
(Compare \eqref{eq: Type II range more digits} with \eqref{restrictions on N}.) Since
\begin{align*}
\frac{1}{2}(1-\alpha_d) - \left(\frac{1}{2} - \frac{\gamma_d}{2} \right) > 1 - \gamma_d
\end{align*}
there exists an appropriate choice of $U$ and $V$ in Vaughan's identity \eqref{vaughans identity} (compare with \eqref{admissible choices for U,V}).

At various points in the proof of Theorem \ref{main theorem} we had to perform a short interval decomposition in order to gain control on elements of $\mathcal{A}$ in arithmetic progressions (see the arguments in section \ref{sec: sieve main, using fund} leading up to Lemmas \ref{prototypical estimation lemma one} and \ref{prototypical estimation lemma two}). The short interval decomposition depended on whether or not the missing digit was the zero digit. In the case of Theorem \ref{second theorem} one argues similarly, and finds that the short interval decomposition depends only on whether $0 \in \mathcal{A}_d$.

\section*{Acknowledgements}
The author wishes to thank his adviser Kevin Ford for suggesting the problem which resulted in this paper, and for many helpful conversations. The author expresses gratitude to Junxian Li for some assistance with Mathematica\textregistered. Part of this research was conducted while the author was supported by NSF grant DMS-1501982.

\end{document}